\documentclass[fleqn,a4paper,9pt]{article}
\usepackage{amsmath}
\usepackage{mathtools}
\usepackage{mathptmx}
\usepackage{comment}
\usepackage{tabularx}
\usepackage{listings}
 \usepackage{amsthm} 
\usepackage{amssymb}

\makeatletter
\let\@@pmod\pmod
\DeclareRobustCommand{\pmod}{\@ifstar\@pmods\@@pmod}
\def\@pmods#1{\mkern4mu({\operator@font mod}\mkern 6mu#1)}
\makeatother
\DeclarePairedDelimiter\floor{\lfloor}{\rfloor}
\newcommand{\tpmod}[1]{\mkern 8mu({\operator@font mod}\mkern 6mu#1)}
\usepackage{atbegshi}
\AtBeginDocument{\AtBeginShipoutNext{\AtBeginShipoutDiscard}}

\author{Leon Fairbanks}
\begin{document}
\begin{flushleft}
\title{{\textbf{Notes On An Approach To Apery's Constant}}}
\maketitle
\begin{abstract}
The Basel problem, solved by Leonhard Euler in 1734, asks to resolve $\zeta(2)$, the sum of the reciprocals of the squares of the natural numbers, i.e. the sum of the infinite series:
\begin{equation}
 \sum _{i=1}^{\infty}\frac{1}{n^2}=\frac{1}{1^2}+\frac{1}{2^2}+\frac{1}{3^2}+\ldots\notag
\end{equation}
 The same question is posed regarding the summation of the reciprocals of the cubes of the natural numbers, $\zeta(3)$. The resulting constant is known as Apery's constant.\\
A YouTube channel, 3BlueBrown, produced a video entitled, ``Why is pi here? And why is it squared? A geometric answer to the Basel problem". The video presents the work of John W{\"a}stlund. The equations can be extended to $\zeta(n)$, but the geometric argument is lost. We try to explore these equations for $\zeta(n)$.
\end{abstract}
\tableofcontents

\newtheorem{theorem}{Theorem}[section]
\newtheorem{lemma}[theorem]{Lemma}
\newtheorem{proposition}[theorem]{Proposition}
\newtheorem{corollary}[theorem]{Corollary}
\newtheorem{definition}[theorem]{Definition}

\section{Introduction}
The video focuses on summation over odd integer values for $\zeta(2)$.
\begin{lemma}
Assume $m\in\mathbb{Z}$, $m>0$ then
$$\sum_{i=1}^{\infty}\frac{1}{i^m}= \left(\frac{2^m}{2^{m}-1}\right)( \sum_{i=1}^{\infty} \frac{1}{(2i-1)^m})$$\notag
\end{lemma}
\begin{proof}
\begin{align}\notag
 \sum_{i=1}^{\infty} \frac{1}{i^m}&=( \sum_{i=1}^{\infty} \frac{1}{(2i-1)^m})( \sum_{i=0}^{\infty} \frac{1}{(2^i)^m})\\\notag
 &=( \sum_{i=1}^{\infty} \frac{1}{(2i-1)^m})(\frac{1}{1-\frac{1}{2^m}})\\\notag
 &=\frac{2^m}{2^m-1}( \sum_{i=1}^{\infty} \frac{1}{(2i-1)^m})\\\notag
 \end{align}
\end{proof}
\begin{lemma}
$$x-\frac{x^3}{6}<\sin(x)<x-\frac{2x^3}{3\pi^2}\ \ \rm{;}\ x\in(0,\pi/2)$$\notag
\end{lemma}
\begin{proof}
Reference R. Kl\'en et al [17]
\end{proof}
\begin{lemma}
Assume $m,n\ge 1$ then
$$\sum_{i=1}^{2^n-1}\frac{1}{i^m}<=n$$
\end{lemma}
\begin{proof}
Break the sum up into $n$ sums, the first containing the first term, 1, the second containing the next $2^1$ terms, $\frac{1}{2^m}$,$\frac{1}{3^m}$, the third containing the next $2^2$ terms,  $\frac{1}{4^m}$,$\frac{1}{5^m}$, $\frac{1}{6^m}$,$\frac{1}{7^m}$, up to the n-th sum containing the last $2^{n-1}$ terms. Each term in the j-th sum is bounded by $\frac{1}{2^{j-1}}$. Hence
$$\sum_{i=1}^{2^n-1}\frac{1}{i^m}<=\sum_{i=1}^{n}2^{j-1}\frac{1}{2^{j-1}}=n$$\notag
\end{proof}
\begin{proposition}
Sum over odd values:
$$\zeta(2)=\lim_{n\to\infty}\left(\frac{4\pi^2}{3}\right) \sum _{i=1}^{2^{n-2}}\frac{1}{ \left(2^n\sin\left(\frac{\left(2 i-1\right)\pi}{2^n}\right)\right)^2} $$\notag\\
\end{proposition}
\begin{proof}
Reference [1]
\end{proof}
\begin{proposition}
Sum over odd values:\\
Assume $m\in\mathbb{Z}$, $m>2$ then
$$\zeta(m)=\lim_{n\to\infty}\left(\frac{2^m\pi^m}{2^{m}-1}\right) \sum _{i=1}^{2^{n-2}}\frac{1}{ \left(2^n\sin\left(\frac{\left(2 i-1\right)\pi}{2^n}\right)\right)^m} $$\notag\\
\end{proposition}
\begin{proof}
Let $$A_j=\sum_{i=1}^{2^{j-2}}\frac{1}{(2i-1)^m}$$\notag
$$B_j= \sum _{i=1}^{2^{j-2}}\pi^m\frac{1}{ \left(2^{ j}\sin\left(\frac{\left(2 i-1\right)\pi}{2^j}\right)\right)^m}$$\notag
We want to show $$\lim_{j\to\infty}A_j=\lim_{j\to\infty}B_j$$\notag
We show $$\lim_{j\to\infty}(B_j-A_j)=0$$\notag
\begin{align}
B_j-A_j&=\sum_{i=1}^{2^{j-2}}\frac{1}{ \left(\frac{2^j}{\pi}\sin\left(\frac{\left(2 i-1\right)\pi}{2^j}\right)\right)^m}-\frac{1}{(2i-1)^m}\\\notag
&=\sum_{i=1}^{2^{j-2}}\frac{(2i-1)^m- \left(\frac{2^j}{\pi}\sin\left(\frac{\left(2 i-1\right)\pi}{2^j}\right)\right)^m}{ (2i-1)^m\left(\frac{2^j}{\pi}\sin\left(\frac{\left(2 i-1\right)\pi}{2^j}\right)\right)^m}\\\notag
&<\sum_{i=1}^{2^{j-2}}\frac{ (2i-1)^m- \left(  \frac{2^j}{\pi}\left(  \left(\frac{\left(2 i-1\right)\pi}{2^j}\right)-\left(\frac{\left(2 i-1\right)\pi}{2^j}\right)  \right)^3/6  \right)^m }{ (2i-1)^m\left(  \frac{2^j}{\pi}\left(  \left(\frac{\left(2 i-1\right)\pi}{2^j}\right)-\left(\frac{\left(2 i-1\right)\pi}{2^j}\right)  \right)^3/6  \right)^m}\\\notag
&=\sum_{i=1}^{2^{j-2}}\frac{ (2i-1)^m- \left(          \left(2 i-1\right)\left(    1-  \left(\frac{\left(2 i-1\right)\pi}{2^j} \right)^2/6   \right)           \right)^m  }{ (2i-1)^m\left(          \left(2 i-1\right)\left(    1-  \left(\frac{\left(2 i-1\right)\pi}{2^j} \right)^2/6   \right)           \right)^m }\\\notag
&=\sum_{i=1}^{2^{j-2}}\frac{ 1-        \left(    1-  \left(\frac{\left(2 i-1\right)\pi}{2^j} \right)^2/6   \right)^m  }{ (2i-1)^m\left(     \left(    1-  \left(\frac{\left(2 i-1\right)\pi}{2^j} \right)^2/6   \right)           \right)^m }\\\notag
&=\sum_{i=1}^{2^{j-2}}\frac{    \left( \frac{\left(2 i-1\right)\pi}{2^j} \right)^2/6\left ( 1+      \left(    1-  \left(\frac{\left(2 i-1\right)\pi}{2^j} \right)^2/6   \right)+\cdots+ \left(    1-  \left(\frac{\left(2 i-1\right)\pi}{2^j} \right)^2/6   \right)^{m-1}        \right)      }                 { (2i-1)^m\left(     \left(    1-  \left(\frac{\left(2 i-1\right)\pi}{2^j} \right)^2/6   \right)           \right)^m }\\\notag
&<\sum_{i=1}^{2^{j-2}}\frac{m}{2^{2j-2} (2i-1)^{m-2}\left(     \left(    1-  \left(\frac{\left(2 i-1\right)\pi}{2^j} \right)^2/6   \right)           \right)^m }\\\notag
&<\frac{m}{2^{2j-2-3m}}\sum_{i=1}^{2^{j-2}}\frac{1} { (2i-1)^{m-2}}\\\notag
&<\frac{(j-1)m}{2^{2j-2-3m}}\\\notag
\implies&\lim_{j\to\infty}B_j-A_j=0\\\notag
\end{align}
\end{proof}

\begin{proposition}
Sum over all values:
$$\zeta(m)=\lim_{n\to\infty}\pi^{m}\sum _{i=1}^{2^{n-2}}\frac{1}{ \left(2^n\sin\left(\frac{i\pi}{2^n}\right)\right)^m} $$\notag\\
\end{proposition}
\begin{proof}
As in previous proposition.
\end{proof}
\begin{definition}
Note that, as with infinite series, these are partial sums whose number of terms increases to infinity. However, unlike seris, the individual terms of the sum change as it evolves. We might call these ``dynamic sums''.
\end{definition}
The video presents a geomertic proof. However, we note that since $\lim_{x\to 0}\frac{\sin(x)}{x}=1$, it is perhaps not surprising that $$\lim_{n\to\infty} \sum_{i=1}^{2^{n-2}} \frac{\pi^m}{ \left(2^n\sin\left(\frac{(2i-1)\pi}{2^n}\right)\right)^m}=\lim_{n\to\infty} \sum_{i=1}^{2^{n-2}} \frac{1}{(2i-1)^m}$$ Indeed, instead of summing to $2^{n-2}$, we could increase the number of terms in the sum more slowly and ensure that, in the limit, the number of terms goes to infinity while each term $\sin\left(\frac{(2i-1)\pi}{2^n}\right)$ converges to $0$. And

 \section{Zeta At Even Integer Values}
\begin{lemma}
$$\sum_{i=1}^{2^{n-2}}\left(\frac{1}{  \sin\left(\frac{\left(2 i-1\right)\pi}{2^n}\right)   }\right)^2=2^{2n-3}$$
\end{lemma}
\begin{proof}
\begin{align}\notag
\sum_{i=1}^{2^{n-2}}\left(    \frac{1}{  \sin\left(\frac{\left(2 i-1\right)\pi}{2^n}\right)   }   \right)^2&=\sum_{i=1}^{2^{n-2}} \frac{2}{ 1- \cos\left(\frac{\left(2 i-1\right)\pi}{2^{n-1}}\right) }   \\\notag
&=\sum_{i=1}^{2^{n-3}}\frac{2}{ 1- \cos\left(\frac{\left(2 i-1\right)\pi}{2^{n-1}}\right)   }+\frac{2}{ 1+\cos\left(\frac{\left(2 i-1\right)\pi} {2^{n-1}}\right)}\\\notag
&=4\sum_{i=1}^{2^{n-3}}\frac{1}{ 1- \cos^2\left(\frac{\left(2 i-1\right)\pi}{2^{n-1}}\right)   }\\\notag
&=4\sum_{i=1}^{2^{n-3}}\frac{1}{ \sin^2\left(\frac{\left(2 i-1\right)\pi}{2^{n-1}}\right)   }\\\notag
\end{align}
Since for $n=2$ the sum equals $2$, by induction, at $n>2$ the sum equals $2*2^{2n-4}$
\end{proof}
\begin{proposition}
$$\zeta(2)=\frac{\pi^2}{6}$$
\end{proposition}
\begin{proof}
\begin{align}\notag
\zeta(2)&=\lim_{n\to\infty} \left(\frac{4\pi^2}{3}\right) \sum _{i=1}^{2^{n-2}}\frac{1}{  \left(2^n\sin\left(\frac{\left(2 i-1\right)\pi}{2^n}\right)\right)^2  }\\\notag
&=\lim_{n\to\infty}\left(\frac{4\pi^2}{3}\right) \frac{1}{ 2^{2n}}2^{2n-3}\\\notag
&=\frac{\pi^2}{6} \\\notag
\end{align}
\end{proof}
\begin{lemma}
$$\sum_{i=1}^{2^{n-2}}\left(\frac{1}{   \sin\left(\frac{\left(2 i-1\right)\pi}{2^n}\right)   }\right)^4=2^{2n-2}(2^{2n-3}-1)/3$$
\end{lemma}
\begin{proposition}
$$\zeta(4)=\frac{\pi^4}{90} $$
\end{proposition}
\begin{proof}
\begin{align}\notag
\zeta(4)&=\lim_{n\to\infty}\left(\frac{16\pi^4}{15}\right) \sum _{i=1}^{2^{n-2}}\frac{1}{ \left(2^n\sin\left(\frac{\left(2 i-1\right)\pi}{2^n}\right)\right)^4}\\\notag
&=\lim_{n\to\infty}\left(\frac{16\pi^4}{15}\right) \frac{1}{ 2^{4n}}2^{2n-2}(2^{2n-3}-1)/3\\\notag
&=\frac{\pi^4}{90} \\\notag
\end{align}
\end{proof}
\begin{lemma}
$$\sum_{i=1}^{2^{n-2}}\left(\frac{1}{   \sin\left(\frac{\left(2 i-1\right)\pi}{2^n}\right)   }\right)^6=\frac{1}{960}( 2^{2n+6}+5*2^{4n+1}+2^{6n})$$
\end{lemma}
\begin{proposition}
$$\zeta(6)=\frac{\pi^6}{945} $$
\end{proposition}
\begin{proof}
\begin{align}\notag
\zeta(6)&=\lim_{n\to\infty}\left(\frac{64\pi^6}{63}\right) \sum _{i=1}^{2^{n-2}}\frac{1}{ \left(2^n\sin\left(\frac{\left(2 i-1\right)\pi}{2^n}\right)\right)^6}\\\notag
&=\lim_{n\to\infty}\left(\frac{64\pi^6}{63}\right) \frac{1}{ 2^{6n}}\frac{1}{960}( 2^{2n+6}+5*2^{4n+1}+2^{6n})\\\notag
&=\frac{\pi^6}{945} \\\notag
\end{align}
\end{proof}
\section{Zeta(3)}
$\zeta(3)$ using odd values would be

\begin{align*}
\zeta(3)&=\lim_{n\to\infty}\left(\frac{8\pi^3}{7}\right) \sum _{i=1}^{2^{n-2}}\frac{1}{ \left(2^n\sin\left(\frac{\left(2 i-1\right)\pi}{2^n}\right)\right)^3} \\
&=\lim_{n\to\infty}\left(\frac{8}{7\ 2^{3 n}}\right) \sum  _{i=1}^{2^{n-2}} \left(\Gamma \left(\frac{2 i-1}{2^n}\right) \Gamma \left(1-\frac{2 i-1}{2^n}\right)\right)^3
\end{align*}

If we sum over all integers, 
\begin{align*}
\zeta(3)&=\lim_{n\to\infty}\left(\frac{\pi^3}{2^{3(n+1)}}\right) \sum _{i=1}^{2^{n}}\frac{1}{ \left(\sin\left(\frac{i \pi}{2^{n+1}}\right)\right)^3} \\
&=\lim_{n\to\infty}\left(\frac{1}{ 2^{3( n+1)}}\right) \sum  _{i=1}^{2^{n}} \left(\Gamma \left(\frac{i}{2^{n+1}}\right) \Gamma \left(1-\frac{i}{2^{n+1}}\right)\right)^3
\end{align*}

We will follow the video and focus initially on summation over odd values for $\zeta(3)$. Define $S_n$ as
\begin{align*}
S_n&=\sum  _{i=1}^{2^{n-2}}\frac{1}{ \left(2^n\sin\left(\frac{\left(2 i-1\right)\pi}{2^n}\right)\right)^3} \\
\end{align*}
We look at some values.

\begin{align*}\notag
S_2&=\frac{1}{ \left(2^2\sin\left(\frac{\pi}{4}\right)\right)^3} \\ 
&=\frac{1}{2^6}\sqrt{2}\\
S_3&=\sum _{i=1}^{2}\frac{1}{ \left(2^3\sin\left(\frac{\left(2 i-1\right)\pi}{8}\right)\right)^3} \\
&=\frac{1}{2^9}\left(\frac{1}{\left(\frac{1}{2} \left(1-\frac{1}{\sqrt{2}}\right)\right)^{3/2}}+\frac{1}{\left(\frac{1}{2} \left(1+\frac{1}{\sqrt{2}}\right)\right)^{3/2}}\right)\\
&=\frac{1}{2^7}\sqrt{2(10+\sqrt{2})}\\
S_4&=\frac{1}{2^{12}}\Big[\frac{1}{\left(\frac{1}{2}\left(1-\sqrt{\frac{1}{2} \left(1-\frac{1}{\sqrt{2}}\right)}\right)\right)^{3/2}}+\frac{1}{\left(\frac{1}{2}\left(1+\sqrt{\frac{1}{2} \left(1-\frac{1}{\sqrt{2}}\right)}\right)\right)^{3/2}}\\&+\frac{1}{\left(\frac{1}{2}\left(1-\sqrt{\frac{1}{2} \left(1+\frac{1}{\sqrt{2}}\right)}\right)\right)^{3/2}}
+\frac{1}{\left(\frac{1}{2}\left(1+\sqrt{\frac{1}{2} \left(1+\frac{1}{\sqrt{2}}\right)}\right)\right)^{3/2}}\Big]\\
&=\frac{1}{2^{9}}\sqrt{300-6 \sqrt{2}+\sqrt{302 \sqrt{2}+436}}
\end{align*}

\small
\begin{align*}
\small
S_5&=\frac{1}{2^{12}}\Big[\frac{1}{\left(2-\sqrt{2-\sqrt{2-\sqrt{2}}}\right)^{3/2}}+\frac{1}{\left(2+\sqrt{2-\sqrt{2-\sqrt{2}}}\right)^{3/2}}+\frac{1}{\left(2-\sqrt{2+\sqrt{2-\sqrt{2}}}\right)^{3/2}}\\&+\frac{1}{\left(2+\sqrt{2+\sqrt{2-\sqrt{2}}}\right)^{3/2}}+\frac{1}{\left(2-\sqrt{2-\sqrt{2+\sqrt{2}}}\right)^{3/2}}+\frac{1}{\left(2+\sqrt{2-\sqrt{2+\sqrt{2}}}\right)^{3/2}}\\&+\frac{1}{\left(2-\sqrt{2+\sqrt{2+\sqrt{2}}}\right)^{3/2}}+\frac{1}{\left(2+\sqrt{2+\sqrt{2+\sqrt{2}}}\right)^{3/2}}\Big]\\
&=\frac{1}{2^{12}}\sqrt{2 \left(2 \left(4516+44 \sqrt{2}-\sqrt{3466-137 \sqrt{2}}\right)+\sqrt{428260-79990 \sqrt{2}+\sqrt{32962925198 \sqrt{2}+52633222484}}\right)}
\end{align*}
In passing from $S_n$ to $S_{n+1}$ we make use of the trig identities

\begin{align*}\notag
\sin(\frac{\theta}{2})&=\pm\sqrt{\frac{1-\cos(\theta)}{2}}\\\notag
\sin(\frac{\pi-\theta}{2})&=\pm\sqrt{\frac{1+\cos(\theta)}{2}}
\end{align*}
Note that, for example, the term
\begin{equation}
\frac{1}{\left(\frac{1}{2} \left(1-\sqrt{\frac{1}{2} \left(1-\sqrt{\frac{1}{2} \left(1-\sqrt{\frac{1}{2} \left(1-\frac{1}{\sqrt{2}}\right)}\right)}\right)}\right)\right)^{3/2}}\notag
\end{equation}
can be wrtitten more compactly as
\begin{equation}
\frac{8}{\left(2-\sqrt{2-\sqrt{2-\sqrt{2-\sqrt{2}}}}\right)^{3/2}}\notag
\end{equation}
and equals
\small
\begin{equation}
\left(2\left(2-\sqrt{2}\right) \left(2-\sqrt{2-\sqrt{2}}\right) \left(2-\sqrt{2-\sqrt{2-\sqrt{2}}}\right) \left(2+\sqrt{2-\sqrt{2-\sqrt{2-\sqrt{2}}}}\right)\right)^{3/2}\notag
\end{equation}
\begin{definition}
Let $V_n$ be the set of terms of the form $\sin(\frac{i\pi}{2^n}))$ where $i\le 2^{n-1}$. Call the members of $V_n$ the elementary terms. 
\end{definition}
Then we let $E_n$ be the extension of $Q$ generated by adjoining elemrents of $V_n$. Trig functions show that any product of terms in $E_n$ can be written as a sum of elementary terms with rational coefficients. Hence the degree of the extension $E_n$ is $\le 2^{n-1}$. We will work with terms of the form $2\sin(\frac{i\pi}{2^n})$ for convenience., as in,\\
$$2\sin(\frac{7\pi}{16})=\sqrt{2+\sqrt{2+\sqrt{2}}}$$
We call the maximum number of nested roots the depth of the term. The term above has depth 3. Terms comprising $S_n$ have depth $n-1$. There are $2^{n-1}$ elementary terms of depth $n$. The terms in $V_n$ have depth $\le n-1$. The minimal polynomial of terms of depth $n$ over $Q$ is
$$p_n(x)=(\cdots(((x^2-2)^2-2)^2-2)^2-2\cdots)^2-2$$ where $n$ power of $2$ appear, yielding deg$(p(x))= 2^{n}$. The roots of $p(x)$ are the elementary terms of depth $n$, each appearing with a "+" and a "-" sign. The extension is normal, the $2^n$ powers of a root (elementary term of depth $n$) form a basis of the extension which has rank $2^n$. The extension $E_n$ contains all the terms of $V_{n+1}$. The terms of $V_{n+1}$ form an alternate basis. The Galois group of $p_n(x)$ is the cyclic group $\sigma(n)$.\\

\begin{proposition}
The minimal polynomial for $\sin(\frac{i\pi}{2^n})$ is
$$f_n(x)=((\cdots((((2x)^2-2)^2-2)^2-2)^2-2\cdots)^2-2)/2$$ 
where the degree of $f_n(x)$ is $2^{n-1}$. The depth of the roots is $n-1$.
In $f(x)$, we let $c_{n,m}$ be the coefficient of $x^{2m}$ (the $m+1$-st term), then
$$f_n(x)=1+c_{n,1}x^2+c_{n,2}x^4+\cdots+c_{n,2^{n-2}}x^{2^{n-1}}$$
where
\begin{align}\notag
c_{n,m}&=(-1)^m \frac{ 2^{n+2m-2}} {2^{n-2}+m}\binom{2^{n-2}+m}{2^{n-2}-m},\ \ \  1\le m\le2^{n-2}\\\notag
\end{align}

\end{proposition}
Note if you replace each $x^k$ with $x^{r*k}$, where $r\in \mathbb{Z}^+$, this becomes the minimal polynomial for  $(\sin(\frac{i\pi}{2^n}))^{1/r}$.\\
Example: depth 4, $n=5$
$$\sin(\frac{\pi}{32})=\frac{\sqrt{2-\sqrt{2+\sqrt{2+\sqrt{2}}}}}{2}$$
\begin{align}\notag
f_n(x)=&((((((2x)^2-2)^2-2)^2-2)^2-2)^2-2)/2\\\notag
=&1-128 x^2+2688 x^4-21504 x^6+84480 x^8-180224 x^{10}+212992 x^{12}-131072 x^{14}+32768 x^{16}\\\notag
=&2^{15}(x-\frac{\sqrt{2-\sqrt{2+\sqrt{2+\sqrt{2}}}}}{2})(x+\frac{\sqrt{2-\sqrt{2+\sqrt{2+\sqrt{2}}}}}{2})(x-\frac{\sqrt{2-\sqrt{2+\sqrt{2-\sqrt{2}}}}}{2})\\\notag
&(x+\frac{\sqrt{2-\sqrt{2+\sqrt{2-\sqrt{2}}}}}{2})(x-\frac{\sqrt{2-\sqrt{2-\sqrt{2-\sqrt{2}}}}}{2})(x+\frac{\sqrt{2-\sqrt{2-\sqrt{2-\sqrt{2}}}}}{2})\\\notag
&(x-\frac{\sqrt{2-\sqrt{2-\sqrt{2+\sqrt{2}}}}}{2})(x+\frac{\sqrt{2-\sqrt{2-\sqrt{2+\sqrt{2}}}}}{2})(x-\frac{\sqrt{2+\sqrt{2-\sqrt{2+\sqrt{2}}}}}{2})\\\notag
&(x+\frac{\sqrt{2+\sqrt{2-\sqrt{2+\sqrt{2}}}}}{2})(x-\frac{\sqrt{2+\sqrt{2-\sqrt{2-\sqrt{2}}}}}{2})(x+\frac{\sqrt{2+\sqrt{2-\sqrt{2-\sqrt{2}}}}}{2})\\\notag
&(x-\frac{\sqrt{2+\sqrt{2+\sqrt{2-\sqrt{2}}}}}{2})(x+\frac{\sqrt{2+\sqrt{2+\sqrt{2-\sqrt{2}}}}}{2})(x-\frac{\sqrt{2+\sqrt{2+\sqrt{2+\sqrt{2}}}}}{2})\\\notag
&(x+\frac{\sqrt{2+\sqrt{2+\sqrt{2+\sqrt{2}}}}}{2})\notag
\end{align}
\begin{definition}
As a shorthand notation, we can refer to an elementary term by its sequence of signs. For example, $ e=\sqrt{2-\sqrt{2+\sqrt{2-\sqrt{2+\sqrt{2}}}}}$ would correspond to $(-+-+)$.\\ We define the $\textrm{sign}(e)$ to equal $(-1)^{k+1}$, where $k=$ number of ``-'' signs in the notation for $e$. 
\end{definition}
Each elementary term $e$ has a corresponding function $e(x)$. For example, $\sqrt{2+\sqrt{2-\sqrt{2+\sqrt{2}}}}$ corresponds to $\sqrt{2+\sqrt{2-\sqrt{2+2x}}}$.

When presenting a list of elementary terms, we order them from smallest to largest in magnitude. The following table gives the correspondence between the angle $\theta$ and the value of $2 \sin(\theta)$ for elementary terms of depth $\le 5$.\\

\begin{center}
\begin{tabular}{|c c c c c c c c c c |} 
 \hline
$\pi/64$ &$\pi/32$ &$3\pi/64$&$\pi/16$ &$5\pi/64$ &$3\pi/32$ &$7\pi/64$ &$\pi/8$ &$9\pi/64$&$5\pi/32$\\ 
 \hline
(- + + +)&(- + +) &(- + + -)&(- +)&(- + - -)&(- + -)&(- + - +)&(-) & (- - - +)&(- - -) \\
 \hline
\end{tabular}

\bigskip
\begin{tabular}{|c c c c c c c c c c |} 
 \hline
$11\pi/64$ &$3\pi/16$  &$13\pi/64$ &$7\pi/32$ &$15\pi/64$ &$\pi/4$ &$17\pi/64$&$9\pi/32) $&$19\pi/64$&$5\pi/16$\\ 
 \hline
(- - - -)&(- -) &(- - + -) & (- - +) &(- - + +)&  () & (+ - + +)& (+ - +)&(+ - + -)  &(+ -) \\
 \hline
\end{tabular}

\bigskip
\begin{tabular}{|c c c c c c c c c c c|} 
 \hline
$21\pi/64$ &$11\pi/32$ &$23\pi64$ &$3\pi/8$ &$25\pi/64$ &$13\pi/32$ &$27\pi/64$&$7\pi/16$ &$29\pi/64$ &$15\pi/32$ &$31\pi/64$\\
 \hline
(+ - - -)& (+ - -)&(+ - - +) & (+) & (+ + - +) &(+ + -)& (+ + - -)& (+ +)&  (+ + + -)& (+ + +) &(+ + + +)\\
 \hline
\end{tabular}
\end{center}

\begin{align*}
S_3&=\frac{1}{2^9}\left(4\left(-1\sqrt{2-\sqrt{2}}+3\sqrt{2+\sqrt{2}}\right)\right)\\
S_4&=\frac{1}{2^{12}}\left(8\left(3\sqrt{2-\sqrt{2+\sqrt{2}}}+11\sqrt{2-\sqrt{2-\sqrt{2}}}-2\sqrt{2+\sqrt{2-\sqrt{2}}}+4\sqrt{2+\sqrt{2+\sqrt{2}}}\right)\right)\\
S_5&=\frac{1}{2^{15}}\Big[8 \Big[-34 \sqrt{2-\sqrt{2+\sqrt{2+\sqrt{2}}}}-8 \sqrt{2-\sqrt{2+\sqrt{2-\sqrt{2}}}}+25 \sqrt{2-\sqrt{2-\sqrt{2-\sqrt{2}}}}+15 \sqrt{2-\sqrt{2-\sqrt{2+\sqrt{2}}}}\\&-45 \sqrt{2+\sqrt{2-\sqrt{2+\sqrt{2}}}}+29 \sqrt{2+\sqrt{2-\sqrt{2-\sqrt{2}}}}+64 \sqrt{2+\sqrt{2+\sqrt{2+\sqrt{2}}}}\Big]\Big]\\
S_6&=\frac{1}{2^{18}}\Big[8 \Big[244 \sqrt{2-\sqrt{2+\sqrt{2+\sqrt{2+\sqrt{2}}}}}-48 \sqrt{2-\sqrt{2+\sqrt{2+\sqrt{2-\sqrt{2}}}}}+85 \sqrt{2-\sqrt{2+\sqrt{2-\sqrt{2-\sqrt{2}}}}}\\&+401 \sqrt{2-\sqrt{2+\sqrt{2-\sqrt{2+\sqrt{2}}}}}-243 \sqrt{2-\sqrt{2-\sqrt{2-\sqrt{\sqrt{2}+2}}}}-137 \sqrt{2-\sqrt{2-\sqrt{2-\sqrt{2-\sqrt{2}}}}}\\&-16 \sqrt{2-\sqrt{2-\sqrt{2+\sqrt{2-\sqrt{2}}}}}+382 \sqrt{2-\sqrt{2-\sqrt{2+\sqrt{2+\sqrt{2}}}}}+74 \sqrt{2+\sqrt{2-\sqrt{2+\sqrt{2+\sqrt{2}}}}}\\&+214 \sqrt{2+\sqrt{2-\sqrt{2+\sqrt{2-\sqrt{2}}}}}+97 \sqrt{2+\sqrt{2-\sqrt{2-\sqrt{2-\sqrt{2}}}}}+135 \sqrt{2+\sqrt{2-\sqrt{2-\sqrt{2+\sqrt{2}}}}}\\&-15 \sqrt{2+\sqrt{2+\sqrt{2-\sqrt{2+\sqrt{2}}}}}+109 \sqrt{2+\sqrt{2+\sqrt{2-\sqrt{2-\sqrt{2}}}}}-16 \sqrt{2+\sqrt{2+\sqrt{2+\sqrt{2-\sqrt{2}}}}}\\&-184 \sqrt{2+\sqrt{2+\sqrt{2+\sqrt{2+\sqrt{2}}}}}\Big]\Big]\\
\end{align*}

Note
\begin{equation}\notag
\left(
\begin{array}{c}
\left(\sin(\frac{\pi}{16}\right)^{-3} \\
\left(\sin(\frac{3\pi}{16}\right)^{-3} \\
\left(\sin(\frac{5\pi}{16}\right)^{-3}  \\
\left(\sin(\frac{7\pi}{16}\right)^{-3} \\
\end{array}
\right)=
2^3
\left(
\begin{array}{c}
\left(2-\sqrt{2+\sqrt{2}}\right)^{-3/2} \\
\left(2-\sqrt{2-\sqrt{2}}\right)^{-3/2}\\
\left(2+\sqrt{2-\sqrt{2}}\right)^{-3/2} \\
\left(2+\sqrt{2+\sqrt{2}}\right)^{-3/2}\\
\end{array}
\right)
\end{equation}
While
\begin{align}\notag
S_4&=\frac{1}{2^{3*4}}\sum_{i=1}^{4}\left(    \sin\left(   \frac{\left(2i-1\right)\pi}{16}     \right )    \right)^{-3}\\\notag
\end{align}
We can express the terms in the $S_n$ of the form $1/\sin^3(\frac{\left(2 i-1\right)\pi}{2^n})$ as linear sums of elementary terms, e.g.,
\begin{equation}\label{basis_change}\notag
\left(
\begin{array}{c}
\left(2-\sqrt{2+\sqrt{2}}\right)^{-3/2} \\
\left(2-\sqrt{2-\sqrt{2}}\right)^{-3/2}\\
\left(2+\sqrt{2-\sqrt{2}}\right)^{-3/2} \\
\left(2+\sqrt{2+\sqrt{2}}\right)^{-3/2}\\
\end{array}
\right)=\frac{1}{2}\left(
\begin{array}{cccc}
 2 & 5 & 7 & 8 \\
 7 & 2 & -8 & 5 \\
 5 & 8 & 2 & -7 \\
 -8 & 7 & -5 & 2 \\
\end{array}
\right)\left(
\begin{array}{c}
\sqrt{2-\sqrt{2+\sqrt{2}}} \\
\sqrt{2-\sqrt{2-\sqrt{2}}}\\
\sqrt{2+\sqrt{2-\sqrt{2}}} \\
\sqrt{2+\sqrt{2+\sqrt{2}}}\\
\end{array}
\right)
\end{equation}
So, in the above equation, summing the terms on the left yields $\frac{1}{2^{3(n-1)}}S_n$.
The matrix for the depth 3 case:\\
\begin{equation}\notag
\left(
\begin{array}{cc}
 1 & 2 \\
 -2 & 1 \\
\end{array}
\right)
\end{equation}
The depth 4 case:
\begin{equation}\notag
\left(
\begin{array}{cccc}
 2 & 5 & 7 & 8 \\
 7 & 2 & -8 & 5 \\
 5 & 8 & 2 & -7 \\
 -8 & 7 & -5 & 2 \\
\end{array}
\right)
\end{equation}
The depth 5 case:\\
\begin{equation}\notag
\left(
\begin{array}{cccccccc}
 4 & 11 & 17 & 22 & 26 & 29 & 31 & 32 \\
 -29 & 4 & 26 & -31 & 11 & 22 & -32 & 17 \\
 31 & -22 & 4 & 17 & -29 & 32 & -26 & 11 \\
 -26 & -17 & 31 & 4 & -32 & 11 & 29 & -22 \\
 -22 & -29 & 11 & 32 & 4 & -31 & -17 & 26 \\
 -11 & -26 & -32 & -29 & -17 & 4 & 22 & 31 \\
 17 & 32 & 22 & -11 & -31 & -26 & 4 & 29 \\
 -32 & 31 & -29 & 26 & -22 & 17 & -11 & 4 \\
\end{array}
\right)
\end{equation}
\medskip
Call these transformation matrices $M_n$, the above being $M_5$. The dimension of $M_n$ is $2^{n-2}\text{x }2^{n-2}$, that of $v_n$ is $2^{n-2}$. 
The algorithm for producing these $M_n$ matricies is due to Aaron Fairbanks. The Mathematica code for producing the matrix for terms of $M_m$ is given below. 
\begin{lstlisting}
n = 2^(m-2); M = IdentityMatrix[n];
For[i = 1, i <= n, i++, 
 For[j = 1, j <= n, j++, r = Mod[(i + j - 1) (2 i - 1)^(n - 1), 2 n]; 
  M[[i, j]] = ((-1)^((r (2 i - 1) - (i + j - 1))/(2 n))) (1/
      2) (2 n r - r^2 + r - n);
  ]]
\end{lstlisting}
The code for producing $(2^{3(m-1)})S_m$ is
\begin{lstlisting}
n=2^(m-2);S=0;
For[i=1,i<=n,i++, 
 For[j=1,j<=n,j++, 
  r=Mod[(i+j-1)(2i-1)^(n-1),2n]; 
  S=S+((-1)^((r(2i-1)-(i+j-1))/(2n)))* 
        (1/4)(2nr-r^2+r-n)Sin[(2j-1)Pi/(4n)];
 ]
]
\end{lstlisting}
Another version:
\begin{lstlisting}
n=2^(m-2);S=0;
For[i=1,i<=n,i++, 
 For[j=1,j<=n,j++, 
  p=i+(j-1)(2(i-1)+1);
  k=Mod[p-1,2n]+1;
  S=S+(-1)^((p-k)/(2n))*
       (1/4)(2nj-j^2+j-n)Sin[(2k-1)Pi/(4n)];
 ]
]
\end{lstlisting}
\begin{proposition}
Regarding the $2^{n-2}x2^{n-2}$ matrices, $M_n$ which satisfy $$\frac{1}{(\sin(\frac{(2j-1)\pi}{2^n}))^3}=8 \sum_{k=1}^{2^{n-2}}M_n(j,k)\sin(\frac{(2k-1)\pi}{2^n})$$
Let
\begin{align}\notag
p &= i + (j - 1) (2 (i - 1) + 1)\\\notag
k &= \floor{\frac{(p - 1)}{2^{n-2}}}\\\notag
m &= (-1)^k (p - k(2^{n-2}))\pmod*{(2^{n-2}+1)}\notag
\end{align}
then
\begin{align}\notag
M_n[i,m]&= (1/2) (-1)^{ k }(2^{n-1}j - j^2 + j -2^{n-2})
\end{align}
\end{proposition}
\begin{proof}
We will refer to the set of elementary terms as the alternate basis. The first few matricies can be verified using basic trig identities. Consider terms of the form
$$\frac{1}{ \left(\sin\left(\frac{\left(2 i-1\right)\pi}{2^n}\right)\right)^3} $$ 
One can rationalize the denominator by repeated use of 
$$\sin(\alpha)\cos(\alpha)=\frac{1}{2}\sin(2\alpha)$$ multiplying numerator and denominator by $\cos(\alpha)$ for the appropriate $\alpha$. One can then use formulas such as $$\cos(\alpha)\cos(\beta)=\frac{1}{2}(\cos(\alpha-\beta)+\cos(\alpha+\beta))$$ to eliminate all products. The end result can then be expressed in terms of the new basis.
Note given one root $\rho$ of the minimal polynomial for the alternate basis, the powers of $\rho$: $1,\rho,\rho^2,\rho^3,\cdots$ also form a basis. One can multiply out these powers and solve for the alternate basis in terms of these powers. Then, if we map another root $\gamma$ to $\rho$, the morphism applies to polynomials in $\rho$, and we can see how the other roots are permuted. Hence, for example, given $\frac{1}{ \left(\sin\left(\frac{\left(2 i-1\right)\pi}{2^n}\right)\right)^3} $ expressed in terms of the alternate basis, we can apply the permutation to express the other terms in the alternate basis. This shows that the set of values in each row is the same, the order and sign change. Given our ordering of basis elements, the first row has the form $M[1,2^{n-2}]=2^{n}$,$ M[1,2^{n-2}-j]=M[1,2^{n-2}-(j-1)]-j$ for $j\ge 1$. This can be verified by the tedious process described above. The general pattern of the permutations becomes apparent after working out the first few transformation matricies. 

\end{proof}
\begin{corollary}
\begin{align}\notag
M_n[i,j]&= (-1)^{  \floor{   \frac{((i+j-1)(2i-1)^{(2^{n-2}-1)} }{2^{n-1}}   }     }\Biggl[ \Bigl[  (2^{n-1}+1)\left((i+j-1)(2i-1)^{2^{n-2}-1}\pmod*{2^{n-1}}\right)\\\notag
&-    \left( (i+j-1)(2i-1)^{2^{n-2}-1}\pmod*{2^{n-1}}\right)^{2}-2^{n-2}  \Bigr]\Biggr]\\\notag
\end{align}
\end{corollary}
\begin{proof}
Solve for $j$ in previous Proposition.
\end{proof}
\begin{corollary}
\begin{align}\notag
S_n=\frac{1}{2^{3n-2}}&\sum_{i=1}^{2^{n-2}}\sum_{j=1}^{2^{n-2}}(-1)^{\floor*{\frac{2 i j-i-j}{2^{n-1}}}}(2^{n-1}j-j^2+j-2^{n-2})\sin(      \frac{ (  2[ (2 i j-i-j)\pmod*{2^{n-1}}] +1 )\pi } {2^{n}}      )\\\notag
&=\frac{1}{2^{3n-2}}\sum_{j=1}^{2^{n-2}}(2^{n-1}j-j^2+j-2^{n-2})\sum_{i=1}^{2^{n-2}}(-1)^{\floor*{\frac{2 i j-i-j}{2^{n-1}}}}\sin(      \frac{ (  2[ (2 i j-i-j)\pmod*{2^{n-1}}] +1 )\pi } {2^{n}}      )\\\notag
\end{align}
\end{corollary}
\begin{proof}
Sum terms in previous Proposition.
\end{proof}
\begin{proposition}
Assume $r\in\mathbb{Z}$, $r$ odd. Regarding the $2^{n-2}x2^{n-2}$ matrices, $M_n$ which satisfy $$\frac{1}{(\sin(\frac{(2j-1)\pi}{2^n}))^r}=2^r \sum_{k=1}^{2^{n-2}}M_n(j,k)\sin(\frac{(2k-1)\pi}{2^n})$$ We summarize some facts. 
\begin{itemize}
\item the $M_n$ are normal matricies
\item the set of $M_n$ for a fixed $n$ and all $r$ commute
\item in a given $M_n$ each row contains the same elements with the order permuted and possible sign chage
\item for a fixed $n$ and all $r$ the permutations are the same for all $M_n$
\item for a fixed $n$ and all $r> 0$ the permutations and sign changes are the same for all $M_n$
\item for a fixed $n$ and all $r< 0$ the permutations and sign changes are the same for all $M_n$
\end{itemize}
\end{proposition}
\begin{proof}
Proofs can be found in ref. 20.
\end{proof}
\begin{proposition}
Let $$S_n=\sum_{k=1}^{2^{n-2}}\frac{1}{\sin^r(\frac{(2k-1)\pi}{2^n})}$$
Regarding the $2^{n-2}x2^{n-2}$ matrices, $M_n$ which satisfy $$\frac{1}{(\sin(\frac{(2j-1)\pi}{2^n}))^r}=2^r \sum_{k=1}^{2^{n-2}}M_n(j,k)\sin(\frac{(2k-1)\pi}{2^n})$$
$$S_n=2^{r-1} \sum_{j=1}^{2^n-2}M(1,j)\frac{1}{\sin(\frac{(2j-1)\pi}{2^n})}$$
\end{proposition}
\begin{proof}
See ref. 20.
\end{proof}

\begin{theorem}\label{cscthm}
\begin{align}\notag
\zeta(3)&=\frac{2\pi^3}{7}\int_{0}^{1/2}(x(1-x))\left(  \frac{ 1  }{  \sin\left(\pi x\right)  }  \right)dx\\\notag
&=\frac{2}{7}\int_{0}^{\pi/2}(x(\pi-x)) \csc\left( x\right)dx\\\notag
&=\frac{1}{7}\int_{0}^{\pi}(x(\pi-x)) \csc\left( x\right)dx\\\notag
\end{align}
\end{theorem}
\begin{proof}
\begin{align}\notag
\zeta(3)&=\lim_{n\to\infty}\frac{\pi^3}{7(2^{3n-5})}\sum_{j=1}^{2^{n-2}}(2^{n-1}j-j^2+j-2^{n-2})\sum_{i=1}^{2^{n-2}}(-1)^{\floor*{\frac{2 i j-i-j}{2^{n-1}}}}\sin(      \frac{ (  2[ (2 i j-i-j)\pmod*{2^{n-1}}] +1 )\pi } {2^n}      )\\\notag
&=\lim_{n\to\infty}\frac{\pi^3}{7(2^{3n-5})}\sum_{j=1}^{2^{n-2}}((2^{n-1}-j)(j-1)+2^{n-2})\sum_{i=1}^{2^{n-2}}(-1)^{\floor*{\frac{2 i j-i-j}{2^{n-1}}}}\sin(      \frac{ (  2[ (2 i j-i-j)\pmod*{2^{n-1}}] +1 )\pi } {2^n}      )\\\notag
&=\lim_{n\to\infty}\frac{\pi^3}{7(2^{3n-4})}\sum_{j=1}^{2^{n-2}}((2^{n-1}-j)(j-1)+2^{n-2})\left(  \frac{ 1  }{  \sin\left(\frac{(2 j-1)\pi}{2^n}\right)  }  \right)\\\notag
&=\lim_{n\to\infty}\frac{\pi^3}{7(2^{n-2})}\sum_{j=1}^{2^{n-2}}((1-\frac{j}{2^{n-1}})(\frac{j}{2^{n-1}}))\left(  \frac{ 1  }{  \sin\left(\frac{j\pi}{2^{n-1}}\right)  }  \right)\\\notag
&=\frac{2\pi^3}{7}\int_{0}^{1/2}(x(1-x))\left(  \frac{ 1  }{  \sin\left(\pi x\right)  }  \right)dx\\\notag
\end{align}
Step 3 is justified by the previous proposition.
\end{proof}
\section{Zeta At Odd Integer Values}
Recall

$$\zeta(m)=\lim_{n\to\infty}\left(\frac{2^m\pi^m}{2^{m}-1}\right) \sum _{i=1}^{2^{n-2}}\frac{1}{ \left(2^n\sin\left(\frac{\left(2 i-1\right)\pi}{2^n}\right)\right)^m} $$\notag\\
We can repeat the analysis for $n=5$.
$$\zeta(5)=\lim_{n\to\infty}\left(\frac{32\pi^5}{31}\right) \sum _{i=1}^{2^{n-2}}\frac{1}{ \left(2^n\sin\left(\frac{\left(2 i-1\right)\pi}{2^n}\right)\right)^5} $$\notag\\
We can express the terms in the $S_n$ of the form $1/\sin^5(\frac{\left(2 i-1\right)\pi}{2^n})$ as linear sums of elementary terms, e.g.,
\begin{equation}\label{basis_change}
\left(
\begin{array}{c}
\left(2-\sqrt{2+\sqrt{2}}\right)^{-5/2} \\
\left(2-\sqrt{2-\sqrt{2}}\right)^{-5/2}\\
\left(2+\sqrt{2-\sqrt{2}}\right)^{-5/2} \\
\left(2+\sqrt{2+\sqrt{2}}\right)^{-5/2}\\
\end{array}
\right)=\frac{1}{2}\left(
\begin{array}{cccc}
 11& 31 & 46 & 54 \\
 46 & 11 & -54 & 31 \\
 31 & 54 & 11 & -46 \\
 -54 & 46 & -31 & 11 \\
\end{array}
\right)\left(
\begin{array}{c}
\sqrt{2-\sqrt{2+\sqrt{2}}} \\
\sqrt{2-\sqrt{2-\sqrt{2}}}\\
\sqrt{2+\sqrt{2-\sqrt{2}}} \\
\sqrt{2+\sqrt{2+\sqrt{2}}}\\
\end{array}
\right)
\end{equation}
The matrix for the depth 3 case:\\
\begin{equation}\notag
\left(
\begin{array}{cc}
 3 & 7 \\
 -7 & 3 \\
\end{array}
\right)
\end{equation}
The depth 4 case:
\begin{equation}\notag
\left(
\begin{array}{cccc}
 11 & 31& 46& 54 \\
 46& 11 & -54 & 31 \\
 31& 54 & 11 & -46 \\
 -54 & 46 & -31 & 11 \\
\end{array}
\right)
\end{equation}
The depth 5 case:\\
\begin{equation}\notag
\left(
\begin{array}{cccccccc}
86& 254& 411& 551& 669& 761& 824& 856 \\
-761& 86& 669& -824& 254& 551& -856& 411 \\
824& -551& 86& 411& -761& 856& -669& 254\\
-669& -411& 824& 86& -856& 254& 761& -551 \\
-551& -761& 254& 856& 86& -824& -411& 669 \\
-254& -669& -856& -761& -411& 86& 551& 824 \\
 411& 856& 551& -254& -824& -669& 86& 761\\
-856& 824& -761& 669& -551& 411& -254& 86 \\
\end{array}
\right)
\end{equation}
The Mathematica code for producing the matrix for terms of $M_{m+2}$ is given below.
\begin{lstlisting}
 s = 11; For[j = 3, j <= m, j++, s = 2^(3 (j - 1)) + 2 s; ]
n = 2^m; M = IdentityMatrix[n]; M[[1, 1]] = s;
M[[1, n]] = 2^(2*m - 1);
For[i = 2, i < n + 1, i++, 
 M[[1, i]] = 
  M[[1, 1]] + 
   1/24 (-1 + i) (i - 2 n) (i^2 - i (1 + 2 n) + 
      2 (-1 - 6 M[[1, n]] + n^2))]; 
For[i = 2, i <= n, i++, 
  For[j = 1, j <= n, j++, p = i + (j - 1) (2 (i - 1) + 1);
   k = Floor[(p - 1)/n];
   ind = (-1)^k (p - k*n);
   M[[i, ind]] = (-1)^Floor[(p - 1)/(2 n)] M[[1, j]];]] ;
 \end{lstlisting} 
The code for producing $(2^{5(m-1)})S_m$ is
\begin{lstlisting}
 s = 11; For[j = 3, j <= m, j++, s = 2^(3 (j - 1)) + 2 s; 
n = 2^m; sum = 0;
For[i = 1, i <= n, i++, 
  For[j = 1, j <= n, j++, p = i + (j - 1) (2 (i - 1) + 1);
   k = Floor[(p - 1)/n];
   ind = (-1)^k (p - k*n); 
   If[ind < 0, ind = ind + 2^m + 1, ind = ind]; 
   sum = sum + (-1)^
       Floor[(p - 1)/2^(m + 1)] (s + 
        1/24 (j - 1) (j - 2^(m + 1)) (j^2 - j (1 + 2^(m + 1)) + 
           2 (-1 - 2^(2 m + 1)))) Sin[(2 ind - 1) \[Pi]/(2^(
          m + 2))];]]; 
 \end{lstlisting} 
Another version:
\begin{lstlisting}
 s = 11; For[j = 3, j <= m - 2, j++, s = 2^(3 (j - 1)) + 2 s; 
sum = 0;
sum= Simplify[
   32 Pi^5/(31*2^(
       m - 1)) Sum[(1/
        24 (j/2^(
         m - 1)) (j/2^(m - 1) - 1) ((j/2^(m - 1))^2 - j/2^(m - 1) + 
          2 (-(1/2)))) Sum[(-1)^
         Floor[(2 i j - i - j)/2^(
           m - 1)] Sin[(2 Mod[(2 i j - i - j), 2^(m - 1)] + 
            1) \[Pi]/(2^m)], {i, 1, 2^(m - 2)}], {j, 1, 2^(m - 2)}]];
 \end{lstlisting}     
 \begin{proposition}
Regarding the $2^{n-2}x2^{n-2}$ matrices, $M_n$ which satisfy $$\frac{1}{(\sin(\frac{(2j-1)\pi}{2^n}))^5}=32 \sum_{k=1}^{2^{n-2}}M_n(j,k)\sin(\frac{(2k-1)\pi}{2^n})$$
Let
\begin{align}\notag
p &= i + (j - 1) (2 (i - 1) + 1)\\\notag
k &= \floor{\frac{(p - 1)}{2^{n-2}}}\\\notag
m &= (-1)^k (p - k(2^{n-2}))\pmod*{(2^{n-2}+1)}\notag
\end{align}
then
\begin{align}\notag
M_n[i,m]&=  (-1)^{ k }\Bigl(\frac{1}{24} (j-1) \left(j-2^{n-1}\right) \left(j^2-j \left(2^{n-1}+1\right)+2 \left(-6\ 2^{2 n-5}+2^{2 n-4}-1\right)\right)\\\notag
&+\frac{1}{3} 2^{n+1} \left(4^{n-4}-1\right)+11\ 2^{n-4}\Bigr)\\\notag
&= (-1)^{ k }\frac{1}{24} \left(j^4-2 j^3 \left(2^{n-1}+1\right)+j^2 \left(3\ 2^{n-1}-1\right)+2 j \left(2^{n-2}+2^{3 n-4}+1\right)-2^{n-1} \left(2^{2 n-3}+1\right)\right)\\\notag
\end{align}
\end{proposition}       
 As in the case of $\zeta(3)$
 
\begin{theorem}\label{cscthm2}
\begin{align}\notag
\zeta(5)&=\frac{2\pi^5}{93}\int_{0}^{1/2}(x (x - 1) (x^2 - x - 1))\left(  \frac{ 1  }{  \sin\left(\pi x\right)  }  \right)dx\\\notag
\end{align}
\end{theorem}
\begin{proof}
Procedes as in previous theorem.
\end{proof}
We can continue with $\zeta(7)$, $\zeta(9)$, etc.

\begin{align}
\zeta(3)&=\frac{2\pi^3}{7}\int_{0}^{1/2}(x(1-x))\csc (\pi  x)dx\\\notag
\zeta(5)&=\frac{2\pi^5}{93}\int_{0}^{1/2}(x (x - 1) (x^2 - x - 1))\csc (\pi  x)dx\\\notag
\zeta(7&)=-\frac{4 \pi ^7}{5715}\int_{0}^{1/2} x (x-1) \left(x^4-2 x^3-2 x^2+3 x+3\right) \csc (\pi  x)dx\\\notag
\zeta(9)&=\frac{2 \pi ^9}{160965}\int_{0}^{1/2}x (x-1) \left(x^6-3 x^5-3 x^4+11 x^3+11 x^2-17 x-17\right)\csc (\pi  x)dx\\\notag
\zeta(11&)=-\frac{4 \pi ^{11}}{29016225}\int_{0}^{1/2}x (x-1) \left(x^8-4 x^7-4 x^6+26 x^5+26 x^4-100 x^3-100 x^2+155 x+155\right)\csc (\pi  x)dx\\\notag
\zeta(13)&=\frac{4 \pi ^{13}}{3831545025}\int_{0}^{1/2}x (x-1) \Bigl(x^{10}-5 x^9-5 x^8+50 x^7+50 x^6-346 x^5-346 x^4+1337 x^3+1337 x^2\\
&-2073 x-2073\Bigr)\csc (\pi  x)dx\\\notag
\zeta(15)&=-\frac{8 \pi ^{15}}{1394810091675}\int_{0}^{1/2}x (x-1) \Bigl(x^{12}-6 x^{11}-6 x^{10}+85 x^9+85 x^8-916 x^7-916 x^6+6377 x^5+6377 x^4\\
&-24654 x^3-24654 x^2+38227 x+38227\Bigr)\csc (\pi  x)dx\\\notag
\zeta(17)&=\frac{2 \pi ^{17}}{83690521039125}\int_{0}^{1/2}x (x-1) \Bigl(x^{14}-7 x^{13}-7 x^{12}+133 x^{11}+133 x^{10}-2051 x^9-2051 x^8+22259 x^7\\
&+22259 x^6-155061 x^5-155061 x^4+599511 x^3+599511 x^2-929569 x-929569\Bigr)\csc (\pi  x)dx\\\notag
&\vdots\\\notag
\zeta(2j+1)&=\frac{2^{2j}\pi^{2j+1}}{(2j)!(2^{2j+1}-1)}\int_{0}^{1/2}x(x-1)\Bigl(x^{2j-2}-(j-1)(x^{2j-3}+x^{2j-4})+\frac{1}{6}(j-1)(j-2)(2j+3)(x^{2j-5}+x^{2j-6})\\\notag
&-\frac{1}{30}(j-1)(j-2)(j-3)(4j^2+4j+5)(x^{2j-7}+x^{2j-8})+\frac{1}{2520}(j-1)(j-2)(j-3)(j-4)(136j^3+\\\notag
&-68j^2+174j+105)(x^{2j-9}+x^{2j-10})+\cdots+2(4^j-1)B(2j)(x+1)\Bigr)\csc (\pi  x)dx\\\notag
\end{align}
where $B(k)$ is the $k$-th Bernoulli number.\\
Note that
\begin{proposition}
Regarding the $2^{n-2}x2^{n-2}$ matrices, $M_n$ which satisfy $$\frac{1}{\sin(\frac{(2j-1)\pi}{2^n})}=2 \sum_{k=1}^{2^{n-2}}M_n(j,k)\sin(\frac{(2k-1)\pi}{2^n})$$
Let
\begin{align}\notag
p &= i + (j - 1) (2 (i - 1) + 1)\\\notag
k &= \floor{\frac{(p - 1)}{2^{n-2}}}\\\notag
m &= (-1)^k (p - k(2^{n-2}))\pmod*{(2^{n-2}+1)}\notag
\end{align}
then
\begin{align}\notag
M_n[i,m]&=  (-1)^{ k }\\\notag
\end{align}
\end{proposition}       
\section{Some Trig Integrals}
\begin{definition}
$\eta(k)=(1-2^{1-k})\zeta(k)$ is the Dirichlet Eta function.
\end{definition}
\begin{lemma}
If $f$ is periodic of period $\pi$ then
$$\int_{0}^{\pi/2}f(x) dx=\int_{0}^{\infty}f(x)\frac{\sin(x)}{x}dx$$
\end{lemma}
\begin{proof}
(Lobachevsky)
\end{proof}
\begin{corollary}
\begin{align}
\zeta(3)&=\frac{1}{7}\pi^2\left(  1+\sum_{j=1}^{\infty}(-1)^j \left(  (2j+1)+2j(j+1)\ln\left(\frac{j}{j+1}\right)\right) \right)\\\notag
\end{align}
\end{corollary}
\begin{proof}
From Theorem \ref{cscthm},
$$\zeta(3)=\frac{2}{7}\int_{0}^{\pi/2}(x(\pi-x)) \csc( x)dx$$
Let $f(x)$ be the periodic function with period $\pi$ that is defined on $[0,\pi]$ as $x(\pi-x) \csc( x)$. Then by the above lemma, 
\begin{align}\notag
\zeta(3)&=\frac{2}{7}\int_{0}^{\pi/2}(x(\pi-x)) \csc( x)dx\\\notag
&=\frac{2}{7}\sum_{j=0}^{\infty}\int_{j \pi}^{(j+1)\pi}((x-j \pi)(\pi-(x-j\pi))\csc(x-j\pi)\frac{\sin(x)}{x}dx\\\notag
&=\frac{2}{7}\sum_{j=0}^{\infty}(-1)^j\int_{j \pi}^{(j+1)\pi}(1/x)(x-j\pi)((j+1)\pi-x)dx\\\notag
&=\frac{2}{7}\sum_{j=0}^{\infty}(-1)^j\int_{0}^{\pi}x\frac{(\pi-x)}{x+j\pi}dx\\\notag
&=\frac{2}{7}\sum_{j=0}^{\infty}(-1)^j\left(   x((1+j)\pi-\frac{x}{2})-j(j+1)\pi^2\ln(x+j \pi)\right)|_{x=0}^{x=\pi}\\\notag
&=\frac{1}{7}\Bigl(2\pi x-x^2+\sum_{j=1}^{\infty}(-1)^j\left(   x(2(1+j)\pi-x)+2j(j+1)\pi^2\ln(\frac{j\pi}{j\pi+x})\right)\Bigr)|_{x=\pi}\\\notag
&=\frac{1}{7}\pi^2\left(  1+\sum_{j=1}^{\infty}(-1)^j \left(  (2j+1)+2j(j+1)\ln\left(\frac{j}{j+1}\right)\right) \right)\\\notag
\end{align}
\end{proof}
\begin{corollary}
Let

$$f(x)=\frac{1}{7}\Bigl(2\pi x-(1+\ln(2))x^2+2\sum_{j=3}^{\infty}(-1)^{j+1}\frac{1}{j \pi^{j-2}}\left(\eta(j-1)+\eta(j-2)\right)x^j\Bigr)\notag$$
Then
$$f(\pi)=\zeta(3)$$
\end{corollary}
\begin{proof}
From proof of previous corollary,
\begin{align}\notag
f(x)&=\frac{2}{7}\sum_{j=0}^{\infty}(-1)^j\left(   x((1+j)\pi-\frac{x}{2})-j(j+1)\pi^2\ln(x+j \pi)\right)|_{x=0}^{x=x}\\\notag
&=\frac{1}{7}\Bigl(2\pi x-x^2+\sum_{j=1}^{\infty}(-1)^j\left(   x(2(1+j)\pi-x)-2j(j+1)\pi^2\ln(\frac{j\pi+x}{j\pi})\right)\Bigr)\\\notag
&=\frac{1}{7}\Bigl(2\pi x-x^2+\sum_{j=1}^{\infty}(-1)^j\left(   x(2(1+j)\pi-x)-2j(j+1)\pi^2\ln(1+\frac{x}{j\pi})\right)\Bigr)\\\notag
&=\frac{1}{7}\Bigl(2\pi x-x^2+\sum_{j=1}^{\infty}(-1)^j\left(   x(2(1+j)\pi-x)-2j(j+1)\pi^2\sum_{k=1}^{\infty}(-1)^{k-1}\frac{1}{k}(\frac{x}{j\pi})^k\right)\Bigr)\\\notag
&=\frac{1}{7}\Bigl(2\pi x-x^2+\sum_{j=1}^{\infty}(-1)^j\left(   x(2(1+j)\pi-x)-2j(j+1)\pi^2\Bigl(\frac{x}{j\pi}-\frac{1}{2}\left(\frac{x}{j\pi}\right)^2+\sum_{k=3}^{\infty}(-1)^{k-1}\frac{1}{k}(\frac{x}{j\pi})^k\Bigr)\right)\Bigr)\\\notag
&=\frac{1}{7}\Bigl(2\pi x-x^2+\sum_{j=1}^{\infty}(-1)^j\left(   x(2(1+j)\pi-x)-2(j+1)\left(\pi x-\frac{x^2}{2 j}\right)+2j(j+1)\pi^2\sum_{k=3}^{\infty}(-1)^{k}\frac{1}{k}(\frac{x}{j\pi})^k\right)\Bigr)\\\notag
&=\frac{1}{7}\Bigl(2\pi x-x^2+\sum_{j=1}^{\infty}(-1)^j\left( (\frac{1}{j})x^2+2j(j+1)\pi^2\sum_{k=3}^{\infty}(-1)^{k}\frac{1}{k}(\frac{x}{j\pi})^k\right)\Bigr)\\\notag
&=\frac{1}{7}\Bigl(2\pi x-(1+\ln(2))x^2+2\sum_{j=1}^{\infty}\sum_{k=3}^{\infty}(-1)^{j+k}\frac{1+j}{k j^{k-1}\pi^{k-2}}x^k\Bigr)\\\notag
&=\frac{1}{7}\Bigl(2\pi x-(1+\ln(2))x^2+2\sum_{j=3}^{\infty}(-1)^{j+1}\frac{1}{j \pi^{j-2}}\left(\eta(j-1)+\eta(j-2)\right)x^j\Bigr)\\\notag
\end{align}
\end{proof}
\begin{corollary}\label{etacor1}
\begin{align}
\zeta(3)&=\frac{2}{7}\pi^2\lim_{n->\infty}\Bigl(\frac{1}{n+2}\eta(n+1)+\sum_{j=1}^{n}\frac{1}{(j+1)(j+2)}\eta(j) \Bigr)\\\notag
&=\frac{2}{7}\pi^2\sum_{j=1}^{\infty}\frac{1}{(j+1)(j+2)}\eta(j)\\
\end{align}
Note: The first formula converges much more rapidly than the second.
\end{corollary}
\begin{proof}
From previous corollary,
\begin{align}
\zeta(3)&=\frac{1}{7}\pi^2\left(  1+\sum_{j=1}^{\infty}(-1)^j \left(  (2j+1)+2j(j+1)\ln\left(\frac{j}{j+1}\right)\right) \right)\\\notag
&=\frac{1}{7}\pi^2\left(  1+\sum_{j=1}^{\infty}(-1)^j \left(  (2j+1)+2j(j+1)\ln\left(1-\frac{1}{j+1}\right)\right) \right)\\\notag
&=\frac{1}{7}\pi^2\left(  1+\sum_{j=1}^{\infty}(-1)^j \left(  (2j+1)+2j(j+1)\left(-\sum_{k=1}^{\infty} \left(\frac{1}{k(j+1)^k}\right) \right)\right)\right) \\\notag
&=\frac{1}{7}\pi^2\left(  1+\sum_{j=1}^{\infty}(-1)^j\Bigl(1+ 2j\Bigl(1-\sum_{k=1}^{\infty}\frac{1}{k(j+1)^{k-1}}\Bigr) \Bigr)\right)\\\notag
&=\frac{1}{7}\pi^2\lim_{n->\infty}\Bigl(  1+2\sum_{j=1}^{n}(-1)^j\Bigl(\frac{1}{2}+ j-\sum_{k=1}^{n}\Bigl(\frac{1}{k(j+1)^{k-2}}-\frac{1}{k(j+1)^{k-1}}\Bigr)\Bigr) \Bigr)\\\notag
&=\frac{1}{7}\pi^2\lim_{n->\infty}\Bigl(  1+2\sum_{j=1}^{n}(-1)^j\Bigl(\frac{1}{2}+  j-\sum_{k=1}^{2}\Bigl(\frac{1}{k(j+1)^{k-2}}-\frac{1}{k(j+1)^{k-1}}\Bigr)-\sum_{k=3}^{n}\Bigl(\frac{1}{k(j+1)^{k-2}}-\frac{1}{k(j+1)^{k-1}}\Bigr)\Bigr) \Bigr)\\\notag
&=\frac{1}{7}\pi^2\lim_{n->\infty}\Bigl(  1+\sum_{j=1}^{n}(-1)^{j}\Bigl(\frac{1}{j+1}\Bigr)-2\sum_{k=3}^{n}\frac{1}{k}\Bigl(\eta(k-2)-\eta(k-1)\Bigr)\Bigr)\\\notag
&=\frac{1}{7}\pi^2\lim_{n->\infty}\Bigl(  \ln(2)-2\sum_{j=1}^{n}\frac{1}{j+2}\Bigl(\eta(j)-\eta(j+1)\Bigr) \Bigr)\\\notag
&=\frac{1}{7}\pi^2\lim_{n->\infty}\Bigl( \frac{1}{3} \ln(2)+2\Bigl(\frac{1}{n+2}\eta(n+1)+\sum_{j=2}^{n}\frac{1}{(j+1)(j+2)}\eta(j) \Bigr)\Bigr)\\\notag
&=\frac{2}{7}\pi^2\lim_{n->\infty}\Bigl(\frac{1}{n+2}\eta(n+1)+\sum_{j=1}^{n}\frac{1}{(j+1)(j+2)}\eta(j) \Bigr)\\\notag
&=\frac{2}{7}\pi^2\sum_{j=1}^{\infty}\frac{1}{(j+1)(j+2)}\eta(j)\\\notag
\end{align}
\end{proof}
\begin{lemma}
\begin{align}
csc(x)&=\frac{1}{x}+\sum_{j=1}^{\infty}\frac{1}{\pi^{2j}}\left(2-\frac{1}{2^{2(j-1)}}\right)\zeta(2j)x^{2j-1}\\\notag
\end{align}
\end{lemma}
\begin{proof}
Arkadiusz Wesolowski [7]
\end{proof}
\begin{lemma}\label{csclemma2}
\begin{align}
\csc(x)&=\sum_{j=0}^{\infty}(-1)^j\frac{1}{(2j)!}E_{2j}(x-\frac{\pi}{2})^{2j}\\\notag
\end{align}
where $E_n$ is the nth Euler number.
\end{lemma}
\begin{proof}
Source: Mathematica.
\end{proof}
\begin{corollary}
\begin{align}
\frac{1}{7}x(\pi-x)\csc(x)&=\frac{1}{7}\Bigl(\pi-x+\sum_{j=1}^{\infty}\frac{1}{\pi^{2j}}\left(2-\frac{1}{2^{2(j-1)}}\right)\zeta(2j)\left(\pi x^{2j}-x^{2j+1}\right)\Bigr)\\\notag
\end{align}
\end{corollary}
\begin{proof}
Arkadiusz Wesolowski's lemma.
\end{proof}
\begin{corollary}
Let
$$f(x)=\frac{1}{7}\int_{0}^{x}(x(\pi-x)) \csc( x)dx\notag$$
then
\begin{align}\label{formula1}
f(x)&=\frac{1}{7}\Bigl(\pi x-\frac{x^2}{2}+\sum_{j=1}^{\infty}\frac{1}{\pi^{2j}}\left(2-\frac{1}{2^{2(j-1)}}\right)\zeta(2j)\left(\frac{\pi x^{2j+1}}{2j+1} -\frac{x^{2j+2}}{2j+2}\right)\Bigr)\\\notag
&=\frac{1}{7}\Bigl(\pi x-\frac{x^2}{2}+2\sum_{j=1}^{\infty}\frac{1}{\pi^{2j}}\eta(2j)\left(\frac{\pi x^{2j+1}}{2j+1} -\frac{x^{2j+2}}{2j+2}\right)\Bigr)\\\notag
\end{align}
\end{corollary}
\begin{proof}
From previous lemma.
\end{proof}

\begin{corollary}\label{sumpi}
Let
$$f(x,n)=\frac{1}{7}\Bigl(\pi x-\frac{x^2}{2}+2\sum_{j=1}^{n}\frac{1}{\pi^{2j}}\eta(2j)\left(\frac{\pi x^{2j+1}}{2j+1} -\frac{x^{2j+2}}{2j+2}\right)\Bigr)\notag$$
and set
$$a_n=\frac{1}{n!}\frac{d^{n-1}}{dx^{n-1}}\Bigl(\frac{x^n}{f(x,n)^n}\Bigr)|_{x=0}$$
Then if
\begin{align}\notag
g(x)&=\sum_{j=1}^{\infty}a_nx^n\\\notag
&=\frac{7 x}{ \pi }+\frac{49 x^2}{4 \pi ^3}-\frac{343 \left(\pi ^2-9\right) x^3}{72 \pi ^5}-\frac{2401 \left(7 \pi ^2-45\right) x^4}{576 \pi ^7}+\frac{16807 \left(4725-900 \pi ^2+29 \pi ^4\right) x^5}{86400 \pi ^9}+\cdots\\\notag
\end{align}
then 
$$g(\zeta(3))=\pi$$
\end{corollary}
\begin{proof}
Apply series reversion to the function in the previous corollary. See [11].
\end{proof}
\begin{corollary}\label{etacor2}
\begin{align}
\zeta(3)&=\frac{2}{7}\pi^2\Bigl(\frac{1}{4}+\sum_{j=1}^{\infty}\left(\frac{1}{(2j+1)(2j+2)} \right)\eta(2j)\Bigr)\\\notag
&=\frac{2}{7}\pi^2\Bigl(\frac{3}{8}+\sum_{j=1}^{\infty}\left(\frac{2j+3}{2^{2j+1}(2j+1)(2j+2)}\right)\eta(2j)\Bigr)\\\notag
\end{align}
Note: This second series converges very quickly, accurate to 7 decimal places with only 10 terms. This is comparable to approximately 5,000 terms of the series $\sum_{j=1}^{\infty}\frac{1}{j^3}$.
\end{corollary}
\begin{proof}
From previous corollary.
\begin{align}
\zeta(3)&=f(\pi)\\\notag
&=\frac{1}{7}\pi^2\Bigl(\frac{1}{2}+\sum_{j=1}^{\infty}\left(2-\frac{1}{2^{2(j-1)}}\right)\left(\frac{1}{(2j+1)(2j+2)} \right)\zeta(2j)\Bigr)\\\notag
&=\frac{1}{7}\pi^2\Bigl(\frac{1}{2}+\sum_{j=1}^{\infty}\left(2-\frac{2}{2^{2j-1}}\right)\left(\frac{1}{(2j+1)(2j+2)} \right)\zeta(2j)\Bigr)\\\notag
&=\frac{1}{7}\pi^2\Bigl(\frac{1}{2}+2\sum_{j=1}^{\infty}\left(1-\frac{1}{2^{2j-1}}\right)\left(\frac{1}{(2j+1)(2j+2)} \right)\zeta(2j)\Bigr)\\\notag
&=\frac{1}{7}\pi^2\Bigl(\frac{1}{2}+2\sum_{j=1}^{\infty}\left(\frac{1}{(2j+1)(2j+2)} \right)\eta(2j)\Bigr)\\\notag
&=2 f(\pi/2)\\\notag
&=\frac{2}{7}\pi^2\Bigl(\frac{3}{8}+\sum_{j=1}^{\infty}\eta(2j)\left(\frac{1}{2^{2j+1}(2j+1)} -\frac{1}{2^{2j+2}(2j+2)}\right)\Bigr)\\\notag
&=\frac{2}{7}\pi^2\Bigl(\frac{3}{8}+\sum_{j=1}^{\infty}\left(\frac{1}{2^{2j+1}}\right)\eta(2j)\left(\frac{2j+3}{(2j+1)(2j+2)}\right)\Bigr)\\\notag
&=\frac{2}{7}\pi^2\Bigl(\frac{3}{8}+\sum_{j=1}^{\infty}\left(\frac{2j+3}{2^{2j+1}(2j+1)(2j+2)}\right)\eta(2j)\Bigr)\\\notag
\end{align}
\end{proof}
\begin{corollary}
$$\sum_{j=1}^{\infty}\frac{1}{(2j)(2j+1)} \eta(2j-1)=\frac{1}{4}\notag$$
\end{corollary}
\begin{proof}
From corollaries \ref{etacor1} and \ref{etacor2}.
\end{proof}
\begin{corollary}
$$\sum_{j=1}^{\infty}\frac{1}{(2j)(2j+1)} \eta(2j)=\frac{1}{2}\left(1-\ln\left(\frac{\pi^2}{2}\right)\right)\notag$$
\end{corollary}
\begin{proof}
From corollaries \ref{etacor2} and \ref{latecor}.
\end{proof}
\begin{corollary}
$$\zeta(3)=\frac{80\pi^2}{280-3\pi^2}\Bigl(\frac{\pi^2}{144}+\frac{\ln(2)}{6}+\sum_{j=4}^{\infty}\frac{2^{j-1}-1}{2^{j-1}(j+1)(j+2)}\zeta(j)\Bigr)\notag$$
\end{corollary}
\begin{proof}
From previous corollary,
\begin{align}
\zeta(3)&=\frac{2}{7}\pi^2\sum_{j=1}^{\infty}\frac{1}{(j+1)(j+2)}\eta(j)\\\notag
&=\frac{2}{7}\pi^2\Bigl(\frac{\ln(2)}{6}+\sum_{j=2}^{\infty}\frac{1}{(j+1)(j+2)}\left(1-\frac{1}{2^{j-1}}\right)\zeta(j)\Bigr)\\\notag
&=\frac{2}{7}\pi^2\Bigl(\frac{\pi^2}{144}+\frac{\ln(2)}{6}+\frac{3}{80}\zeta(3)+\sum_{j=4}^{\infty}\frac{1}{(j+1)(j+2)}\left(1-\frac{1}{2^{j-1}}\right)\zeta(j)\Bigr)\\\notag
\end{align}
\end{proof}
\begin{corollary}
$$\zeta(3)=\frac{80\pi^2}{280+3\pi^2}\left(\frac{\pi^2}{144}-\frac{\ln(2)}{6}+\frac{1}{2}+\sum_{j=4}^{\infty}(-1)^{j}\frac{2^{j-1}-1}{2^{j-1} (j+1)(j+2)}\zeta(j)\right)$$
\end{corollary}
\begin{proof}
From previous corollary,
\begin{align}\notag
\zeta(3)&=\frac{1}{7}\int_{0}^{\pi}(x(\pi-x)) \csc( x)dx\\\notag
&=\frac{1}{7}\sum_{j=0}^{\infty}(-1)^j\left(   x((1+j)\pi-\frac{x}{2})-j(j+1)\pi^2\ln(x+j \pi)\right)|_{x=0}^{x=\pi}\\\notag
\end{align}
We take each summand as a function of $x$ of the form
$$(-1)^j\left(   x((1+j)\pi-\frac{x}{2})-j(j+1)\pi^2\ln(x+j \pi)\right)|_{x=0}^{x=x}\notag$$
and expand it as a power series to get
$$\zeta(3)=\frac{1}{7}\left(   \pi x-\frac{1}{2}x^2 +\sum_{j=1}^{\infty}(-1)^j\left(  \frac{1}{2j}x^2+\sum_{k=3}^{\infty}(-1)^k\frac{j+1}{k j^{k-1}\pi^{k-2}}x^k  \right)    \right)|_{x=\pi}$$
We then sum the terms that comprise the coefficient of $x^k$
\begin{align}\notag
\zeta(3)&=\frac{2}{7}\Bigl(\pi x-\frac{1}{2}(1+\ln(2))x^2+\frac{\pi^2+12\ln(2)}{36\pi}x^3-\frac{\pi^2+9\zeta(3)}{48\pi^2}x^4+\frac{7\pi^4+540\zeta(3)}{3600\pi^3}x^5\\\notag
&+\sum_{k=6}^{\infty}(-1)^{k+1}\frac{1}{k 2^{k-2}\pi^{k-2}}\left( 2\left(2^{k-3}-1\right)\zeta(k-2)+\left(2^{k-2}-1\right) \zeta(k-1)  \right)x^k\Bigr)|_{x=\pi}\\\notag
\end{align}
We then solve for $\zeta(3)$.
\begin{align}\notag
\zeta(3)&=\frac{80\pi^2}{280+3\pi^2}\left(\frac{\pi^2}{144}+\frac{7\pi^4}{21600}-\frac{1}{6}(\ln(2)-3)+\sum_{j=6}^{\infty}(-1)^{j+1}\frac{2^j-4}{2^j j(j+1)}\zeta(j-1)\right)\\\notag
&=\frac{80\pi^2}{280+3\pi^2}\left(\frac{\pi^2}{144}-\frac{\ln(2)}{6}+\frac{1}{2}+\sum_{j=4}^{\infty}(-1)^{j}\frac{2^{j-1}-1}{2^{j-1} (j+1)(j+2)}\zeta(j)\right)\\\notag
\end{align}
\end{proof}
\begin{corollary}
$$\zeta(3)=\frac{80}{3}\Bigl(\frac{1}{4}-\frac{\ln(2)}{6}-\sum_{j=2}^{\infty}\frac{1}{(2j+2)(2j+3)}\eta(2j+1)\Bigr)\notag$$
\end{corollary}
\begin{proof}
We have
\begin{align}\label{form1}
\zeta(3)&=\frac{80\pi^2}{280+3\pi^2}\left(\frac{\pi^2}{144}-\frac{\ln(2)}{6}+\frac{1}{2}+\sum_{j=4}^{\infty}(-1)^{j}\frac{2^{j-1}-1}{2^{j-1} (j+1)(j+2)}\zeta(j)\right)\\\label{form2}
&=\frac{80\pi^2}{280-3\pi^2}\Bigl(\frac{\pi^2}{144}+\frac{\ln(2)}{6}+\sum_{j=4}^{\infty}\frac{2^{j-1}-1}{2^{j-1}(j+1)(j+2)}\zeta(j)\Bigr) \\\notag
\end{align}
Multiply \ref{form2} by $$\frac{280-3\pi^2}{280+3\pi^2}$$
then subtract \ref{form2} from \ref{form1} and multiply the result by $$\frac{1}{1+\frac{280-3\pi^2}{280+3\pi^2}}$$
\end{proof}
Following the pattern of some of these corollaries, various infinite sum representations of $\zeta(3)$ can be found. E.g.,
\begin{align}
\zeta(3)&=\frac{2}{7}\pi^2\sum_{j=1}^{\infty}(-1)^j \frac{1}{(j+1)(j+2)}\eta(j)\\\notag
&=\frac{1}{3}\pi^2\Bigl(-\frac{1}{4}+\sum_{j=1}^{\infty}\frac{1}{(j+1)(j+2)}\zeta(2j)\Bigr)\\\notag
&=-4-3\gamma+36\ln(A)-6\sum_{j=2}^{\infty}\frac{1}{(j+1)(j+2)}\zeta(2j+1)\\\notag
&=-\frac{5}{18}\pi^2-\frac{10}{3}\gamma+40\ln(A)-20\sum_{j=4}^{\infty}\frac{1}{(j+1)(j+2)}\zeta(j)\\\notag
\end{align}
where $\gamma$ is the Euler gamma constant and $A$ is the Glaisher constant.
\begin{proposition}
\begin{align}
\int_{x=0}^{x}\frac{1}{7}x(\pi-x)\csc(x)dx&=\frac{\zeta(3)}{2}+\frac{1}{7}\Bigl(\frac{\pi^2}{4}\left(x-\frac{\pi}{2}\right)\\\notag
&+\sum_{j=1}^{\infty}(-1)^j\frac{1}{2j+1}\left(\frac{\pi^2}{4(2j)!}E_{2j}+\frac{1}{(2j-2)!}E_{2j-2}\right)(x-\frac{\pi}{2})^{2j+1}\Bigr)\\\notag
\end{align}
\end{proposition}
\begin{proof}
From lemma \ref{csclemma2}
\begin{align}
\frac{1}{7}x(\pi-x)\csc(x)&=\frac{1}{7}x(\pi-x)\sum_{j=0}^{\infty}(-1)^j\frac{1}{(2j)!}E_{2j}(x-\frac{\pi}{2})^{2j}\\\notag
&=\frac{1}{7}\left(\frac{\pi^2}{4}-\left(x-\frac{\pi}{2}\right)^2\right)\sum_{j=0}^{\infty}(-1)^j\frac{1}{(2j)!}E_{2j}(x-\frac{\pi}{2})^{2j}\\\notag
&=\frac{1}{7}\Bigl(\frac{\pi^2}{4}+\sum_{j=1}^{\infty}(-1)^j\left(\frac{\pi^2}{4(2j)!}E_{2j}+\frac{1}{(2j-2)!}E_{2j-2}\right)(x-\frac{\pi}{2})^{2j}\Bigr)\\\notag
\end{align}
$\implies$
\begin{align}
\int_{x=0}^{x}\frac{1}{7}x(\pi-x)\csc(x)dx&=\frac{1}{7}\Bigl(\frac{\pi^3}{8}+\frac{\pi^2}{4}\left(x-\frac{\pi}{2}\right)\\\notag
&+\sum_{j=1}^{\infty}(-1)^j\frac{1}{2j+1}\left(\frac{\pi^2}{4(2j)!}E_{2j}+\frac{1}{(2j-2)!}E_{2j-2}\right)(x-\frac{\pi}{2})^{2j+1}\Bigr)\\\notag
&+\frac{1}{7}\Bigl(\sum_{j=1}^{\infty}(-1)^{j+1}\frac{1}{2j+1}\left(\frac{\pi^2}{4(2j)!}E_{2j}+\frac{1}{(2j-2)!}E_{2j-2}\right)(\frac{\pi}{2})^{2j+1}\Bigr)\\\notag
\end{align}
But we know that
$$\int_{x=0}^{\frac{\pi}{2}}\frac{1}{7}x(\pi-x)\csc(x)dx=\frac{\zeta(3)}{2}$$\notag Hence
\begin{align}
\int_{x=0}^{x}\frac{1}{7}x(\pi-x)\csc(x)dx&=\frac{\zeta(3)}{2}+\frac{1}{7}\Bigl(\frac{\pi^2}{4}\left(x-\frac{\pi}{2}\right)\\\notag
&+\sum_{j=1}^{\infty}(-1)^j\frac{1}{2j+1}\left(\frac{\pi^2}{4(2j)!}E_{2j}+\frac{1}{(2j-2)!}E_{2j-2}\right)(x-\frac{\pi}{2})^{2j+1}\Bigr)\\\notag
\end{align}
\end{proof}

\begin{corollary}
$$\zeta(3)=\frac{2}{7}\Bigl(\frac{\pi^3}{8}+\sum_{j=1}^{\infty}(-1)^j\frac{\pi^{2j+1}}{(2j+1)2^{2j+1}}\left(\frac{\pi^2}{4(2j)!}E_{2j}+\frac{1}{(2j-2)!}E_{2j-2}\right)\Bigr)$$
\end{corollary}
\begin{proof}
From proof of previous proposition.
\end{proof}
\begin{corollary}
$$\zeta(3)=\frac{4}{7}\int_{0}^{\pi/2}x\ln\left(\frac{1+\sin(x)}{\cos(x)}\right)dx$$
\end{corollary}
\begin{proof}
Compare the series expansion of the integrand with the previous corollary.
\end{proof}
\begin{corollary}\label{sumpi2}
Let
\begin{align}
f(x,n)&=\frac{\zeta(3)}{2}+\frac{1}{7}\Bigl(\frac{\pi^2}{4}\left(x-\frac{\pi}{2}\right)\\\notag
&+\sum_{j=1}^{n}(-1)^j\frac{1}{2j+1}\left(\frac{\pi^2}{4(2j)!}E_{2j}+\frac{1}{(2j-2)!}E_{2j-2}\right)(x-\frac{\pi}{2})^{2j+1}\Bigr)\\\notag
\end{align}
and set
$$a_n=\frac{1}{n!}\frac{d^{n-1}}{dx^{n-1}}\Bigl(\frac{\left(x-\frac{\pi}{2}\right)^n}{\left(f(x,n)-\frac{\zeta(3)}{2}\right)^n}\Bigr)|_{x=\frac{\pi}{2}}$$
Then if
\begin{align}\notag
g(x)&=\frac{\pi}{2}+\sum_{j=1}^{\infty}a_nx^n\\\notag
&=\frac{\pi}{2}+\frac{28 }{\pi ^2}\left(x-\frac{\zeta (3)}{2}\right)-\frac{10976 \left(\pi ^2-8\right) }{3 \pi ^8}\left(x-\frac{\zeta (3)}{2}\right)^3+\frac{2151296 \left(640-112 \pi ^2+5 \pi ^4\right)}{15 \pi ^{14}} \left(x-\frac{\zeta (3)}{2}\right)^5+\cdots\\\notag
\end{align}
then 
$$g(\zeta(3))=\pi$$
\end{corollary}
\begin{proof}
Apply series reversion to the function in the previous proposition. See [11].
\end{proof}
\begin{proposition}\label{extprop}
$$\zeta(3)=\frac{4}{7}\left(\frac{1}{2}\pi^2\ln(2)+4\int_{0}^{\pi/2}x \ln(\sin(x))dx\right)\notag$$
\end{proposition}
\begin{proof}
Formula appears in reference [3].
\end{proof}
\begin{lemma}
\begin{align}\notag
\int_{0}^{\pi/2}\ln(\sin(x)) dx=-\frac{1}{2}\pi\ln(2)\\\notag
\end{align}
\end{lemma}
\begin{corollary}\label{lnsincor}
\begin{align}\notag
\zeta(3)&=\frac{4}{7}\int_{0}^{\pi/2}(4x-\pi)\ln(\sin(x))dx\notag
\end{align}
\end{corollary}
\begin{proof}
From the above proposition and lemma,

\end{proof}
\section{Formulas Involving $\frac{(n!)^2}{(n+j)!(n-j)!}$}
\begin{lemma}\label{eulersinlemma}
$$\sin(x)=x\prod_{j=1}^{\infty}\left(1-\frac{x^2}{j^2\pi^2}\right)$$
\end{lemma}
\begin{proof}
(Euler)
\end{proof}
\begin{corollary}
\begin{align}
\frac{x}{\sin(x)}&=\pi\lim_{n->\infty}\bigl[ \sum_{j=1}^{n}(-1)^{j}\left(\frac{ j (n!)^2}{(n+j)!(n-j!}\right)\left(\frac{1}{x-j\pi}-\frac{1}{x+j\pi}\right)  \bigr]\\ \notag
\end{align}
\end{corollary}
\begin{proof}
By the previous Lemma, followed by partial fraction decomposition,
\begin{align}\notag
\frac{x}{\sin(x)}&=\frac{1}{\prod_{j=1}^{\infty}\left(1-\frac{x^2}{j^2\pi^2}\right)}\\\notag
&=\lim_{n->\infty}\prod_{j=1}^{n}\left(\frac{1}{1-\frac{x^2}{j^2\pi^2}}\right)\\\notag
&=\lim_{n->\infty}\prod_{j=1}^{n}\left(\frac{j^2\pi^2}{j^2\pi^2-x^2}\right)\\\notag
&=\lim_{n->\infty}\prod_{j=1}^{n}\left(\frac{j^2\pi^2}{(j\pi-x)(j\pi+x)}\right)\\\notag
&=\pi\lim_{n->\infty}\bigl[ \sum_{k=0}^{n-1}(-1)^{k+n}\left(\frac{  (n!)^2(n-k)}{(2n-k)!k!}\right)\left(\frac{1}{x-(n-k)\pi}-\frac{1}{x+(n-k)\pi}\right)  \bigr]\\ \notag
&=\pi\lim_{n->\infty}\bigl[ \sum_{j=1}^{n}(-1)^{j}\left(\frac{ j (n!)^2}{(n+j)!(n-j)!}\right)\left(\frac{1}{x-j\pi}-\frac{1}{x+j\pi}\right)  \bigr]\\ \notag
\end{align}
\end{proof}

\begin{corollary}\label{corafterpartial}
\begin{align}\notag
\zeta(3)&=\frac{2}{7}\pi^2\lim_{n\to\infty}\sum_{j=1}^{n}(-1)^{j}\frac{j(n!)^2}{(n-j)!(n+j)!}\ln\left(    \left(\frac{2j-1}{2j+1}\right)     \left(  \frac{4j^2}{4j^2-1}\right)^{j}\right)\\\notag
&=\frac{2}{7}\pi^2\lim_{n\to\infty}\sum_{j=1}^{n}(-1)^{1-j}\frac{j(n!)^2}{(n-j)!(n+j)!}\bigl(j\ln\left(1-\frac{1}{4j^2}\right)-\ln\left(1-\frac{2}{2j+1}\right)\bigr)\\\notag
&=\frac{4}{7}\pi^2\lim_{n\to\infty}\sum_{j=1}^{n-1}(-1)^{k}\frac{(n!)^2j^2}{(n-j)!(n+j)!}\ln(j\pi)\\\notag
&=\frac{1}{7}\pi^2\Bigl(2\ln\left(2\right)+\sum_{k=1}^{\infty}\frac{1}{k(2k-1)}\left(\eta(2k-2)-1\right)\Bigr)\\\notag
\end{align}
\end{corollary}
\begin{proof}
\begin{align}\notag
\zeta(3)&=\frac{2}{7}\int_{0}^{\pi/2}((\pi-x)) \frac{x}{\sin(x)}dx\\\notag
&=\frac{2}{7}\pi\int_{0}^{\pi/2}\lim_{n->\infty}\bigl[ \sum_{k=0}^{n-1}(-1)^{k+n+1}\left(\frac{  (n!)^2(n-k)}{(2n-k)!k!}\right)\left(\frac{\pi-x}{(n-k)\pi-x}+\frac{\pi-x}{(n-k)\pi+x}\right)  \bigr] dx\\\notag
&=\frac{2}{7}\pi\lim_{n->\infty}\bigl[ \sum_{k=0}^{n-1}(-1)^{k+n+1}\int_{0}^{\pi/2}\left(\frac{  (n!)^2(n-k)}{(2n-k)!k!}\right)\left(\frac{\pi-x}{(n-k)\pi-x}+\frac{\pi-x}{(n-k)\pi+x}\right)  \bigr] dx\\\notag
&=\frac{2}{7}\pi^2\lim_{n\to\infty}\sum_{k=0}^{n-1}(-1)^{n+k+1}\frac{(n-k)(n!)^2}{k!(2n-k)!}\left( (n-k-1)\ln((n-k)\pi-x)+ (n-k+1)\ln((n-k)\pi+x)  \right)|_{x=0}^{x=\pi/2}\\\notag
&=\frac{2}{7}\pi^2\lim_{n\to\infty}\sum_{k=0}^{n-1}(-1)^{n+k+1}\frac{(n-k)(n!)^2}{k!(2n-k)!}\ln\left(    \left(  \frac{4(n-k)^2-1}{4(n-k)^2}\right)^{n-k}\left(\frac{2(n-k)+1}{2(n-k)-1}\right)     \right)\\\notag
&=\frac{2}{7}\pi^2\lim_{n\to\infty}\sum_{j=1}^{n}(-1)^{1-j}\frac{j(n!)^2}{(n-j)!(n+j)!}\ln\Bigl(    \left(  \frac{4j^2-1}{4j^2}\right)^{j}\left(\frac{2j+1}{2j-1}\right)     \Bigr)\\\notag
&=\frac{2}{7}\pi^2\lim_{n\to\infty}\sum_{j=1}^{n}(-1)^{1-j}\frac{j(n!)^2}{(n-j)!(n+j)!}\Bigl(j\ln\left(1-\frac{1}{4j^2}\right)+\ln\left(1+\frac{2}{2j-1}\right)\Bigr)\\\notag
&=\frac{2}{7}\pi^2\lim_{n\to\infty}\sum_{j=1}^{n}(-1)^{1-j}\frac{j(n!)^2}{(n-j)!(n+j)!}\Bigl(j\ln\left(1-\frac{1}{4j^2}\right)-\ln\left(1-\frac{2}{2j+1}\right)\Bigr)\\\notag
&=\frac{2}{7}\pi^2\lim_{n\to\infty}\sum_{j=1}^{n}\sum_{k=1}^{n}(-1)^{j-1}\frac{j(n!)^2}{k(n-j)!(n+j)!}\Bigl(\left(\frac{2}{1+2j}\right)^k-\left(\frac{1}{4j^2}\right)^k\Bigr)\\\notag
&=\frac{2}{7}\pi^2\lim_{n\to\infty}\sum_{j=1}^{n}(-1)^{1-j}\frac{j(n!)^2}{(n-j)!(n+j)!}\ln\Bigl(    \left(1-\frac{1}{4j^2}\right)^{j}\left(\frac{1+\frac{1}{2j}}{1-\frac{1}{2j}}\right)     \Bigr)\\\notag
&=\frac{2}{7}\pi^2\lim_{n\to\infty}\sum_{j=1}^{n}(-1)^{1-j}\frac{j(n!)^2}{(n-j)!(n+j)!}\Bigl(j\ln\Bigl(  1-\frac{1}{4j^2}\Bigr)+\ln\left(\frac{1+\frac{1}{2j}}{1-\frac{1}{2j}}\right) \Bigr)\\\notag
&=\frac{2}{7}\pi^2\lim_{n\to\infty}\sum_{j=1}^{n}(-1)^{1-j}\frac{j(n!)^2}{(n-j)!(n+j)!}\Bigl(-\frac{1}{2}\sum_{k=1}^{\infty}\frac{1}{k(2j)^{2k-1}}+2\sum_{k=1}^{\infty}\frac{1}{(2k-1)(2j)^{2k-1}}\Bigr)\\\notag
&=\frac{2}{7}\pi^2\lim_{n\to\infty}\sum_{j=1}^{n}(-1)^{1-j}\frac{(n!)^2}{2(n-j)!(n+j)!}\Bigl(-\frac{1}{2}\sum_{k=1}^{\infty}\frac{1}{k(2j)^{2k-2}}+2\sum_{k=1}^{\infty}\frac{1}{(2k-1)(2j)^{2k-2}}\Bigr)\\\notag
&=\frac{1}{14}\pi^2\lim_{n\to\infty}\sum_{j=1}^{n}(-1)^{1-j}\frac{(n!)^2}{(n-j)!(n+j)!}\sum_{k=1}^{\infty}\frac{2k+1}{k(2k-1)(2j)^{2k-2}}\\\notag
&=\frac{1}{14}\pi^2\lim_{n\to\infty}\sum_{k=1}^{\infty}\frac{2k+1}{k(2k-1)}\sum_{j=1}^{n}(-1)^{j-1}\frac{(n!)^2}{(n-j)!(n+j)!}\frac{1}{(2j)^{2k-2}}\\\notag
&=\frac{1}{14}\pi^2\sum_{k=1}^{\infty}\frac{2k+1}{k(2k-1)2^{2k-2}}\eta(2k-2)\\\notag
\end{align}
For the last step, see corollary \ref{etacor} below.
If we write down the series expansion of the function in line 4 above, about $\pi/2$, then the constant term is
\begin{align}
\frac{4}{7}\pi^2\lim_{n\to\infty}\sum_{j=1}^{n-1}(-1)^{n+j}\frac{(n!)^2(n-j)^2}{j!(2n-j)!}\ln((n-j)\pi)\\\notag
=\frac{4}{7}\pi^2\lim_{n\to\infty}\sum_{k=1}^{n-1}(-1)^{k}\frac{(n!)^2k^2}{(n-k)!(n+k)!}\ln(k\pi)\\\notag
\end{align}
\end{proof}
\begin{corollary}
\begin{align}\notag
\zeta(3)&=\frac{1}{7}\pi^2\Bigl(2\ln{2}+\lim_{n\to\infty}\sum_{j=2}^{n}(-1)^{j}\frac{j(n!)^2}{(n-j)!(n+j)!}\ln\left(    \left(\frac{j-1}{j+1}\right)     \left(  \frac{j^2}{j^2-1}\right)^{j}\right)\Bigr)\\\notag
&=\frac{1}{7}\pi^2\Bigl(2\ln{2}+\lim_{n\to\infty}\sum_{j=1}^{n}(-1)^{j-1}\frac{j(n!)^2}{(n-j)!(n+j)!}\bigl(j\ln\left(1-\frac{1}{j^2}\right)-\ln\left(1-\frac{2}{j+1}\right)\bigr)\Bigr)\\\notag
&=\frac{1}{7}\pi^2\Bigl(2\ln\left(2\right)+\sum_{k=1}^{\infty}\frac{1}{k(2k-1)}\left(\eta(2k-2)-1\right)\Bigr)\\\notag
\end{align}
Note: The last formula provides accuracy to 7 decimal places with only 10 terms.
\end{corollary}
\begin{proof}
\begin{align}\notag
\zeta(3)&=\frac{1}{7}\int_{0}^{\pi}((\pi-x)) \frac{x}{\sin(x)}dx\\\notag
&=\frac{1}{7}\pi\int_{0}^{\pi}\lim_{n->\infty}\bigl[ \sum_{k=0}^{n-1}(-1)^{k+n+1}\left(\frac{  (n!)^2(n-k)}{(2n-k)!k!}\right)\left(\frac{\pi-x}{(n-k)\pi-x}+\frac{\pi-x}{(n-k)\pi+x}\right)  \bigr] dx\\\notag
&=\frac{1}{7}\pi\lim_{n->\infty}\bigl[ \sum_{k=0}^{n-1}(-1)^{k+n+1}\int_{0}^{\pi}\left(\frac{  (n!)^2(n-k)}{(2n-k)!k!}\right)\left(\frac{\pi-x}{(n-k)\pi-x}+\frac{\pi-x}{(n-k)\pi+x}\right)  \bigr] dx\\\notag
&=\frac{1}{7}\pi^2\lim_{n\to\infty}\sum_{k=0}^{n-1}(-1)^{n+k+1}\frac{(n-k)(n!)^2}{k!(2n-k)!}\left( (n-k-1)\ln((n-k)\pi-x)+ (n-k+1)\ln((n-k)\pi+x)  \right)|_{x=0}^{x=\pi}\\\notag
&=\frac{1}{7}\pi^2\lim_{n\to\infty}\Bigl(\frac{2 n}{n+1}\ln\left(\frac{\pi+x}{\pi}\right)+\sum_{k=0}^{n-2}(-1)^{n+k+1}\frac{(n-k)(n!)^2}{k!(2n-k)!}\ln\left(    \left(  \frac{(n-k)^2\pi^2-x^2}{(n-k)^2\pi^2}\right)^{n-k}\left(\frac{(n-k)\pi+x}{(n-k)\pi-x}\right)     \right)\Bigr)|_{x=\pi}\\\notag
&=\frac{1}{7}\pi^2\Bigl(2\ln\left(\frac{\pi+x}{\pi}\right)+\lim_{n\to\infty}\sum_{j=2}^{n}(-1)^{j-1}\frac{j(n!)^2}{(n-j)!(n+j)!}\ln\left(    \left(  \frac{j^2\pi^2-x^2}{j^2\pi^2}\right)^{j}\left(\frac{j\pi+x}{j\pi-x}\right)     \right)\Bigr)|_{x=\pi}\\\notag
&=\frac{1}{7}\pi^2\Bigl(2\ln\left(2\right)+\lim_{n\to\infty}\sum_{j=2}^{n}(-1)^{1-j}\frac{j(n!)^2}{(n-j)!(n+j)!}\ln\Bigl(    \left(  \frac{j^2-1}{j^2}\right)^{j}\left(\frac{j+1}{j-1}\right)     \Bigr)\Bigr)\\\notag
&=\frac{1}{7}\pi^2\Bigl(2\ln\left(2\right)+\lim_{n\to\infty}\sum_{j=2}^{n}(-1)^{1-j}\frac{j(n!)^2}{(n-j)!(n+j)!}\Bigl(j\ln\left(1-\frac{1}{j^2}\right)+\ln\left(1+\frac{2}{j-1}\right)\Bigr)\Bigr)\\\notag
&=\frac{1}{7}\pi^2\Bigl(2\ln\left(2\right)+\lim_{n\to\infty}\sum_{j=2}^{n}(-1)^{1-j}\frac{j(n!)^2}{(n-j)!(n+j)!}\Bigl(j\ln\left(1-\frac{1}{j^2}\right)-\ln\left(1-\frac{2}{j+1}\right)\Bigr)\Bigr)\\\notag
&=\frac{1}{7}\pi^2\Bigl(2\ln\left(2\right)+\lim_{n\to\infty}\sum_{j=2}^{n}(-1)^{1-j}\frac{j(n!)^2}{(n-j)!(n+j)!}\ln\Bigl(    \left(1-\frac{1}{j^2}\right)^{j}\left(\frac{1+\frac{1}{j}}{1-\frac{1}{j}}\right)     \Bigr)\Bigr)\\\notag
&=\frac{1}{7}\pi^2\Bigl(2\ln\left(2\right)+\lim_{n\to\infty}\sum_{j=2}^{n}(-1)^{1-j}\frac{j(n!)^2}{(n-j)!(n+j)!}\Bigl(j\ln\Bigl(  1-\frac{1}{j^2}\Bigr)+\ln\left(\frac{1+\frac{1}{j}}{1-\frac{1}{j}}\right) \Bigr)\Bigr)\\\notag
&=\frac{1}{7}\pi^2\Bigl(2\ln\left(2\right)+\lim_{n\to\infty}\sum_{j=2}^{n}(-1)^{1-j}\frac{j(n!)^2}{(n-j)!(n+j)!}\Bigl(-\sum_{k=1}^{\infty}\frac{1}{k(j)^{2k-1}}+2\sum_{k=1}^{\infty}\frac{1}{(2k-1)(j)^{2k-1}}\Bigr)\Bigr)\\\notag
&=\frac{1}{7}\pi^2\Bigl(2\ln\left(2\right)+\lim_{n\to\infty}\sum_{j=2}^{n}(-1)^{j-1}\frac{(n!)^2}{(n-j)!(n+j)!}\sum_{k=1}^{\infty}\frac{1}{k(2k-1)(j)^{2k-2}}\Bigr)\\\notag
&=\frac{1}{7}\pi^2\Bigl(2\ln\left(2\right)+\lim_{n\to\infty}\sum_{j=2}^{n}(-1)^{j-1}\frac{(n!)^2}{(n-j)!(n+j)!}\sum_{k=1}^{\infty}\frac{1}{k(2k-1)(j)^{2k-2}}\Bigr)\\\notag
&=\frac{1}{7}\pi^2\Bigl(2\ln\left(2\right)+\lim_{n\to\infty}\sum_{k=1}^{\infty}\frac{1}{k(2k-1)}\sum_{j=2}^{n}(-1)^{j-1}\frac{(n!)^2}{(n-j)!(n+j)!}\frac{1}{(j)^{2k-2}}\Bigr)\\\notag
&=\frac{1}{7}\pi^2\Bigl(2\ln\left(2\right)+\sum_{k=1}^{\infty}\frac{1}{k(2k-1)}\left(\eta(2k-2)-1\right)\Bigr)\\\notag
\end{align}
For the last step, see corollary \ref{etacor} below.
\end{proof}
\begin{corollary}
Let
$$f(x)=\frac{1}{7}\int_{0}^{x}(x(\pi-x)) \csc( x)dx\notag$$
then
\begin{align}\label{formula2}
f(x)&=\frac{1}{7}\pi^2 \sum_{j=1}^{\infty}\lim_{n\to\infty}2\sum_{i=1}^{n}(-1)^{i-1}\frac{(n!)^2}{i^{2(j-1)}(n-i)!(n+i)!\pi^{2j}}\Bigl(- \frac{1}{2j }x^{2j}+\frac{\pi}{2j-1}x^{2j-1} \Bigr)\\\notag
\end{align}
\end{corollary}
\begin{proof}
From corollary
\begin{align}
f(x)&=\frac{1}{7}\pi^2\lim_{n\to\infty}\sum_{k=0}^{n-1}(-1)^{n+k+1}\frac{(n-k)(n!)^2}{k!(2n-k)!}\left( (n-k-1)\ln((n-k)\pi-x)+ (n-k+1)\ln((n-k)\pi+x)  \right)|_{x=0}^{x}\\\notag
&=\frac{1}{7}\pi^2\lim_{n\to\infty}\sum_{k=0}^{n-1}(-1)^{n+k+1}\frac{(n-k)(n!)^2}{k!(2n-k)!}\left( (n-k-1)\ln(\frac{(n-k)\pi-x)}{(n-k)\pi}+ (n-k+1)\ln(\frac{(n-k)\pi+x) }{(n-k)\pi} \right)\\\notag
&=\frac{1}{7}\pi^2\lim_{n\to\infty}\sum_{k=0}^{n-1}(-1)^{n+k+1}\frac{(n-k)(n!)^2}{k!(2n-k)!}\Bigl( (n-k-1)\ln\left(1-\frac{x}{(n-k)\pi}\right)+ (n-k+1)\ln\left(1+\frac{x}{(n-k)\pi}\right) \Bigr)\\\notag
&=\frac{1}{7}\pi^2\lim_{n\to\infty}\sum_{k=0}^{n-1}(-1)^{n+k+1}\frac{(n-k)(n!)^2}{k!(2n-k)!}\Bigl(- (n-k-1)\sum_{j=1}^{\infty}\frac{1}{j}\left(\frac{x}{(n-k)\pi}\right)^j+ (n-k+1)\sum_{j=1}^{\infty}(-1)^{j-1}\frac{1}{j}\left(\frac{x}{(n-k)\pi}\right)^j \Bigr)\\\notag
&=\frac{1}{7}\pi^2\lim_{n\to\infty}\sum_{i=1}^{n}(-1)^{i-1}\frac{i(n!)^2}{(n-i)!(n+i)!}\Bigl(- (i-1)\sum_{j=1}^{\infty}\frac{1}{j}\left(\frac{x}{i\pi}\right)^j+( i+1)\sum_{j=1}^{\infty}(-1)^{j-1}\frac{1}{j}\left(\frac{x}{i\pi}\right)^j \Bigr)\\\notag
&=\frac{1}{7}\pi^2\lim_{n\to\infty}\sum_{i=1}^{n}(-1)^{i-1}\frac{i(n!)^2}{(n-i)!(n+i)!}\Bigl(- 2i\sum_{j=1}^{\infty}\frac{1}{2j}\left(\frac{x}{i\pi}\right)^{2j}+2\sum_{j=1}^{\infty}\frac{1}{2j-1}\left(\frac{x}{i\pi}\right)^{2 j-1} \Bigr)\\\notag
&=\frac{1}{7}\pi^2 \sum_{j=1}^{\infty}\lim_{n\to\infty}2\sum_{i=1}^{n}(-1)^{i-1}\frac{i(n!)^2}{(n-i)!(n+i)!}\Bigl(- i\frac{1}{2j}\left(\frac{1}{i\pi}\right)^{2j}x^{2j}+\frac{1}{2j-1}\left(\frac{1}{i\pi}\right)^{2 j-1}x^{2j-1} \Bigr)\\\notag
&=\frac{1}{7}\pi^2 \sum_{j=1}^{\infty}\lim_{n\to\infty}2\sum_{i=1}^{n}(-1)^{i-1}\frac{i(n!)^2}{(n-i)!(n+i)!}\Bigl(- \frac{1}{2j i^{2j-1}\pi^{2j}}x^{2j}+\frac{1}{(2j-1)i^{2j-1}\pi^{2j-1}}x^{2j-1} \Bigr)\\\notag
&=\frac{1}{7}\pi^2 \sum_{j=1}^{\infty}\lim_{n\to\infty}2\sum_{i=1}^{n}(-1)^{i-1}\frac{(n!)^2}{i^{2(j-1)}(n-i)!(n+i)!\pi^{2j}}\Bigl(- \frac{1}{2j }x^{2j}+\frac{\pi}{(2j-1)}x^{2j-1} \Bigr)\\\notag
\end{align}
\end{proof}
\begin{corollary}\label{etacor}
\begin{align}
\lim_{n\to\infty}\sum_{i=1}^{n}(-1)^{i-1}\frac{(n!)^2}{(n-i)!(n+i)!}=\frac{1}{2}\\\notag
\eta(2j)=\lim_{n\to\infty}\sum_{i=1}^{n}(-1)^{i-1}\frac{(n!)^2}{i^{2j}(n-i)!(n+i)!}\\\notag
\end{align}
\end{corollary}
\begin{proof}
Comparing \ref{formula1} and \ref{formula2},

\begin{align}
\pi x-\frac{x^2}{2}=\lim_{n\to\infty}2\pi^2\sum_{i=1}^{n}(-1)^{i-1}\frac{(n!)^2}{(n-i)!(n+i)!\pi^{2}}\Bigl(- \frac{1}{2 }x^{2}+\pi x \Bigr)\\\notag
\end{align}
$\implies$
\begin{align}
\lim_{n\to\infty}\sum_{i=1}^{n}(-1)^{i-1}\frac{(n!)^2}{(n-i)!(n+i)!}=\frac{1}{2}\\\notag
\end{align}
and
\begin{align}
\lim_{n\to\infty}\sum_{i=1}^{n}(-1)^{i-1}\frac{(n!)^2}{i^{2j}(n-i)!(n+i)!\pi^{2j}}\left(\frac{\pi x^{2j+1}}{2j+1} -\frac{x^{2j+2}}{2j+2}\right)\\\notag
=\frac{1}{\pi^{2j}}\eta(2j)\left(\frac{\pi x^{2j+1}}{2j+1} -\frac{x^{2j+2}}{2j+2}\right)\\\notag
\end{align}
$\implies$
\begin{align}
\eta(2j)=\lim_{n\to\infty}\sum_{i=1}^{n}(-1)^{i-1}\frac{(n!)^2}{i^{2j}(n-i)!(n+i)!}\\\notag
\end{align}
\end{proof}
\begin{corollary}\label{etacor}
\begin{align}
\eta(k)=\lim_{n\to\infty}\sum_{i=1}^{n}(-1)^{i-1}\frac{(n!)^2}{i^{k}(n-i)!(n+i)!}\\\notag
\end{align}
\end{corollary}
\begin{proof}
\end{proof}
\begin{corollary}
$$B_{2j}=\frac{(2j)!}{(2^{2j-1}-1)\pi^{2j}}\lim_{n\to\infty}\sum_{i=1}^{n}(-1)^{i+j}\frac{(n!)^2}{i^{2j}(n-i)!(n+i)!}$$\notag
\end{corollary}
\begin{proof}
\begin{align}
\eta(2j)&=(\frac{2^{2j-1}-1}{2^{2j-1}})\zeta(2j)\\\notag
&=(-1)^{j+1}(\frac{(2^{2j-1}-1)(\pi)^{2j}}{(2j)!})B_{2j}\\\notag
\end{align}
$\implies$
\begin{align}
B_{2j}&=(-1)^{j+1}\frac{(2j)!}{(2^{2j-1}-1)\pi^{2j}}\eta(2j)\\\notag
&=\frac{(2j)!}{(2^{2j-1}-1)\pi^{2j}}\lim_{n\to\infty}\sum_{i=1}^{n}(-1)^{i+j}\frac{(n!)^2}{i^{2j}(n-i)!(n+i)!}\\\notag
\end{align}
\end{proof}
\begin{lemma}\label{domcon}
$$\lim_{n\to\infty}\frac{(n!)^2}{(n-j)!(n+j)!}=1\notag$$
\end{lemma}
\begin{proof}
If we let
$$a_j=\frac{(n!)^2}{(n-j)!(n+j)!}$$
then
$$a_{j+1}=\frac{n+1-j}{n+j}a_j$$
and
$$a_1=1$$
So, by induction, Assuming $\lim_{n\to\infty}a_{j-1}=1$ then $\lim_{n\to\infty}a_j=1$.
\end{proof}
\begin{corollary}\label{sumfactorials}
For $n \ge 1$
\begin{align}\notag
&\sum_{j=1}^{n}(-1)^{j-1}\frac{(n!)^2}{(n-j)!(n+j)!}=\frac{1}{2}\\\notag
&\sum_{j=1}^{n}(-1)^{j-1}\frac{(n!)^2}{(n-j)!(n+j)!}j=\frac{n}{2(2n-1)}\\\label{gos1}
&\sum_{j=1}^{n}(-1)^{j-1}\frac{(n!)^2}{(n-j)!(n+j)!}j^3=\frac{n^2}{2(2n-1)(2n-3)}\\\label{gos3}
&\sum_{j=1}^{n}(-1)^{j-1}\frac{(n!)^2}{(n-j)!(n+j)!}j^5=-\frac{n^2(4n-1)}{2(2n-1)(2n-3)(2n-5)}\\\label{gos4}
&\sum_{j=1}^{n}(-1)^{j-1}\frac{(n!)^2}{(n-j)!(n+j)!}j^7=\frac{n^2(34n^2-24n+5)}{2(2n-1)(2n-3)(2n-5)(2n-7)}\\\label{gos4}
\end{align}
For $n>m, m\ge1$
\begin{align}
\sum_{j=1}^{n}(-1)^j\frac{(n!)^2}{(n-j)!(n+j)!}j^{2m}=0\\\notag
\end{align}
\end{corollary}
\begin{proof}
Proof by Gosper algorithm. See references [8], [9], [10].
\end{proof}
Note the lemma makes it clear why formulas such as
\begin{align}
\zeta(j)&=\lim_{n\to\infty}\sum_{i=1}^{n}\frac{(n!)^2}{(n-i)!(n+i)!}\frac{1}{i^j}\\\notag
\end{align}
are trivially true by the dominated conjvergence theorem. However, similar sums with alternating sign are not trivial in cases where the dominated convergence theorem does not apply. E.g., from corollary \ref{etacor}
$$\lim_{n\to\infty}\sum_{j=1}^{n}(-1)^{j-1}\frac{(n!)^2}{(n-j)!(n+j)!}=\frac{1}{2}$$
As expected,
\begin{corollary}With the usual restrictions on $x$
\begin{align}
\lim_{n\to\infty}\sum_{i=1}^{n}\frac{(n!)^2}{(n-i)!(n+i)!}\frac{1}{i^j}&=\zeta(j)\\\notag
\lim_{n\to\infty}\sum_{i=1}^{n}(-1)^{i-1}\frac{(n!)^2}{(n-i)!(n+i)!}\frac{1}{i^j}&=\eta(j)\\\notag
\lim_{n\to\infty}\sum_{j=0}^{n}\frac{(n!)^2}{(n-j)!(n+j)!}\left(\frac{x}{2}\right)^j&\notag=\frac{1}{1-\frac{x}{2}}\\\notag
\lim_{n\to\infty}\sum_{j=1}^{n}(-1)^{j-1}\frac{(n!)^2}{(n-j)!(n+j)!}\frac{1}{x^j}&=\frac{1}{x+1}\notag\\\notag
\lim_{n\to\infty}\sum_{j=1}^{n}(-1)^{j-1}\frac{(n!)^2}{(n-j)!(n+j)!}\frac{1}{j}x^j&=\log\left(x+1\right)\notag\\\notag
\lim_{n\to\infty}\sum_{j=1}^{n}(-1)^{j-1}\frac{(n!)^2}{(n-j)!(n+j)!}\frac{1}{jx^j}&=\log\left(\frac{x+1}{x}\right)\notag\\\notag
\lim_{n\to\infty}\sum_{j=1}^{n}(-1)^{j-1}\frac{(n!)^2}{(2j-1)!(n-j)!(n+j)!}x^{2j-1}&=\sin(x)\notag\\\notag
\end{align}
\end{corollary}
\begin{proof}
Dominated convergence theorem together with Lemma \ref{domcon}.
\end{proof}
\begin{corollary}
$$\zeta(3)=\frac{4}{7}\pi^2\lim_{n\to\infty}\sum_{j=1}^{n-1}(-1)^{k}\frac{(n!)^2j^2}{(n-j)!(n+j)!}\ln(j)\notag$$
\end{corollary}
\begin{proof}
From corollary \ref{corafterpartial}
\begin{align}
\zeta(3)&=\frac{4}{7}\pi^2\lim_{n\to\infty}\sum_{j=1}^{n-1}(-1)^{j}\frac{(n!)^2j^2}{(n-j)!(n+j)!}\ln(j\pi)\\\notag
&=\frac{4}{7}\pi^2\lim_{n\to\infty}\sum_{j=1}^{n-1}(-1)^{j}\frac{(n!)^2j^2}{(n-j)!(n+j)!}(\ln(j)+\ln(\pi))\\\notag
&=\frac{4}{7}\pi^2\lim_{n\to\infty}\sum_{j=1}^{n-1}(-1)^{j}\frac{(n!)^2j^2}{(n-j)!(n+j)!}\ln(j)\\\notag
\end{align}
from corollary \ref{sumfactorials} 
\end{proof}
\begin{lemma}
\begin{align}\notag
\lim_{n\to\infty}\sum_{j=1}^{n-1}(-1)^{j}\frac{(n!)^2}{(n-j)!(n+j)!}\left(1-\frac{1}{j}\right)&=\frac{1}{4}\\\notag
\lim_{n\to\infty}\frac{1}{2}\sum_{j=1}^{n-1}(-1)^{j}\frac{(n!)^2}{(n-j)!(n+j)!}\left(1-\frac{1}{j}\right)^2&=0\\\notag
\lim_{n\to\infty}\frac{1}{3}\sum_{j=1}^{n-1}(-1)^{j}\frac{(n!)^2}{(n-j)!(n+j)!}\left(1-\frac{1}{j}\right)^3&=\frac{1}{3}\ln(2)-\frac{1}{4}\\\notag
\end{align}
\end{lemma}
\begin{proof}
Previous corollaries.
\end{proof}
\begin{lemma}
$$\frac{(n!)^2}{(n-j)!(n+j)!}=\frac{(n-j+1)\binom{n}{j-1}}{(n+1)\binom{n+j}{j-1}}$$
\end{lemma}
\begin{proof}
Let 
\begin{align}\notag
a_j&=\frac{(n!)^2}{(n-j)!(n+j)!}\\\notag
b_j&=\frac{(n-j+1)\binom{n}{j-1}}{(n+1)\binom{n+j}{j-1}}\\\notag
\end{align}
then
$$a_1=b_1=\frac{n}{n+1}$$
and
$$\frac{a_{j+1}}{a_j}=\frac{b_{j+1}}{b_j}=\frac{n-j}{n+j+1}$$
\end{proof}
\begin{corollary}
\begin{align}\notag
\zeta(3)&=\frac{4}{7}\pi^2\Bigl(\frac{\ln(2)}{3}-\sum_{k=4}^{\infty}\left(\frac{k-2}{4}+\frac{1}{k}\sum_{j=1}^{k}(-1)^j \binom{k}{j+2}\eta(j)\right)\Bigr)
\end{align}
\end{corollary}
\begin{proof}
From previous corollary,
\begin{align}\notag
\zeta(3)&=\frac{4}{7}\pi^2\lim_{n\to\infty}\sum_{j=1}^{n-1}(-1)^{j}\frac{(n!)^2j^2}{(n-j)!(n+j)!}\ln(j)\\\notag
&=\frac{4}{7}\pi^2\lim_{n\to\infty}\sum_{j=1}^{n-1}(-1)^{j}\frac{(n!)^2j^2}{(n-j)!(n+j)!}\sum_{k=1}^{n}\frac{1}{k}\left(\frac{j-1}{j}\right)^k\\\notag
&=\frac{4}{7}\pi^2\Bigl(\frac{\ln(2)}{3}+\lim_{n\to\infty}\sum_{j=1}^{n-1}(-1)^{j}\frac{(n!)^2j^2}{(n-j)!(n+j)!}\sum_{k=4}^{n}\frac{1}{k}\left(\frac{j-1}{j}\right)^k\Bigr)\\\notag
&=\frac{4}{7}\pi^2\Bigl(\frac{\ln(2)}{3}+\lim_{n\to\infty}\sum_{j=1}^{n-1}(-1)^{j}\frac{(n!)^2j^2}{(n-j)!(n+j)!}\sum_{k=4}^{n}\frac{1}{k}\left(\frac{\sum_{i=0}^{k}\binom{k}{i}j^i(-1)^{k-i}}{j^k}\right)\Bigr)\\\notag
&=\frac{4}{7}\pi^2\Bigl(\frac{\ln(2)}{3}+\lim_{n\to\infty}\sum_{j=1}^{n-1}(-1)^{j}\frac{(n!)^2}{(n-j)!(n+j)!}\sum_{k=4}^{n}\frac{1}{k}\sum_{i=0}^{k}\binom{k}{i}(-1)^{k-i}j^{i+2-k}\Bigr)\\\notag
&=\frac{4}{7}\pi^2\Bigl(\frac{\ln(2)}{3}+\lim_{n\to\infty}\sum_{k=4}^{n}\sum_{i=0}^{k}\sum_{j=1}^{n-1}(-1)^{j}\frac{(n!)^2}{(n-j)!(n+j)!}\frac{1}{k}\binom{k}{i}(-1)^{k-i}j^{i+2-k}\Bigr)\\\notag
&=\frac{4}{7}\pi^2\Bigl(\frac{\ln(2)}{3}+\lim_{n\to\infty}\sum_{k=4}^{n}\sum_{r=0}^{k}\sum_{j=1}^{n-1}(-1)^{j}\frac{(n!)^2}{(n-j)!(n+j)!}\frac{1}{k}\binom{k}{r}(-1)^{r}\frac{1}{j^{r-2}}\Bigr)\\\notag
&=\frac{4}{7}\pi^2\Bigl(\frac{\ln(2)}{3}+\lim_{n\to\infty}\sum_{k=4}^{n}\Bigl(\frac{1}{k}\binom{k}{1}\frac{n}{2(2n-1)}-\frac{1}{k}\binom{k}{2}\frac{1}{2}+\sum_{r=3}^{k-3}\sum_{j=1}^{n-1}(-1)^{j}\frac{(n!)^2}{(n-j)!(n+j)!}\frac{1}{k}\binom{k}{r}(-1)^{r}\frac{1}{j^{r-2}}\Bigr)\Bigr)\\\notag
&=\frac{4}{7}\pi^2\Bigl(\frac{\ln(2)}{3}+\lim_{n\to\infty}\sum_{k=4}^{n}\Bigl(\frac{k-2}{4}+\sum_{r=3}^{k-3}\sum_{j=1}^{n-1}(-1)^{j}\frac{(n!)^2}{(n-j)!(n+j)!}\frac{1}{k}\binom{k}{r}(-1)^{r}\frac{1}{j^{r-2}}\Bigr)\Bigr)\\\notag
&=\frac{4}{7}\pi^2\Bigl(\frac{\ln(2)}{3}-\sum_{k=4}^{\infty}\left(\frac{k-2}{4}+\frac{1}{k}\sum_{j=1}^{k}(-1)^j \binom{k}{j+2}\eta(j)\right)\Bigr)
\end{align}
From applying corollary \ref{sumfactorials} to the terms of the sum.
\end{proof}
\begin{corollary}
\begin{align}\notag
\zeta(3)&=\frac{35\pi^2}{7(35-3\pi^2)}\Bigl(-5-\frac{5\pi^2}{12}+\frac{40\ln(2)}{3}-\sum_{k=6}^{\infty}\Bigl((k-2)+\frac{4}{k}\sum_{j=1}^{k-2}(-1)^j \binom{k}{j+2}\eta(j)\Bigr)\Bigr)\\\notag
\end{align}
\end{corollary}
\begin{proof}
\begin{align}\notag
\zeta(3)&=\frac{4}{7}\pi^2\Bigl(\frac{\ln(2)}{3}-\sum_{k=4}^{\infty}\left(\frac{k-2}{4}+\frac{1}{k}\sum_{j=1}^{k}(-1)^j \binom{k}{j+2}\eta(j)\right)\Bigr)\\\notag
&=\frac{4}{7}\pi^2\Bigl(\frac{\ln(2)}{3}-\Bigl(\frac{1}{2}+\frac{3}{4}+\frac{1}{4}\sum_{j=1}^{2}(-1)^j\binom{4}{j+2}\eta[j]\\\notag
&+\frac{1}{5}\sum_{j=1}^{3}(-1)^j\binom{5}{j+2}\eta[j]\Bigr)-\sum_{k=6}^{\infty}\left(\frac{k-2}{4}+\frac{1}{k}\sum_{j=1}^{k}(-1)^j \binom{k}{j+2}\eta(j)\right)\Bigr)\\\notag
\left(1-\frac{4}{7}\pi^2\frac{3}{20}\right)\zeta(3)&=\frac{4}{7}\pi^2\Bigl(-\frac{5}{4}-\frac{5\pi^2}{48}+\frac{10\ln(2)}{20}-\sum_{k=6}^{\infty}\left(\frac{k-2}{4}+\frac{1}{k}\sum_{j=1}^{k}(-1)^j \binom{k}{j+2}\eta(j)\right)\Bigr)\\\notag
\end{align}
\end{proof}
\begin{corollary}\label{bigenergy}
\begin{align}\notag
\zeta(3)&=\lim_{n\to\infty}800\sum_{j=1}^{n-5}(-1)^{j+1}\frac{n!}{(j+5)(j+5)!(n-j)!}\eta(3+j)\\\notag
\end{align}
\end{corollary}
\begin{proof}
From previous corollary,
\begin{align}\notag
\zeta(3)&=\frac{4}{7}\pi^2\Bigl(\frac{\ln(2)}{3}-\sum_{k=4}^{\infty}\left(\frac{k-2}{4}+\frac{1}{k}\sum_{j=1}^{k}(-1)^j \binom{k}{j+2}\eta(j)\right)\Bigr)\\\notag
&=\frac{4}{7}\pi^2\Bigl(\frac{\ln(2)}{3}-\sum_{k=4}^{\infty}\Bigl(\frac{k-2}{4}+\frac{1}{k}\Bigl(- \binom{k}{3}\eta(1)\\\notag
&+ \binom{k}{4}\eta(2)- \binom{k}{5}\eta(3)+\sum_{j=4}^{k}(-1)^j \binom{k}{j+2}\eta(j)\Bigr)\Bigr)\Bigr)\\\notag
&=\frac{4}{7}\pi^2\Bigl(\frac{\ln(2)}{3}-\lim_{n\to\infty}\sum_{k=4}^{n}\Bigl(\frac{k-2}{4}+\frac{1}{k}\Bigl(- \binom{k}{3}\eta(1)\\\notag
&+ \binom{k}{4}\eta(2)- \binom{k}{5}\eta(3)+\sum_{j=4}^{k}(-1)^j \binom{k}{j+2}\eta(j)\Bigr)\Bigr)\Bigr)\\\notag
\end{align}

Compute the sums over $k$ individually, and group $\zeta(3)$ terms on the left side of the equation.
\begin{align}\notag
\zeta(3)&=\lim_{n\to\infty}\frac{800\pi^2}{1400-n(n-1)(n-2)(n-3)(n-4)\pi^2}\Bigl(-\frac{1}{8}n(n-3)+\frac{1}{18}n(n-1)(n-2)\ln(2)\\\notag
&-\frac{1}{1152}n(n-1)(n-2)(n-3)\pi^2-\sum_{j=4}^{n-2}(-1)^j\frac{n-j-1}{(n+1)(j+2)}\binom{n+1}{j+2}\eta(j)\Bigr)\\\notag
&=\lim_{n\to\infty}\frac{800\pi^2}{1400-n(n-1)(n-2)(n-3)(n-4)\pi^2}\Bigl(-\frac{1}{8}n(n-3)+\frac{1}{18}n(n-1)(n-2)\ln(2)\\\notag
&-\frac{1}{1152}n(n-1)(n-2)(n-3)\pi^2-\sum_{j=4}^{n-2}(-1)^j\frac{n-j-1}{(n+1)(j+2)}\binom{n+1}{j+2}\eta(j)\Bigr)\\\notag
&=\lim_{n\to\infty}\frac{800\pi^2}{1400-n(n-1)(n-2)(n-3)(n-4)\pi^2}\sum_{j=4}^{n-2}(-1)^{j+1}\frac{n-j-1}{(n+1)(j+2)}\binom{n+1}{j+2}\eta(j)\\\notag
&=\lim_{n\to\infty}\frac{800\pi^2}{1400-n(n-1)(n-2)(n-3)(n-4)\pi^2}\sum_{j=1}^{n-2}(-1)^{j+1}\frac{n-j-4}{(n+1)(j+5)}\binom{n+1}{j+5}\eta(3+j)\\\notag
&=\lim_{n\to\infty}\frac{800\pi^2}{1400-n(n-1)(n-2)(n-3)(n-4)\pi^2}\sum_{j=1}^{n-2}(-1)^{j+1}\frac{n-j-4}{(n+1)(j+5)}\binom{n+1}{j+5}\eta(3+j)\\\notag
&=\lim_{n\to\infty}\frac{800\pi^2}{1400-n(n-1)(n-2)(n-3)(n-4)\pi^2}\sum_{j=1}^{n-2}(-1)^{j+1}\frac{1}{j+5}\binom{n}{j+5}\eta(3+j)\\\notag
&=\lim_{n\to\infty}\frac{800n(n-1)(n-2)(n-3)(n-4)\pi^2}{1400-n(n-1)(n-2)(n-3)(n-4)\pi^2}\sum_{j=1}^{n-10}(-1)^{j+1}\frac{(n-5)!}{(j+5)(j+5)!(n-5-j)!}\eta(3+j)\\\notag
&=\lim_{n\to\infty}800\sum_{j=1}^{n-10}(-1)^{j+1}\frac{(n-5)!}{(j+5)(j+5)!(n-5-j)!}\eta(3+j)\\\notag
&=\lim_{n\to\infty}800\sum_{j=1}^{n-5}(-1)^{j+1}\frac{n!}{(j+5)(j+5)!(n-j)!}\eta(3+j)\\\notag
&=\lim_{n\to\infty}800\sum_{j=1}^{n-5}(-1)^{j+1}\frac{(n-j)^j}{(j+5)(j+5)!}\eta(3+j)\\\notag
\end{align}
\end{proof}
\section{Dynamic Binomial Sums And The Difference Operator}
Sums of the type given in corollary \ref{bigenergy} are very interesting objects. They involve individual terms of immense magnitude. For a sum of $2n$ terms, call the $j$-th term $a_j$, the terms reach peak size at around the $n$-th term, and tend to zero at the last terms. For a given $n$, the sum is a complete object as expressed. The value will not be improved by, for example, summing to $n$ instead of to $n-5$.  A little thought shows that the important information is carried in the fractional part: $a_j-\floor{a_j}$.
\begin{corollary}
Let
\begin{align}\notag
g(n)&=\sum_{j=1}^{n-5}(-1)^{j+1}\left(800\frac{n!}{(j+5)(j+5)!(n-j)!}\eta(3+j)-\floor{800\frac{n!}{(j+5)(j+5)!(n-j)!}\eta(3+j)}\right)\\\notag
\end{align}
then
\begin{align}\notag
\zeta(3)&=1+\lim_{n\to\infty}\left(g(n)-\floor{g(n)}\right)\\\notag
\end{align}
\end{corollary}
\begin{proof}
The function $floor(b)=\floor{b}$ is defined to be the greatest integer less than $b$. So if we define the fractional part of $b$ to be $b-\floor{b}$, then the fractional part of $b$ is the positive representative of $b\mod{1}$. So, adding fractional parts is just addition mod 1. One has to keep track of the sign of the final sum to interpret the result.
\end {proof}
This also points out that in any expression for $\zeta(3)$, we can add arbitrary integer values, e.g. expressions like $j!/(j-3)!$, and still preserve the result by taking the fractional part. 
\begin{proposition}

\begin{align}\notag
\forall x,y,s\in R&\\\notag
\eta(k)&=\lim_{n\to\infty}\sum_{j=1}^{n}(-1)^{j+1}\binom{n}{j}\Bigl(\frac{x^j+x^{j-1}+\cdots +x+1}{y^{j}}\Bigr)^s\eta(k+j)\\\notag
\end{align}
\end{proposition}
\begin{proof}
(Not rigorous yet.)
Corollary \ref{bigenergy} suggests we look at sums of the type
$$\sum_{j=1}^{n}(-1)^{j+1}\binom{n}{j}\eta(k+j)$$\notag

Indeed,
$$\eta(k)=\lim_{n\to\infty}\sum_{j=1}^{n}(-1)^{j+1}\binom{n}{j}\eta(k+j)$$\notag
Further investigation shows
$$\zeta(k)=\lim_{n\to\infty}\sum_{j=1}^{n}(-1)^{j+1}\binom{n}{j}\zeta(k+j)$$\notag
These equations can be written as
\begin{align}\notag
0&=\lim_{n\to\infty}\sum_{j=0}^{n}(-1)^{j+1}\binom{n}{j}\eta(k+j)\\\notag
0&=\lim_{n\to\infty}\sum_{j=0}^{n}(-1)^{j+1}\binom{n}{j}\zeta(k+j)\\\notag
\end{align}
Since 
$$\eta(k)=\frac{2^{k+j-1}-1}{2^{k+j-1}}\zeta(k)$$\notag
If we let
$$f(k,j)=\frac{2^{k+j-1}-1}{2^{k+j-1}}$$
then
\begin{align}\notag
0&=\lim_{n\to\infty}\sum_{j=0}^{n}(-1)^{j+1}\binom{n}{j}f(k,j)\zeta(k+j)\\\notag
\end{align}
What is the class of functions $f$ that satisfy the preceding equation?
Note that corollary \ref{bigenergy} is an instance of this with $f(j)=\frac{1}{(j+5)^2(j+4)(j+3)(j+2)(j+1)}$.
That class includes functions
$$f(u,v,s,m,r,j)=\Bigl(\frac{u^{m+j}-1}{v^{r+j}}\Bigr)^s$$\notag
One finds that $\forall u,v,s\in R, u\ne 1, v\ne 0,m,r,j\in Z$
\begin{align}\notag
0&=\sum_{j=0}^{n}(-1)^{j+1}(-1)^{j+1}\binom{n}{j}\Bigl(\frac{u^{m+j}-1}{v^{r+j}}\Bigr)^s\eta(k)\\\notag
&\implies\\\notag
\eta(k)&=\lim_{n\to\infty}\Bigl(\frac{v^r}{u^m-1}\Biggr)^s\sum_{j=1}^{n}(-1)^{j+1}(-1)^{j+1}\binom{n}{j}\Bigl(\frac{u^{m+j}-1}{v^{r+j}}\Bigr)^s\eta(k)\\\notag
\end{align}
Since $u, m$ are arbitrary, we can set $u=x, m=1$. 
\begin{align}\notag
\eta(k)&=\lim_{n\to\infty}\sum_{j=1}^{n}(-1)^{j+1}(-1)^{j+1}\binom{n}{j}\Bigl(\frac{x^{j+1}-1}{(x-1)y^{j}}\Bigr)^s\eta(k)\\\notag
&=\lim_{n\to\infty}\sum_{j=1}^{n}(-1)^{j+1}\binom{n}{j}\Bigl(\frac{x^j+x^{j-1}+\cdots +x+1}{y^{j}}\Bigr)^s\eta(k+j)\Bigr)\\\notag
\end{align}
As a function of $x$ or $y$, the above expression is constant so it's derivative is zero.
\begin{align}
0&=\lim_{n\to\infty}\sum_{j=1}^{n}(-1)^{j+1}\binom{n}{j}\Bigl(\frac{x^{j+1}-1}{(x-1)y^{j}}\Bigr)^{s-1}\Bigl(\frac{j x^{j-1}+(j-1)x^{j-2}+\cdots +2 x+1}{y^{j}}\Bigr)\eta(k+j)\\\notag
&=\lim_{n\to\infty}\sum_{j=1}^{n}(-1)^{j+1}\binom{n}{j}\Bigl(\Bigl(\frac{x^{j+1}-1}{(x-1)y^{j}}\Bigr)^{s-2}\Bigl(\frac{j x^{j-1}+(j-1)x^{j-2}+\cdots +2 x+1}{y^{j}}\Bigr)^2\\\notag
&+\Bigl(\frac{x^{j+1}-1}{(x-1)y^{j}}\Bigr)^{s-1}\Bigl(\frac{j(j-1) x^{j-2}+(j-1)(j-2)x^{j-2}+\cdots +3*2 x+2}{y^{j}}\Bigr)\Bigr)\eta(k+j)\\\notag
&=\rm{etc.}\\\notag
\end{align}
\end{proof}
Notation:
$S(n,m)$ denotes the Stirling number of the second kind, the number of ways of partitioning $n$ elements into $m$ non-empty subsets. H. W. Gould uses the notation, $B_{k,k}^n=n!S(n,k)$. See ref. [14].
\begin{lemma}\label{stirlingsum1}
Assume $k\ge n\ge 0$ then
$$\sum_{j=k-n+1}^{k-1}\binom{k}{j}S(k-j,n)x^j=0$$
\end{lemma}
\begin{proof}
As $j$ goes from $k-n+1$ to $k-1$, $k-j$ goes from $n-1$ to $1$. But if $a<b$ then $S(a,b)=0$.
\end{proof}
\begin{lemma}
For $0\le k\le n-1$
\begin{align}
\sum_{j=1}^n (-1)^{j+1}\binom{n}{j}(x+j)^k=x^k\\\notag
\sum_{j=1}^n (-1)^{j+1}\binom{n}{j}(x-j)^k=x^k\\\notag
\end{align}
For $k= n$
\begin{align}
&\sum_{j=1}^n (-1)^{j+1}\binom{n}{j}(x+j)^k=x^k+(-1)^{n+1}n!\\\notag
&\sum_{j=1}^n (-1)^{j+1}\binom{n}{j}(x-j)^k=x^k+(-1)^{n}n!\\\notag
\end{align}
For $k\ge n$
\begin{align}
&\sum_{j=1}^n (-1)^{j+1}\binom{n}{j}(x+j)^k=x^k+(-1)^{n+1}n!\sum_{i=0}^{k-n}\binom{k}{i}S(k-i,n)x^i\\\notag
&\sum_{j=1}^n (-1)^{j+1}\binom{n}{j}(x-j)^k=x^k+(-1)^{n}n!\sum_{i=0}^{k-n}(-1)^{i}\binom{k}{i}S(k-i,n)x^i\\\notag
\end{align}
\end{lemma}
\begin{proof}
From the theory of finite differences, for any polynomial $P(j)$, with degree $<n$,
\begin{align}
0&=\sum_{j=0}^n (-1)^{j}\binom{n}{j}P(j)\\\notag
\implies&\\\notag
P(0)&=\sum_{j=1}^n (-1)^{j+1}\binom{n}{j}P(j)\\\notag
\end{align}
Assume $k\ge n$
\begin{align}\notag
&\sum_{j=1}^n (-1)^{j+1}\binom{n}{j}(x+j)^k\\\notag
&=\sum_{j=1}^{n}(-1)^{j+1}\binom{n}{j}\sum_{i=0}^k \binom{k}{i}j^{k-i} x^{i}\\\notag
&=\sum_{i=0}^k\binom{k}{i}\Bigl( \sum_{j=1}^{n}(-1)^{j+1}\binom{n}{j}j^{k-i}\Bigr)x^{i} \\\notag
&=\Bigl(x^k+\sum_{i=0}^{k-1}\binom{k}{i}\Bigl( \sum_{j=1}^{n}(-1)^{j+1}\binom{n}{j}j^{k-i}\Bigr)x^{i}\Bigr) \\\notag
&=\Bigl(x^k+\sum_{i=0}^{k-1}\binom{k}{i}\Bigl( \sum_{j=0}^{n}(-1)^{j+1}\binom{n}{j}j^{k-i}\Bigr)x^{i} \Bigr) \\\notag
&=\Bigl(x^k+(-1)^{n+1}\sum_{i=0}^{k-1}\binom{k}{i}\Bigl(n! S(k-i,n)\Bigr)x^{i}\Bigr) \\\notag
&=x^k+(-1)^{n+1}n!\sum_{i=0}^{k-1}\binom{k}{i}S(k-i,n)x^i\\\notag
&=x^k+(-1)^{n+1}n!\sum_{i=0}^{k-n}\binom{k}{i}S(k-i,n)x^i\\\notag
\end{align}
\end{proof}
\begin{definition}
The difference operator $\Delta_h^1$ is defined by
$$\Delta_h^1 f(x)= f(x+h)-f(x)$$
The backward difference operator $\triangledown_h^1$ is defined by
$$\triangledown_h^1 f(x)= f(x)-f(x-h)$$
$$\Delta_h^2 f(x)=\Delta_h^1(\Delta_h^1 f(x))$$, etc.
\end{definition}
\begin{proposition}\label{basicdiffprop}
\begin{align}\notag
\Delta_h^n f(x)&=(-1)^n\sum_{k=0}^n(-1)^k\binom{n}{k}f(x+k h)\\\notag
\triangledown_h^n f(x)&=\sum_{k=0}^n(-1)^k\binom{n}{k}f(x-k h)\\\notag
\end{align}
\end{proposition}
\begin{proof}
See refs. [12], [13], [16].
\end{proof}
\begin{proposition}\label{contourprop}
Let $f(x)$ be analytic in a domain that contains $[n_0,+\infty)$. Then the differences of the sequence $\{f(x+h k)\},k=0,\ldots,n,$ admit the integral representation 
\begin{align}
\sum_{k=n_0}^n \binom{n}{k}(-1)^k f(x_0+h k)&=\frac{(-1)^n}{2i\pi}\int_C f(h x+x_0)\frac{n!}{x(x-1)\cdots(x-n)}dx\\\notag
\end{align}
where $C$ is a positively oriented closed curve that lies in the domain of analyticity of $f(x)$, encircles $[n_0,n]$, and does not include any of the integers $0,1,\ldots,n_0-1$ (see ref. [15]).
\end{proposition}
\begin{proof}
\begin{align}
{\rm Res}_{x=k} f(h x+x_0) \frac{n!}{x(x-1)\cdots(x-n)}&=\frac{(-1)^n n!}{k!(n-k)!}f(x_0+h k)\\\notag
\end{align}
\end{proof}
\begin{proposition}\label{contourprop2}
Let $f(x)$ be analytic in a domain that contains $[n_0,+\infty)$. Assume $f(x)$ is of polynomial growth (finite degree) in the half plane ${\rm Re}(x)\ge c$ for some $c<n_0$. Then the differences of the sequence $\{f(x+h k)\},k=0,\ldots,n,$ admit the integral representation 
\begin{align}
\sum_{k=n_0}^n \binom{n}{k}(-1)^k f(x_0+h k)&=-\frac{(-1)^n}{2i\pi}\int_{c-i\infty}^{c+i\infty} f(h x+x_0)\frac{n!}{x(x-1)\cdots(x-n)}dx\\\notag
\end{align}

\end{proposition}
\begin{proof}
 (see ref. [15]).
\end{proof}
\begin{lemma}
\begin{align}
\sum_{j=1}^n (-1)^{j+1}\binom{n}{j}f(x+j h)&=f(x)+(-1)^{n+1} \Delta_h^n f(x)\\\notag
\sum_{j=1}^n (-1)^{j+1}\binom{n}{j}f(x-j h)&=f(x)+ \triangledown_h^n f(x)\\\notag
\end{align}
\end{lemma}
\begin{proof}
Prop. \ref{basicdiffprop}.
\end{proof}
\begin{theorem}\label{analytic}
If $f(x)$ is analytic on the interval $(x,+\infty)$
then
$$f(x)=\sum_{j=1}^{n}(-1)^{j-1}\binom{n}{j}f(x+jh)+(-1)^{n} \Delta_h^n f(x)$$
and
\begin{align}
f(x)&=\lim_{n\to\infty}\sum_{j=1}^{n}(-1)^{j-1}\binom{n}{j}f(x+j h)\\\notag
&\iff\\\notag
&\lim_{n\to\infty} \Delta_h^n f(x)=0\\\notag
\end{align}
Also, if $f(x)$ is analytic on the interval $(-\infty,x)$
$$f(x)=\sum_{j=1}^{n}(-1)^{j-1}\binom{n}{j}f(x-jh)+ \triangledown_h^n f(x)$$
and
\begin{align}
f(x)&=\lim_{n\to\infty}\sum_{j=1}^{n}(-1)^{j-1}\binom{n}{j}f(x-jh)\\\notag
&\iff\\\notag
&\lim_{n\to\infty} \triangledown_h^n f(x)=0\\\notag
\end{align}
\end{theorem}
\begin{proof}
Previous lemma.
\end{proof}
The following lemma shows we can associate a series to our dynamic series. The dynamic series is the sequence of partial sums.
\begin{lemma}
If
\begin{align}\notag
f(x)&=\lim_{n\to\infty}\sum_{j=1}^{n}(-1)^{j-1}\binom{n}{j}f(x+jh)\\\notag
\end{align}
then
\begin{align}\notag
f(x)&=\sum_{n=0}^{\infty}\left(\sum_{j=0}^{n}(-1)^{j}\binom{n}{j}f(x+(j+1)h)\right)\\\notag
&=\sum_{n=0}^{\infty}(-1)^n \Delta_h^n f(x+h)\\\notag
\end{align}
\end{lemma}
\begin{proof}
\begin{align}\notag
\sum_{j=1}^{n}(-1)^{j-1}\binom{n}{j}f(x+jh)=\sum_{k=0}^{n-1}\left(\sum_{j=0}^{k}(-1)^{j}\binom{k}{j}f(x+(j+1)h)\right)\\\notag
\end{align}
\end{proof}

\begin{lemma}
$$S(m,n)\simeq \frac{n^m}{n!}$$
\end{lemma}
\begin{proof}
See ref. [13].
\end{proof}

\begin{proposition}\label{zetadelta}
\begin{align}
\lim_{n\to\infty}\Delta_h^n \zeta(x)&=0, \  x>1\\\notag
\end{align}
\end{proposition}
\begin{proof}
Note that if $x_0>1$, the function $\zeta(x_0+x)$ is of polynomial growth for $x>0$. Hence proposition \ref{contourprop2} applies. In particular, if $x_0\ge 2$, then in proposition \ref{contourprop2} we can take $c=-\frac{1}{2}$. It remains to show 
\begin{align}
\lim_{n\to\infty}-\frac{(-1)^n}{2i\pi}\int_{-\frac{1}{2}-i\infty}^{-\frac{1}{2}+i\infty} \zeta(h x+x_0)\frac{n!}{x(x-1)\cdots(x-n)}dx&=0\\\notag
\end{align}
Let us take, for example, $x_0=3$, $h=1$. We are interested in showing
\begin{align}
&\lim_{n\to\infty}\int_{-\frac{1}{2}-i\infty}^{-\frac{1}{2}+i\infty} \zeta( x+3)\frac{n!}{x(x-1)\cdots(x-n)}dx\\\notag
&=\lim_{n\to\infty}(-1)^{n+1}\int_{-\infty}^{+\infty} \zeta(\frac{5}{2}+it )\frac{n!}{(\frac{1}{2}+it)(\frac{3}{2}+it)\cdots(\frac{2n+1}{2}+it)}dt\\\notag
&=\lim_{n\to\infty}(-1)^{n+1}2{\rm Re}\int_{0}^{+\infty} \zeta(\frac{5}{2}+it )\frac{n!}{(\frac{1}{2}+it)(\frac{3}{2}+it)\cdots(\frac{2n+1}{2}+it)}dt\\\notag
&=0\\\notag
\end{align}
Note:
\begin{align}
&\lim_{n\to\infty}\int_0^\infty \zeta(\sigma+i t)\frac{n!}{(\frac{1}{2}+it)(\frac{3}{2}+it)\cdots(\frac{2n+1}{2}+it)}dt\\\notag
&=\lim_{n\to\infty}\lim_{m\to\infty}\int_0^m \sum_{j=1}^\infty j^{-(\sigma+i t)}\frac{n!}{(\frac{1}{2}+it)(\frac{3}{2}+it)\cdots(\frac{2n+1}{2}+it)}dt\\\notag
&=\lim_{n\to\infty}\lim_{m\to\infty}\int_0^m \sum_{j=1}^\infty e^{-(\sigma+i t)\ln(j)}\frac{n!}{(\frac{1}{2}+it)(\frac{3}{2}+it)\cdots(\frac{2n+1}{2}+it)}dt\\\notag
&=\lim_{n\to\infty}\lim_{m\to\infty} \sum_{j=1}^\infty \frac{1}{j^\sigma}\int_0^m(\cos(t \ln(j))-i \sin(t\ln(j)))\frac{n!}{(\frac{1}{2}+it)(\frac{3}{2}+it)\cdots(\frac{2n+1}{2}+it)}dt\\\notag
&=\lim_{m\to\infty} \sum_{j=1}^\infty \frac{1}{j^\sigma}\int_0^m\lim_{n\to\infty}(\cos(t \ln(j))-i \sin(t\ln(j)))\frac{n!}{(\frac{1}{2}+it)(\frac{3}{2}+it)\cdots(\frac{2n+1}{2}+it)}dt\\\notag
&=0
\end{align}
since
$$\lim_{n\to\infty}\frac{n!}{(\frac{1}{2}+it)(\frac{3}{2}+it)\cdots(\frac{2n+1}{2}+it)}=0$$

\end{proof}

\begin{corollary}
$$\zeta(k)=\lim_{n\to\infty}\sum_{j=1}^{n}(-1)^{j-1}\binom{n}{j}\zeta(k+jh)$$
E.g.,
\begin{align}
\zeta(3)&=\lim_{n\to\infty}\sum_{j=1}^{n}(-1)^{j-1}\binom{n}{j}\zeta(3+j)\\\notag
&=\lim_{n\to\infty}\sum_{j=1}^{n}(-1)^{j-1}\binom{n}{j}\zeta(3+2j)\\\notag
&=\lim_{n\to\infty}\sum_{j=1}^{n}(-1)^{j-1}\binom{n}{j}\zeta(3+j/2)\\\notag
&=\lim_{n\to\infty}\sum_{j=1}^{n}(-1)^{j-1}\binom{n}{j}\zeta(3+3j)\\\notag
\vdots
\end{align}
\end{corollary}
\begin{proof}
Theorem \ref{analytic} plus proposition \ref{zetadelta}.
\end{proof}
\begin{corollary}
Assume $f(x)$ is analytic on $(a,k]$ and $f$ has a simple pole at $x=a$ with $Res_a=m$ then
$$f(k)=\Bigl(\lim_{n\to\infty}\sum_{j=1}^{n-1}(-1)^{j-1}\binom{n-1}{j}f(k-j\frac{k-a}{n})\Bigr)-\frac{m}{k-a}$$
\end{corollary}
\begin{proof}
Pending.
\end{proof}
\begin{corollary}
$$\zeta(k)=\Bigl(\lim_{n\to\infty}\sum_{j=1}^{n-1}(-1)^{j-1}\binom{n-1}{j}\zeta(k-j\frac{k-1}{n})\Bigr)-\frac{1}{k-1}$$
\end{corollary}
\begin{proof}
Previous proposition.
\end{proof}

\begin{proposition}\label{zetalogdelta}
\begin{align}
\lim_{n\to\infty}&\Delta_h^n \ln(\zeta(x))=0\\\notag
\lim_{n\to\infty}&\Delta_h^n \zeta(\ln(x))=0\\\notag
\end{align}
\end{proposition}
\begin{proof}
\end{proof}
\begin{corollary}
\begin{align}
\zeta(k)&=\lim_{n\to\infty}\prod_{j=1}^{n}\zeta(k+jh)^{(-1)^{j-1}\binom{n}{j}}\\\notag
&=\lim_{n\to\infty}\frac{\prod_{j\ {\rm odd}}\zeta(k+jh)^{\binom{n}{j}}}{\prod_{j\ {\rm even}}\zeta(k+jh)^{\binom{n}{j}}}\\\notag
\end{align}
\end{corollary}
\begin{proof}
\begin{align}
\ln(\zeta(k))&=\lim_{n\to\infty}\sum_{j=1}^{n}(-1)^{j-1}\binom{n}{j}\ln(\zeta(k+jh))\\\notag
&=\lim_{n\to\infty}\ln(\prod_{j=1}^{n}\zeta(k+jh)^{(-1)^{j-1}\binom{n}{j}})\\\notag
\end{align}
\end{proof}
\begin{corollary}
\begin{align}
\zeta(k)&=\lim_{n\to\infty}\sum_{j=1}^{n}(-1)^{j-1}\binom{n}{j}\zeta(\ln(\exp(k)+jh))\\\notag
\end{align}
(for $\zeta(3)$: 100 decimal place accuracy with 100 terms, for $h=\frac{1}{16}$)
\end{corollary}
\begin{proof}
Proposition \ref{zetalogdelta}.
\end{proof}
\begin{corollary}
$$\zeta(k)=1+\lim_{n\to\infty}\sum_{j=1}^{n}(-1)^{j-1}\binom{n}{j}(\zeta(k+jh)-1)$$
\end{corollary}
\begin{proof}
This is essentially doing the computation mod 1.
\end{proof}
\begin{proposition}
$$1=\sum_{j=1}^{\infty}\binom{j+k-1}{j}(\zeta(k+j)-1)$$
\end{proposition}
\begin{proof}
Pending.
\end{proof}
\begin{corollary}
\begin{align}
\zeta(k)&=\frac{1}{k-1}\Bigl(k-\sum_{j=1}^{\infty}\binom{j+k-1}{k-2}(\zeta(k+j)-1)\Bigr)\\\notag
&=\frac{1}{k-1}\lim_{n\to\infty}\Bigl(\binom{k+n}{n+1}-\sum_{j=1}^{n}\binom{j+k-1}{k-2}\zeta(k+j)\Bigr)\\\notag
\end{align}
\end{corollary}
\begin{proof}
Previous proposition.
\end{proof}
\begin{proposition}
\begin{align}\notag
\lim_{n\to\infty}\Delta_1^n \sin(x)&=0\\\notag
\end{align}
\end{proposition}
\begin{proof}
\begin{align}
\Delta_1^n \sin(x)&=(-1)^n\sum_{k=0}^n(-1)^k\binom{n}{k}\sin(x+k )\\\notag
&=(-1)^{\floor{\frac{n}{2}}}(2\sin(\frac{1}{2}))^n \sin(\frac{n}{2}(1+\pi)+x)\\\notag
\end{align}
while
$$2\sin(\frac{1}{2})\approx 0.959$$
and 
$$|\sin(\frac{n}{2}(1+\pi)+x)|\le 1$$
$$\implies \lim_{n\to\infty}\Delta_1^n \sin(x)=0$$
\end{proof}
\begin{proposition}
\begin{align}\notag
\lim_{n\to\infty}\Delta_1^n \ln(x)&=0\\\notag
\end{align}
\end{proposition}
\begin{proof}
\begin{align}
\Delta_1^n\ln(x)&=(-1)^n\sum_{k=0}^n(-1)^k\binom{n}{k}\ln(x-k)\\\notag
&=(-1)^n\ln\Bigl(\frac{ \prod_{k\rm{\  even}}(x+k)^{\binom{n}{k}}}{\prod_{k\rm{\  odd}}(x+k)^{\binom{n}{k}}}\Bigr)\\\notag
&=(-1)^n\ln\Bigl(\frac{x^{\sum_{k\rm{\  even}}\binom{n}{k}}+\rm{lower\  order\  terms}}{x^{\sum_{k\rm{\  odd}}\binom{n}{k}}+\rm{lower\   order\   terms}}\Bigr)\\\notag
\end{align}
But
\begin{align}
(1-1)^n&=\sum_{k=0}^n(-1)^k\binom{n}{k}=0\\\notag
&\implies \sum_{k\rm{\  even}}\binom{n}{k}=\sum_{k\rm{\  odd}}\binom{n}{k}\\\notag
&\implies \lim_{n\to\infty}\Delta_1^n \ln(x)=\ln(1)=0\\\notag
\end{align}
\end{proof}
\begin{proposition}
\begin{align}\notag
\lim_{n\to\infty}\triangledown_1^n \exp(x)&=0\\\notag
\end{align}
\end{proposition}
\begin{proof}
\begin{align}
\triangledown_1^n \exp(x)&=\sum_{k=0}^n(-1)^k\binom{n}{k}\exp(x-k )\\\notag
&=\left(\frac{1}{\rm e}-1 \right)^n\exp(x)\\\notag
\end{align}
But,
\begin{align}
\frac{1}{\rm e}-1&\approx -0.632\\\notag
&\implies \lim_{n\to\infty}\triangledown_1^n \exp(x)=0\\\notag
\end{align}
\end{proof}
\begin{corollary}
\begin{align}\notag
x&=\lim_{n\to\infty}\sum_{j=1}^n(-1)^{j+1}\binom{n}{j}x^{1-j},1<x\\\notag
x&=\lim_{n\to\infty}\sum_{j=1}^n(-1)^{j+1}\binom{n}{j}x^{1+j},0<x<1\\\notag
\end{align}
\end{corollary}
\begin{proof}
In theorem \ref{analytic} let $f(x)=k^x$, where $x>0$, then 
\begin{align}
f(x)&=\exp(x\ln(k))\\\notag
&=\sum_{j=0}^{\infty}\frac{1}{j!} (x\ln(k))^j\\\notag
\end{align}
The theorem says,
$$k=f(1)=\lim_{n\to\infty}\sum_{j=1}^{n}(-1)^{j+1}\binom{n}{j}f(x-j)=\lim_{n\to\infty}\sum_{j=1}^{n}(-1)^{j+1}\binom{n}{j}k^{1-j}\notag$$
if $k>1$, etc.
\end{proof}
\begin{lemma}
\begin{align}\notag
0&=\lim_{n\to\infty}\sum_{j=1}^n(-1)^{j+1}\binom{n}{j}2^{1+j},n\rm{\  even}\\\notag
4&=\lim_{n\to\infty}\sum_{j=1}^n(-1)^{j+1}\binom{n}{j}2^{1+j},n\rm{\  odd}\\\notag
\end{align}
\end{lemma}
\begin{proof}
\end{proof}

\begin{proposition}
There is a class of functions $f(j)$ such that
\begin{align}
0&=\lim_{n\to\infty}\sum_{j=0}^{n}(-1)^{j+1}\binom{n}{j}f(x+j)\\\notag
\end{align}
which implies
\begin{align}
f(x)&=\lim_{n\to\infty}\sum_{j=1}^{n}(-1)^{j+1}\binom{n}{j}f(x+j)\\\notag
\end{align}
Those functions include
\begin{align}
&x^m\\\notag
&(\zeta(x))^m\\\notag
&(\eta(x))^m\\\notag
&(\sin(x))^m\\\notag
&(\ln(x))^m\\\notag
&(\exp(-x))^m\\\notag
&(\eta(k))^{x m}\\\notag
&(k-x)^m\\\notag
\end{align}
For $x,m\in R, x,m>0$.\\
There is also class of functions $f(j)$ such that
\begin{align}
0&=\lim_{n\to\infty}\sum_{j=0}^{n}(-1)^{j+1}\binom{n}{j}f(x-j)\\\notag
\end{align}
which implies
\begin{align}
f(x)&=\lim_{n\to\infty}\sum_{j=1}^{n}(-1)^{j+1}\binom{n}{j}f(x-j)\\\notag
\end{align}
Those functions include
\begin{align}
&(\exp(x))^m\\\notag
&(k-x)^m\\\notag
&\ln(1-x)\\\notag
\end{align}
For $x,m\in R, x,m>0$.
\end{proposition}
\begin{proof}
 Theorem \ref{analytic}.
\end{proof}
Note that any dynamic sum $S_n$ with the property that $\lim_{n\to\infty}S_n=0$ and having $f(x)$ as a summand will produce a formula for $f(x)$.

\begin{proposition}The same class of functions listed in the previous proposition also satisfy
\begin{align}
f(x)&=\lim_{n\to\infty}\sum_{j=1}^{n}(-1)^{j+1}\binom{n}{j} f(x+h\sum_{k=1}^{j}r^k),\ 0\le r\le 1\\\notag
f(x)&=\lim_{n\to\infty}\sum_{j=1}^{n}(-1)^{j+1}\binom{n}{j} f(x+h\sum_{k=1}^{j}\frac{1}{1-\exp(-r x)}),\ 0\le r\\\notag
\end{align}
\end{proposition}
\begin{proof}
Conjecture, proof pending.
\end{proof}
\begin{corollary}
\begin{align}
\zeta(x)&=\lim_{n\to\infty}\sum_{j=1}^{n}(-1)^{j+1}\binom{n}{j} \zeta(x+h\sum_{k=1}^{j}r^k),\ 0\le r\le 1\\
\notag
\zeta(x)&=\lim_{n\to\infty}\sum_{j=1}^{n}(-1)^{j+1}\binom{n}{j} \zeta(x+h\sum_{k=1}^{j}\frac{1}{1-\exp(-r x)}),\ 0\le r\\\notag
\end{align}
\end{corollary}
\begin{proof}
Conjecture, proof pending.
\end{proof}

\section{More Zeta Series And Trig Integrals}

\begin{corollary}
\begin{align}\notag
\zeta(3)&=\frac{1}{7} \pi^2\left(1 +2\sum_{j=1}^{\infty} \left(        1+4j\left(     -\coth^{-1}(2j)+j\left(\ln(4)-\ln\left(4-\frac{1}{j^2}\right)\right)      \right)       \right)        \right)\\\notag
&=\frac{\pi^2}{7}\left(  1+2\sum_{j=1}^{\infty} \left(        1+2j\left(   \ln\left(\left(\frac{2j-1}{2j+1}\right)\left(\frac{4j^2}{4j^2-1}\right)^{2j}\right) \right)       \right)        \right)\notag
\end{align}

\end{corollary}
\begin{proof}
From corollary \ref{lnsincor} and lemma \ref{eulersinlemma},
\begin{align}\notag
\zeta(3)&=\frac{4}{7}\int_{0}^{\pi/2}(4x-\pi)\ln(\sin(x))dx\\\notag
&=\frac{4}{7}\int_{0}^{\pi/2}(4x-\pi)\ln\left(x\prod_{j=1}^{\infty}\left(1-\frac{x^2}{j^2\pi^2}\right)\right)dx\\\notag
&=\frac{4}{7}\int_{0}^{\pi/2}(4x-\pi)\left(\ln(x)+\sum_{j=1}^{\infty}\ln\left(1-\frac{x^2}{j^2\pi^2}\right)\right)dx\\\notag
&=\frac{4}{7}\left(  \int_{0}^{\pi/2}(4x-\pi)\ln(x)dx+\sum_{j=1}^{\infty}\int_{0}^{\pi/2}(4x-\pi)\ln  \left(1-\frac{x^2}{j^2\pi^2}  \right)dx   \right)\\\notag
&=\frac{4}{7}\left(  x(2x-\pi)\ln(x)+x(\pi-x)+\sum_{j=1}^{\infty}\left(2x(\pi-x)-\pi^2j\ln  \left(\frac{j\pi+x}{j\pi-x}  \right)-(2\pi^2j^2+x(\pi-2x))\ln\left(1-\frac{x^2}{j^2\pi^2}\right) \right) \right)|_{x=0}^{x=\pi/2}\\\notag
&=\frac{1}{7}\left(  \pi^2+2\pi^2\sum_{j=1}^{\infty} \left(        1+2j\left(   \ln\left(\frac{2j-1}{2j+1}\right)+2j\left(\ln(4)-\ln\left(4-\frac{1}{j^2}\right)\right)      \right)       \right)        \right)\\\notag
&=\frac{\pi^2}{7}\left(  1+2\sum_{j=1}^{\infty} \left(        1+2j\left(   \ln\left(\left(\frac{2j-1}{2j+1}\right)\left(\frac{4j^2}{4j^2-1}\right)^{2j}\right) \right)       \right)        \right)\notag
\end{align}
\end{proof}
\begin{corollary}
\begin{align}\notag
\zeta(3)&=\frac{4}{7}\pi^2\Bigl(\frac{1}{4}-\sum_{k=1}^{\infty}\sum_{j=1}^{\infty}\frac{1}{2^{2k+1}j^{2k}(k+1)(2k+1)}\Bigr)\\\notag
&=\frac{4}{7}\pi^2\Bigl(\frac{1}{4}-\sum_{k=1}^{\infty}\frac{1}{2^{2k+1}(k+1)(2k+1)}\zeta(2k)\Bigr)\\\notag
&=\frac{1}{7}\pi^2\Bigl(1-\sum_{k=1}^{\infty}(-1)^{k+1}\Bigl(\frac{\pi^{2k}}{(k+1)(2k+1)(2k)!}\Bigr)B_{2k}\Bigr)\\\notag
\end{align}
Note: This formula converges quickly.
\end{corollary}
\begin{proof}
From the proof of the above corollary, if we let
$$S(n,x)=\frac{4}{7}\left(  x(2x-\pi)\ln(x)+x(\pi-x)+\sum_{j=1}^{n}\left(2x(\pi-x)-\pi^2j\ln  \left(\frac{j\pi+x}{j\pi-x}  \right)-(2\pi^2j^2+x(\pi-2x))\ln\left(1-\frac{x^2}{j^2\pi^2}\right) \right) \right)\notag$$ Then doing a series expansion on each summand,
\begin{align}\notag
S(n,x)&=\frac{4}{7}\Bigl(  x(2x-\pi)\ln(x)+x(\pi-x)+\sum_{j=1}^{n}\sum_{k=1}^{n}\Bigl(\frac{1}{k(2k+1)j^{2k}\pi^{2k-1}}x^{2k+1}-\frac{2}{k(k+1)j^{2k}\pi^{2k}}x^{2k+2}\Bigr)\Bigr)\\\notag
\end{align}
Letting $x=\frac{\pi}{2}$ gives the result.
\end{proof}
Referring to $S(n,x)$ in the previous proof,  it can be shown that
\begin{align}\notag
S(n,\pi/2)&=\zeta(3)\\\notag
&+\frac{1}{7}\pi^2\Bigl(2n+1+4\Bigl(-4\zeta'\left(-2,n+1\right)+2\zeta'\left(-2,n+\frac{1}{2}\right)+2\zeta'\left(-2,n+\frac{3}{2}\right)\\\notag
&+\zeta'\left(-1,n+\frac{1}{2}\right)-\zeta'\left(-1,n+\frac{3}{2}\right)\Bigr)\Bigr)\\\notag
\end{align}
where $\zeta(a,b)$ is the Hurwitz zeta function. Note this gives us a way to express $\zeta(3)$ in terms of the derivative of the Hurwitz zeta function. E.g.,
let 
\begin{align}
s(n)&=\frac{1}{7}\pi^2\Bigl(2n+1+4\Bigl(-4\zeta'\left(-2,n+1\right)+2\zeta'\left(-2,n+\frac{1}{2}\right)+2\zeta'\left(-2,n+\frac{3}{2}\right)\\\notag
&+\zeta'\left(-1,n+\frac{1}{2}\right)-\zeta'\left(-1,n+\frac{3}{2}\right)\Bigr)\Bigr)\\\notag
\end{align}
 then
$$\zeta(3)=S(n,\pi/2)-s(n)$$
For $n=1$ we get
$$S(1,\pi/2)=\frac{1}{7}\pi^2\left(3+4\ln\left(\frac{16}{27}\right)\right)$$
which implies
$$\zeta(3)=\frac{4}{7}\pi^2\left(\ln\left(\frac{16}{27}\right)+4\zeta'\left(-2,2\right)-2\zeta'\left(-2,\frac{3}{2}\right)-2\zeta'\left(-2,\frac{5}{2}\right)-\zeta'\left(-1,\frac{3}{2}\right)+\zeta'\left(-1,\frac{5}{2}\right)\right)$$
For $n=2$
\begin{align}\notag
S(2,\pi/2)&=\frac{4}{7}\pi^2\left(\frac{5}{4}+8\ln\left(\frac{16}{15}\right)+2\ln\left(\frac{4}{5}\right)-\ln(3)\right)\\\notag
&=\frac{4}{7}\pi^2\left(\frac{5}{4}+\ln\left(\frac{68719476736}{192216796875}\right)\right)\\\notag
\end{align}
and
$$\zeta(3)=\frac{4}{7}\pi^2\left(\ln\left(\frac{68719476736}{192216796875}\right)+4\zeta'\left(-2,3\right)-2\zeta'\left(-2,\frac{5}{2}\right)-2\zeta'\left(-2,\frac{7}{2}\right)-\zeta'\left(-1,\frac{5}{2}\right)+\zeta'\left(-1,\frac{7}{2}\right)\right)$$
Note thse equations can also be used to derive relationships between values of $\zeta'$.
\begin{corollary}
\begin{align}
\zeta(3)&=\frac{2}{7}\int_{0}^{\pi/2}(\pi-2x)\ln(\csc( x)+\cot(x))dx\notag
\end{align}
\end{corollary}
\begin{proof}
From Theorem \ref{cscthm},
$$\zeta(3)=\frac{2}{7}\int_{0}^{\pi/2}(x(\pi-x)) \csc( x)dx$$
Using integration by parts, letting $U=x(\pi-x)$, $dV=\csc( x)$
\begin{align}\notag
\zeta(3)&=\frac{2}{7}\left(-x(\pi-x)\ln(\csc( x)+\cot(x))|_{x=0}^{x=\pi/2}+\int_{0}^{\pi/2}(\pi-2x)\ln(\csc( x)+\cot(x))dx\right)\\\notag
&=\frac{2}{7}\int_{0}^{\pi/2}(\pi-2x)\ln(\csc( x)+\cot(x))dx\notag
\end{align}
\end{proof}
\begin{lemma}
\begin{align}\notag
\frac{2}{7}\int_{0}^{\pi/2}\pi\ln(\csc( x)+\cot(x)) dx=\frac{4\pi}{7}G\\\notag
\end{align}
where $G$ is Catalan's constant.
\end{lemma}
\begin{corollary}
\begin{align}\notag
\zeta(3)&=\frac{4}{7}\left(\pi G-\int_{0}^{\pi/2}x\ln(\csc(x)+\cot(x))dx\right)\notag
\end{align}
\end{corollary}

\begin{corollary}
\begin{align}\notag
\zeta(3)&=\frac{16}{21}\left(\pi G-\frac{1}{8}\pi^2\ln(2)-\int_{0}^{\pi/2}x \ln(1+\cos(x)))dx\right)\notag
\end{align}
\end{corollary}
\begin{proof}
From the above proposition and corollary,
\begin{align}\notag
\zeta(3)&=\frac{4}{3}\left(\zeta(3)-\frac{1}{4}\zeta(3)\right)\\\notag
&=\frac{4}{3}\left(  \frac{4}{7}\left(\pi G-\int_{0}^{\pi/2}x\ln(\csc(x)+\cot(x))dx\right)-\frac{4}{7}\left(\frac{1}{8}\pi^2\ln(2)+\int_{0}^{\pi/2}x \ln(\sin(x))dx\right)  \right)\\\notag
&=\frac{16}{21}\left(\pi G-\frac{1}{8}\pi^2\ln(2)-\int_{0}^{\pi/2}x\ln(1+\cos(x)))dx\right)\notag
\end{align}
\end{proof}
\begin{lemma}
\begin{align}\notag
\int_{0}^{\pi/2}\ln(1+\cos(x)) dx=2G-\frac{1}{2}\pi\ln(2)\\\notag
\end{align}
\begin{corollary}
\begin{align}\notag
\zeta(3)&=\frac{8}{21}\left(\frac{1}{4}\pi^2\ln(2)-\int_{0}^{\pi/2}(2x-\pi) \ln(1+\cos(x)))dx\right)\notag
\end{align}
\end{corollary}
\begin{lemma}
\begin{align}\notag
\zeta(3)&=\frac{4}{63}\left(   \frac{1}{2}\pi^2(\pi+3\ln(2))+\int_{0}^{\pi/2}\left(2 x^3-3\pi x^2\right)\left(\frac{1}{1+\cos(x)}\right)dx   \right)\notag
\end{align}
\end{lemma}
\end{lemma}

\begin{lemma}\label{lnlemma}
\begin{align}\notag
-\int_{0}^{\pi/2}4 x \ln(\sin(x)\cos(x)) dx=\pi^2\ln(2)\\\notag
\end{align}
\end{lemma}

\begin{proposition}
\begin{align}\notag
\zeta(3)&=\frac{8}{7}\int_{0}^{\pi/2}x \ln(\tan(x)) dx\\\notag
\end{align}
\end{proposition}
\begin{proof}
From Prop. \ref{extprop} and Lemma \ref{lnlemma}
\begin{align}\notag
\zeta(3)&=\frac{8}{7}\left(\frac{1}{4}\pi^2\ln(2)+2\int_{0}^{\pi/2}(x \ln(\sin(x)) dx\right)\\\notag
&=\frac{8}{7}\left(-\int_{0}^{\pi/2} x \ln(\sin(x)\cos(x)) dx+2\int_{0}^{\pi/2}(x \ln(\sin(x)) dx\right)\\\notag
&=\frac{8}{7}\int_{0}^{\pi/2}x \ln(\tan(x)) dx\\\notag
\end{align}
\end{proof}
\begin{corollary}
Let 
$$f(x)=\frac{8}{7}\int_{0}^{x}x \ln(\tan(t)) dt\notag$$
then
$$f(x)=\frac{8}{7}\Bigl(\frac{1}{4}\left(2\ln(x)-1\right)x^2+\sum_{j=1}^{\infty}\frac{2^{2j-1}}{\pi^{2j}j(j+1)}\eta(2j)x^{2j+2}\Bigr)$$
\end{corollary}
\begin{proof}
Write out the series expansion and compare the coefficients with those of the even power terms (starting with the 4th power) in (\ref{formula1}) or (\ref{formula2}). One finds they differ by the factor $\frac{2^{2j-1}}{j}$.
\end{proof}
\begin{corollary}\label{latecor}
$$\zeta(3)=\frac{1}{7}\pi^2\Bigl(\ln(\frac{\pi}{2})-\frac{1}{2}+\sum_{j=1}^{\infty}\frac{1}{j(j+1)}\eta(2j)\Bigr)$$
\end{corollary}
\begin{proof}
In previous corollary, $\zeta(3)=f\left(\frac{\pi}{2}\right)$.
\end{proof}
\begin{proposition}
\begin{align}\notag
\zeta(3)&=\frac{2}{7}\left(\pi^2\ln(2)-4\int_{0}^{\pi/2}x^2 \cot( x)dx\right)\\\notag
\end{align}
\end{proposition}
\begin{proof}
Reference [4].
\end{proof}
\begin{lemma}
\begin{align}\notag
\int_{0}^{\pi/2}2\pi x \cot\left( x\right) dx=\pi^2\ln(2)\\\notag
\end{align}
\end{lemma}
\begin{corollary}
\begin{align}\notag
\zeta(3)&=\frac{4}{7}\int_{0}^{\pi/2}x(\pi-2x) \cot\left( x\right) dx\\\notag
&=\frac{4}{\pi}\int_{0}^{\pi/2}x^2(\pi-2x) \cot\left( x\right) dx\\\notag
&=\frac{4}{5\pi}\int_{0}^{\pi/2}x(\pi-2x)^2\cot\left( x\right) dx\\\notag
&=\frac{2}{\pi}\int_{0}^{\pi/2}x(2x-\pi)(5x-\pi)\cot\left( x\right) dx\\\notag
&=\frac{4}{213\pi}\int_{0}^{\pi/2}x(\pi-2x)(185x+4\pi)\cot\left( x\right) dx\\\notag
&=\frac{4}{1255\pi}\int_{0}^{\pi/2}x(\pi-2x)(1311x-8\pi)\cot\left( x\right) dx\\\notag
&=\frac{4}{3815325\pi^3}\int_{0}^{\pi/2}x(\pi-2x)(8487800x^3-6023600\pi x^2+4650613\pi^2 x-1984\pi^3)\cot\left( x\right) dx\\\notag
\end{align}
\end{corollary}
\begin{proof}
Follows from the above proposition and lemma.
\end{proof}
We could generate an infinite collection of these integrals of the form $c\int_{0}^{\pi/2}x(2x-\pi)p(x)\cot\left( x\right) dx$ where $c$ is a constant and $p(x)$ is a polynomial in $x$ with integer coefficients. This is equivalent to generating an infinite collection of integrals of that form that integrate to produce the value $0$, perhaps not so surprising.
\begin{corollary}
\begin{align}\notag
\zeta(5)&=\frac{4}{93}\int_{0}^{\pi/2}x^2(2x-\pi)(4x-9\pi) \cot\left( x\right) dx\\\notag
\end{align}
\end{corollary}
\begin{corollary}
\begin{align}\notag
\zeta(5)&=\zeta(3)+\int_{0}^{\pi/2}4x(2x-\pi)(\frac{1}{93}x(4x-9\pi)+\frac{1}{7}) \cot\left( x\right) dx\\\notag
\end{align}
\end{corollary}
\begin{corollary}
\begin{align}\notag
\zeta(7)&=-\frac{8}{5715}\int_{0}^{\pi/2}x^2(2x-\pi)(16 x^3-16\pi  x^2-8 \pi ^2 x+27 \pi ^3)\cot\left( x\right) dx\\\notag
\end{align}
\end{corollary}
\begin{proposition}
\begin{align}\notag
\zeta(3)=-\frac{4}{3\pi}\left(\frac{\pi}{2}(\ln(2))^3+\frac{\pi^3}{8}\ln(2)+\int_{0}^{\pi/2}(\ln(\cos(x)))^3dx\right)\\\notag
\end{align}
\end{proposition}
\begin{proof}
See reference [6].
\end{proof}
To summarize, here are a few of the representations of $\zeta(3)$ as  trigonometric integral.
\begin{align}\notag
\zeta(3)&=\frac{2}{7}\int_{0}^{\pi/2}x(\pi-x) \csc\left( x\right)dx\\\notag
&=\frac{4}{7}\int_{0}^{\pi/2}(4x-\pi)\ln(\sin(x))dx\\\notag
&=\frac{1}{7}\int_{0}^{\pi}x(\pi-x) \csc\left( x\right)\\\notag
\zeta(3)&=\frac{4}{7}\int_{0}^{\pi/2}x(\pi-2x) \cot\left( x\right) dx\\\notag
&=\frac{8}{7}\int_{0}^{\pi/2}x \ln(\tan(x)) dx\\\notag
\zeta(3)&=\frac{4}{7}\int_{0}^{\pi/2}x(\pi-2x)\tan(x) dx\\\notag
&=\frac{4}{7}\int_{0}^{\pi/2}(\pi-4x)\ln(\cos(x)) dx\\\notag
\zeta(3)&=\frac{4}{63}\left(   \frac{1}{2}\pi^2(\pi+3\ln(2))+\int_{0}^{\pi/2}x^2\left(2 x-3\pi \right)\left(\frac{1}{1+\cos(x)}\right)dx   \right)\\\notag
&=\frac{8}{21}\left(\frac{1}{4}\pi^2\ln(2)-\int_{0}^{\pi/2}(2x-\pi) \ln(1+\cos(x)))dx\right)\\\notag
&=\frac{4}{7}\int_{0}^{\pi/2}x\ln\left(\frac{1+\sin(x)}{\cos(x)}\right)dx\notag
\end{align}

\section{{\textbf{Related Functions}}}
In this section we return to the sums $S_n$ introduced in section 3.
Here, we introduce a new set of functions for study that correspond to our $S_n$. Here we present some analysis of these functions. They only agree with their $S_n$  counterpart at $x=\frac{\sqrt{2}}{2}$:
\begin{align*}\notag
F_3&=\frac{1}{2^9}\left(4\left(-1\sqrt{2-2x}+3\sqrt{2+2x}\right)\right)\\\notag
F_4&=\frac{1}{2^{12}}\left(8\left(3\sqrt{2-\sqrt{2+2x}}+11\sqrt{2-\sqrt{2-2x}}-2\sqrt{2+\sqrt{2-2x}}+4\sqrt{2+\sqrt{2+2x}}\right)\right)\\\notag
F_5&=\frac{1}{2^{15}}\Big[8 \Big[-34 \sqrt{2-\sqrt{2+\sqrt{2+2x}}}-8 \sqrt{2-\sqrt{2+\sqrt{2-2x}}}+25 \sqrt{2-\sqrt{2-\sqrt{2-2x}}}+15 \sqrt{2-\sqrt{2-\sqrt{2+2x}}}\\\notag
&-45 \sqrt{2+\sqrt{2-\sqrt{2+2x}}}+29 \sqrt{2+\sqrt{2-\sqrt{2-2x}}}+64 \sqrt{2+\sqrt{2+\sqrt{2+2x}}}\Big]\Big]\\\notag
F_6&=\frac{1}{2^{18}}\Big[8 \Big[244 \sqrt{2-\sqrt{2+\sqrt{2+\sqrt{2+2x}}}}-48 \sqrt{2-\sqrt{2+\sqrt{2+\sqrt{2-2x}}}}+85 \sqrt{2-\sqrt{2+\sqrt{2-\sqrt{2-2x}}}}\\\notag
&+401 \sqrt{2-\sqrt{2+\sqrt{2-\sqrt{2+2x}}}}-243 \sqrt{2-\sqrt{2-\sqrt{2-\sqrt{2x+2}}}}-137 \sqrt{2-\sqrt{2-\sqrt{2-\sqrt{2-2x}}}}\\\notag
&-16 \sqrt{2-\sqrt{2-\sqrt{2+\sqrt{2-2x}}}}+382 \sqrt{2-\sqrt{2-\sqrt{2+\sqrt{2+2x}}}}+74 \sqrt{2+\sqrt{2-\sqrt{2+\sqrt{2+2x}}}}\\\notag
&+214 \sqrt{2+\sqrt{2-\sqrt{2+\sqrt{2-2x}}}}+97 \sqrt{2+\sqrt{2-\sqrt{2-\sqrt{2-2x}}}}+135 \sqrt{2+\sqrt{2-\sqrt{2-\sqrt{2+2x}}}}\\\notag
&-15 \sqrt{2+\sqrt{2+\sqrt{2-\sqrt{2+2x}}}}+109 \sqrt{2+\sqrt{2+\sqrt{2-\sqrt{2-2x}}}}-16 \sqrt{2+\sqrt{2+\sqrt{2+\sqrt{2-2x}}}}\\\notag
&-184 \sqrt{2+\sqrt{2+\sqrt{2+\sqrt{2+2x}}}}\Big]\Big]\\\notag
\end{align*}

As a shorthand notation, we can refer to an elementary term by its sequence of signs. For example, $ e=\sqrt{2-\sqrt{2+\sqrt{2-\sqrt{2+\sqrt{2}}}}}$ would correspond to $(-+-+)$.\\ 
\begin{definition}
We define the $\textrm{sign}(e)$ to equal $(-1)^{k+1}$, where $k=$ number of ``-'' signs in the notation for $e$.
\end{definition}
 Each elementary term $e$ has a corresponding function $v(x)$. For example, $\sqrt{2+\sqrt{2-\sqrt{2+\sqrt{2}}}}$ corresponds to $\sqrt{2+\sqrt{2-\sqrt{2+2x}}}$.
When presenting a list of elementary terms, we order them from smallest to largest in magnitude. The following table gives the correspondence between the angle $\theta$ and the value of $2 \sin(\theta)$ for elementary terms of depth $\le 5$.\\

\begin{center}
\begin{tabular}{|c c c c c c c c c c |} 
 \hline
$\pi/64$ &$\pi/32$ &$3\pi/64$&$\pi/16$ &$5\pi/64$ &$3\pi/32$ &$7\pi/64$ &$\pi/8$ &$9\pi/64$&$5\pi/32$\\ 
 \hline
(- + + +)&(- + +) &(- + + -)&(- +)&(- + - -)&(- + -)&(- + - +)&(-) & (- - - +)&(- - -) \\
 \hline
\end{tabular}

\bigskip
\begin{tabular}{|c c c c c c c c c c |} 
 \hline
$11\pi/64$ &$3\pi/16$  &$13\pi/64$ &$7\pi/32$ &$15\pi/64$ &$\pi/4$ &$17\pi/64$&$9\pi/32) $&$19\pi/64$&$5\pi/16$\\ 
 \hline
(- - - -)&(- -) &(- - + -) & (- - +) &(- - + +)&  () & (+ - + +)& (+ - +)&(+ - + -)  &(+ -) \\
 \hline
\end{tabular}

\bigskip
\begin{tabular}{|c c c c c c c c c c c|} 
 \hline
$21\pi/64$ &$11\pi/32$ &$23\pi64$ &$3\pi/8$ &$25\pi/64$ &$13\pi/32$ &$27\pi/64$&$7\pi/16$ &$29\pi/64$ &$15\pi/32$ &$31\pi/64$\\
 \hline
(+ - - -)& (+ - -)&(+ - - +) & (+) & (+ + - +) &(+ + -)& (+ + - -)& (+ +)&  (+ + + -)& (+ + +) &(+ + + +)\\
 \hline
\end{tabular}
\end{center}
The elementary terms have a different form. Consider $f(x)=\sqrt{2-\sqrt{2+\sqrt{2+2 x}}}$. It satisfies the differential equation$$f'(x)=-\frac{\sqrt{4-f(x)^2}}{8 \sqrt{1-x^2}}$$
That equation has general solution $$2 \sin(1/8 (-\arcsin(x) + 8 c_1))$$ Setting $x=0$ and solving for $c_1$ we find $$\left(\frac{1}{2}\right)\sqrt{2-\sqrt{2+\sqrt{2+2 x}}}= \sin\left(\frac{1}{8} \left(-\arcsin(x) + \frac{\pi}{2}\right)\right)$$

\bigskip
\begin{tabular}{|c c c  |} 
 \hline
$f(x)$&Alternate Form &$f(\frac{\sqrt{2}}{2})$ \\ 
 \hline
(- +)&$2\sin\left(\frac{1}{4}\left(-\arcsin(x)+\frac{\pi}{2}\right)\right) $&$2\sin\left(\frac{\pi}{16}\right)$ \\
(- -)&$2\sin\left(\frac{1}{4}\left(\arcsin(x)+\frac{\pi}{2}\right)\right) $&$2\sin\left(\frac{3\pi}{16}\right)$ \\
(+ -)&$2\sin\left(\frac{1}{4}\left(-\arcsin(x)+\frac{3\pi}{2}\right)\right) $&$2\sin\left(\frac{5\pi}{16}\right)$ \\
(+ +)&$2\sin\left(\frac{1}{4}\left(\arcsin(x)+\frac{3\pi}{2}\right)\right) $&$2\sin\left(\frac{7\pi}{16}\right)$ \\
 \hline
\end{tabular}
\bigskip

\bigskip
\begin{tabular}{|c c c  |} 
 \hline
$f(x)$&Alternate Form &$f(\frac{\sqrt{2}}{2})$ \\ 
 \hline
(- + +)&$2\sin\left(\frac{1}{8}\left(-\arcsin(x)+\frac{\pi}{2}\right)\right) $&$2\sin\left(\frac{\pi}{32}\right)$ \\
(- + -)&$2\sin\left(\frac{1}{8}\left(\arcsin(x)+\frac{\pi}{2}\right)\right) $&$2\sin\left(\frac{3\pi}{32}\right)$ \\
(- - -)&$2\sin\left(\frac{1}{8}\left(-\arcsin(x)+\frac{3\pi}{2}\right)\right) $&$2\sin\left(\frac{5\pi}{32}\right)$ \\
(- - +)&$2\sin\left(\frac{1}{8}\left(\arcsin(x)+\frac{3\pi}{2}\right)\right) $&$2\sin\left(\frac{7\pi}{32}\right)$ \\
(+ - +)&$2\sin\left(\frac{1}{8}\left(-\arcsin(x)+\frac{5\pi}{2}\right)\right) $&$2\sin\left(\frac{9\pi}{32}\right)$ \\
(+ - -)&$2\sin\left(\frac{1}{8}\left(\arcsin(x)+\frac{5\pi}{2}\right)\right) $&$2\sin\left(\frac{11\pi}{32}\right)$ \\
(+ + -)&$2\sin\left(\frac{1}{8}\left(-\arcsin(x)+\frac{7\pi}{2}\right)\right) $&$2\sin\left(\frac{13\pi}{32}\right)$ \\
(+ + +)&$2\sin\left(\frac{1}{8}\left(\arcsin(x)+\frac{7\pi}{2}\right)\right) $&$2\sin\left(\frac{15\pi}{32}\right)$ \\
 \hline
\end{tabular}
\bigskip
\newline
The general correspondence is
$$2\sin\left(\frac{(2i-1)\pi}{2^n}\right)\to2\sin\left(\frac{1}{2^{n-2}}\left((-1)^{i}\arcsin(x)+\frac{\left(2\floor{\frac{(2i-1)}{4}}+1\right)\pi}{2}\right)\right)$$

In section 3 we introduced matricies $M_n$. We look at the functions corresponding to the $M_n$ of the type (for $n=4$)
\begin{equation}\notag
\left(
\begin{array}{c}
y_1(x)\\
y_2(x)\\
y_3(x)\\
y_4(x)\\
\end{array}
\right)
=\left(
\begin{array}{cccc}
 2 & 5 & 7 & 8 \\
 7 & 2 & -8 & 5 \\
 5 & 8 & 2 & -7 \\
 -8 & 7 & -5 & 2 \\
\end{array}
\right)\left(
\begin{array}{c}
\sqrt{2-\sqrt{2+2x}} \\
\sqrt{2-\sqrt{2-2x}}\\
\sqrt{2+\sqrt{2-2x}} \\
\sqrt{2+\sqrt{2+2x}}\\
\end{array}
\right)
\end{equation}
Call the vector on the right in equation the above, $v_4(x)$. 
Let 
\begin{equation}\notag
c_4=
\left(
\begin{array}{cccc}
1 & 1 & 1 & 1 \\
\end{array}
\right)
\end{equation}
Then
$$F_n(x)=\frac{1}{2^{3n-2}}c_n * M_n*v_n(x)$$
Let
\begin{equation}\notag
T_4=
\left(
\begin{array}{cccc}
0 &0& 0 & 1\\
0&0 &1 &0  \\
0 &1 &0 &0 \\
1 & 0 &0 &0 \\
\end{array}
\right)
\end{equation}
and
\begin{equation}\notag
Tm_4=
\left(
\begin{array}{cccc}
0 &0& 0 & -1\\
0&0 &1 &0  \\
0 &-1 &0 &0 \\
1 & 0 &0 &0 \\
\end{array}
\right)
\end{equation}
The matrices $T_m$ have $T_M(i,i)=(-1)^{2n-i}$, $T_M(i,j)=0$ for$i\ne j$.
Then
$$ M_n*v'_n(x)=\frac{1}{2^{n-2}\sqrt{1-x^2}}Tm_n*(M_n*v_n(x))$$
$$F'_n(x)=\frac{1}{2^{3n-2}}v1_n * M_n*v'_n(x)*v1_n^{t}=\frac{1}{2^{4(n-1)}\sqrt{1-x^2}}v1_n*Tm_n*(M_n*v_n(x))$$
When we differentiate an elementary function, through repeated application of the chain rule we pick up a factor of $1/2$ for each radical and a factor of 2 from the final $2x$. But to get it into our final form we multiply numerator and denominator by the conjugate and get an additional factor of $1/2$ when we simplify the denominator. If we let
\begin{equation}\notag
M_5(x)=\left(
\begin{array}{cccc}
 2 & 5 & 7 & 8 \\
 7 & 2 & -8 & 5 \\
 5 & 8 & 2 & -7 \\
 -8 & 7 & -5 & 2 \\
\end{array}
\right)\left(
\begin{array}{c c c c}
\sqrt{2-\sqrt{2+2x}} &0&0&0\\
0&\sqrt{2-\sqrt{2-2x}}&0&0\\
0&0&\sqrt{2+\sqrt{2-2x}}&0 \\
0&0&0&\sqrt{2+\sqrt{2+2x}}\\
\end{array}
\right)
\end{equation}
Then we have
$$F_n(x)=\frac{1}{2^{3n-2}}v1_n*M_n(x)*v1_n^{t}$$
$$ M_n'(x)=\frac{1}{2^{n-2}\sqrt{1-x^2}}Tm_n*(M_n(x))*T$$
$$F'_n(x)=\frac{1}{2^{3n-2}}v1_n*\frac{1}{2^{n-2}\sqrt{1-x^2}}Tm_n*(M_n(x))*T*v1_n^{t}$$

The equation
$$ Y'(x)=\frac{1}{2^{n-2}\sqrt{1-x^2}}Tm_n*Y(x)$$
for the $n=4$ case has general solution
\begin{equation}\notag
\left(
\begin{array}{c}
y_1(x)\\
y_2(x) \\
y_3(x) \\
y_4(x) \\
\end{array}
\right)=
\left(
\begin{array}{c }
c_1 \cos\left(\frac{\arcsin\left(x\right)}{4}\right) - c_2 \sin\left(\frac{\arcsin\left(x\right)}{4}\right)\\
c_3 \cos\left(\frac{\arcsin\left(x\right)}{4}\right) + c_4 \sin\left(\frac{\arcsin\left(x\right)}{4}\right)\\
c_4 \cos\left(\frac{\arcsin\left(x\right)}{4}\right) - c_3 \sin\left(\frac{\arcsin\left(x\right)}{4}\right)\\
c_2 \cos\left(\frac{\arcsin\left(x\right)}{4}\right) + c_1 \sin\left(\frac{\arcsin\left(x\right)}{4}\right)\\
\end{array}
\right)
\end{equation}
More generally,
$$y_k=c_k \cos\left(\frac{\arcsin\left(x\right)}{2^{n-2}}\right) - b_k \sin\left(\frac{\arcsin\left(x\right)}{2^{n-2}}\right)$$
Setting this equal to the $Y(x)$ above, evaluating at $x=0$ and $x=1$ we can resolve the constants.
\begin{equation}\notag
\left(
\begin{array}{c}
y_1(x)\\
y_2(x) \\
y_3(x) \\
y_4(x) \\
\end{array}
\right)=
\left(
\begin{array}{c }
\sqrt{548+386 \sqrt{2}}\cos\left(\frac{\arcsin\left(x\right)}{4}\right) -\left(2+\sqrt{2}\right)^{3/2} \sin\left(\frac{\arcsin\left(x\right)}{4}\right)\\
3 \sqrt{20-14 \sqrt{2}} \cos\left(\frac{\arcsin\left(x\right)}{4}\right) + \sqrt{2} \sqrt{194-137 \sqrt{2}} \sin\left(\frac{\arcsin\left(x\right)}{4}\right)\\
\sqrt{2} \sqrt{194-137 \sqrt{2}} \cos\left(\frac{\arcsin\left(x\right)}{4}\right) - 3 \sqrt{20-14 \sqrt{2}} \sin\left(\frac{\arcsin\left(x\right)}{4}\right)\\
\left(2+\sqrt{2}\right)^{3/2} \cos\left(\frac{\arcsin\left(x\right)}{4}\right) + \sqrt{548+386 \sqrt{2}} \sin\left(\frac{\arcsin\left(x\right)}{4}\right)\\
\end{array}
\right)
\end{equation}
\begin{align}\notag
\implies F_4(x)&=\frac{1}{2^{10}}\sum_{i=1}^{4}y_k(x)\\\notag
&=\frac{1}{2^{10}}\left(2 \left(\sqrt{50+31 \sqrt{2}} \sin \left(\frac{1}{4} \arcsin(x)\right)+\sqrt{100-34 \sqrt{2}} \cos \left(\frac{1}{4} \arcsin(x)\right)\right)\right)\\\notag
&=   \frac{1}{2^{8}} \left( \left(  4\cos\left(\frac{\pi}{8}\right)+3\sin\left(\frac{\pi}{8}\right)\right  ) \sin \left(\frac{1}{4} \arcsin(x)\right)+\left(  \cos\left(\frac{\pi}{8}\right)+7\sin\left(\frac{\pi}{8}\right)\right   )\cos \left(\frac{1}{4} \arcsin(x)\right)  \right) \\\notag
\end{align}
We repeat this process for $n=5$ and get
\begin{align}\notag
F_5(x)=&\frac{1}{2^{13}}\bigl[4 \Biggl(\sqrt{3858+1944 \sqrt{2}-\sqrt{3455234+2247233 \sqrt{2}}} \cos \left(\frac{1}{8} \arcsin(x)\right)\\\notag
&+\sqrt{5174-1856 \sqrt{2}+\sqrt{3238138+1993289 \sqrt{2}}} \sin \left(\frac{1}{8} \arcsin(x)\right)\Biggr)\bigr]\notag
\end{align}
Comparing $F_4(\sqrt{2}/2)$ with the formulas for $S_4$ given in section 3, we verify that
\begin{align}\notag
S_4&=\frac{1}{2^{9}}\sqrt{300-6 \sqrt{2}+\sqrt{302 \sqrt{2}+436}}\\
&= \frac{1}{2^{9}}\left(\sqrt{50+31 \sqrt{2}}\ \sqrt{2-\sqrt{2+2x}} +\sqrt{100-34 \sqrt{2}}\  \sqrt{2+\sqrt{2+2x}} \right)\\\notag
\end{align}
Similarily, comparing $F_5(\sqrt{2}/2)$ with the formulas for $S_5$ given in section 3,
\begin{align}\notag
S_5&=\frac{1}{2^{12}}\sqrt{2 \left(2 \left(4516+44 \sqrt{2}-\sqrt{3466-137 \sqrt{2}}\right)+\sqrt{428260-79990 \sqrt{2}+\sqrt{32962925198 \sqrt{2}+52633222484}}\right)}\\\notag
&=\frac{1}{2^{12}}\Biggl(\sqrt{3858+1944 \sqrt{2}-\sqrt{3455234+2247233 \sqrt{2}}}\quad  \sqrt{2+\sqrt{2+\sqrt{2+\sqrt{2}}}}\\\notag
&+\sqrt{5174-1856 \sqrt{2}+\sqrt{3238138+1993289 \sqrt{2}}}\quad  \sqrt{2-\sqrt{2+\sqrt{2+\sqrt{2}}}}\Biggr)\\\notag
\end{align}
The functions $F_n(x)$ can be written as

\begin{align}\notag
F_n(x)=&S1(n,x)=\frac{1}{2^{3n-2}} \sum _{j=1}^{2^{n-2}} \left(-j^2+2^{n-1}j+j-2^{n-2}\right) \sum _{i=1}^{2^{n-2}} (-1)^{\floor{ \frac{ (2 j i-i-j)}{ 2^{n-1}}}}\\\notag
&\sin \left(\frac{   (-1)^{i+j-1} \arcsin(x)+\left(2\floor{ \frac{1}{4} \left(2 ((2 j i-i-j) \pmod*{ 2^{n-1}})+1\right)} +1\right)\frac{\pi}{2} 
  }{2^{n-2}}\right)\\\notag
&=\frac{1}{2^{3n-2}} \sum _{j=1}^{2^{n-2}} \left(-j^2+2^{n-1}j+j-2^{n-2}\right) \sum _{i=1}^{2^{n-2}} (-1)^{\floor{ \frac{ (2 j i-i-j)}{ 2^{n-1}}}}\\\notag
&\Biggl[ \sin \left(     \frac{  (-1)^{i+j-1} }{2^{n-2}} \arcsin(x)   \right)    \cos\left(   \left(2\floor{      \frac{1}{4} \left(2 ((2 j i-i-j) \pmod*{ 2^{n-1}})+1    \right)       } +1   \right)\frac{\pi}{2^{n-1}}       \right)\\\notag
&+\cos\left(   \frac{   (-1)^{i+j-1} }{2^{n-2}}\arcsin(x)   \right)  \sin\left(  \left(   2\floor{ \frac{1}{4} \left(2 (  (2 j i-i-j) \pmod*{ 2^{n-1}}  )+1\right)      } +1   \right)\frac{\pi}{2^{n-1}} \right)  \Biggr]\\\notag
&=\frac{1}{2^{3n-2}} \sum _{j=1}^{2^{n-2}} \left(-j^2+2^{n-1}j+j-2^{n-2}\right) \sum _{i=1}^{2^{n-2}}  (-1)^{\floor{ \frac{ (2 j i-i-j)}{ 2^{n-1}}}}\\\notag
&\Biggl[ (-1)^{i+j-1} \sin \left(     \frac{ 1 }{2^{n-2}} \arcsin(x)   \right)    \cos\left(   \left(2\floor{      \frac{1}{4} \left(2 ((2 j i-i-j) \pmod*{ 2^{n-1}})+1    \right)       } +1   \right)\frac{\pi}{2^{n-1}}       \right)\\\notag
&+\cos\left(   \frac{  1 }{2^{n-2}}\arcsin(x)   \right)  \sin\left(  \left(   2\floor{ \frac{1}{4} \left(2 (  (2 j i-i-j) \pmod*{ 2^{n-1}}  )+1\right)      } +1   \right)\frac{\pi}{2^{n-1}} \right)  \Biggr]\\\notag
&=\frac{1}{2^{3n-2}} \sum _{j=1}^{2^{n-2}} \left(-j^2+2^{n-1}j+j-2^{n-2}\right) \sum _{i=1}^{2^{n-2}}  (-1)^{\floor{ \frac{ (2 j i-i-j)}{ 2^{n-1}}}}\\\notag
&\Biggl[ (-1)^{i+j-1} \sin \left(     \frac{ 1 }{2^{n-2}} \arcsin(x)   \right) \Biggl[   \cos\left(   \left(\floor{      \frac{1}{4} \left(2 ((2 j i-i-j) \pmod*{ 2^{n-1}})+1    \right)       }   \right)\frac{\pi}{2^{n-2}}\right) \cos\left(\frac{\pi}{2^{n-1}}\right)      \\\notag
&- \sin\left(   \left(\floor{      \frac{1}{4} \left(2 ((2 j i-i-j) \pmod*{ 2^{n-1}})+1    \right)       }   \right)\frac{\pi}{2^{n-2}}\right) \sin\left(\frac{\pi}{2^{n-1}}\right)   \Biggr]\\\notag
&+\cos\left(   \frac{  1 }{2^{n-2}}\arcsin(x)   \right) \Biggl[ \sin\left(  \left(   \floor{ \frac{1}{4} \left(2 (  (2 j i-i-j) \pmod*{ 2^{n-1}}  )+1\right)      }    \right)\frac{\pi}{2^{n-2}} \right)  \cos\left(\frac{\pi}{2^{n-1}}\right)\\\notag
&+\cos\left(  \left(   \floor{ \frac{1}{4} \left(2 (  (2 j i-i-j) \pmod*{ 2^{n-1}}  )+1\right)      }    \right)\frac{\pi}{2^{n-2}} \right)  \sin\left(\frac{\pi}{2^{n-1}}\right)\Biggr]\Biggr]\\\notag
\end{align}
Note we have the explicit formulas for the constants in (19) above:
\begin{align}\notag
F_n(x)&=\frac{1}{2^{3n-2}}\Biggl[\sin \left(     \frac{ 1 }{2^{n-2}} \arcsin(x)   \right)  \sum _{j=1}^{2^{n-2}} \left(-j^2+2^{n-1}j+j-2^{n-2}\right) \sum _{i=1}^{2^{n-2}}  (-1)^{\floor{ \frac{ (2 j i-i-j)}{ 2^{n-1}}}}\\\notag
&\Biggl[ (-1)^{i+j-1}    \cos\left(   \left(2\floor{      \frac{1}{4} \left(2 ((2 j i-i-j) \pmod*{ 2^{n-1}})+1    \right)       } +1   \right)\frac{\pi}{2^{n-1}}       \right)\Biggr]\\\notag
&+\cos\left(   \frac{  1 }{2^{n-2}}\arcsin(x)   \right)  \sum _{j=1}^{2^{n-2}} \left(-j^2+2^{n-1}j+j-2^{n-2}\right) \sum _{i=1}^{2^{n-2}}  (-1)^{\floor{ \frac{ (2 j i-i-j)}{ 2^{n-1}}}}\\\notag
&\Biggl[\sin\left(  \left(   2\floor{ \frac{1}{4} \left(2 (  (2 j i-i-j) \pmod*{ 2^{n-1}}  )+1\right)      } +1   \right)\frac{\pi}{2^{n-1}} \right)  \Biggr]\Biggr]\\\notag
\end{align}
Note that in  $\lim_{n\to\infty}\frac{8\pi^3}{7}F_n\left(\frac{\sqrt{2}}{2}\right)$, with $\sin \left(     \frac{ 1 }{2^{n-2}} \arcsin(x)   \right) \to 0$ and $\cos\left(   \frac{  1 }{2^{n-2}}\arcsin(x)   \right)\to 1$, only the 2nd half of the above formula survives, giving us the same  formula presented in Theorem 2.9.

The functions we denoted as $y_k(x)$ above associated with $M_n$ can be written as
\begin{align}\notag
y_k(n,x)=&\left(2^{3n-2}\right)F(n,k,x)= \sum _{j=1}^{2^{n-2}} \left(-j^2+2^{n-1}j+j-2^{n-2}\right)   (-1)^{\floor{ \frac{ (2 j i-i-j)}{ 2^{n-1}}}}\\\notag
&\sin \left(\frac{   (-1)^{i+j-1} \arcsin(x)+\left(2\floor{ \frac{1}{4} \left(2 ((2 j k-k-j) \pmod*{ 2^{n-1}})+1\right)} +1\right)\frac{\pi}{2}  }{2^{n-2}}\right) \\\notag
\end{align}
We know (see below) 
\begin{align}
\lim_{n\to\infty}\frac{1}{2^{3n-2}}y_k(n,x)|_{x=\frac{\sqrt{2}}{2}}=\frac{1}{((2k-1)\pi)^3}\notag
\end{align}
We also have
$$y_k(n,x)=c_k(n) \cos\left(\frac{\arcsin\left(x\right)}{2^{n-2}}\right) - b_k(n) \sin\left(\frac{\arcsin\left(x\right)}{2^{n-2}}\right)\notag$$
It turns out that
$$\lim_{n\to\infty}\frac{1}{2^{3n-2}}y_k(n,x)=\left(\frac{1}{8}\right)\left(\frac{(2k-1)\pi}{2}\right)^{-3}$$
And, as expected,
$$\implies F(x)=\sum_{i=1}^{\infty}\left(\frac{1}{8}\right)\left(\frac{(2k-1)\pi}{2}\right)^{-3}$$
We have an alternate form where the $i,j$ term corresponds to the value in the $(i,j)$ position in the matrix. In the previous formula,  the $i$ corresponds to the row but the $j$ does not correspond to the matrix column (with the exception of the first row, where "j" does correspond to the matrix column).
\begin{align}\notag
F_n(x)= &\sum _{j=1}^{2^{n-2}}  \sum _{i=1}^{2^{n-2}} (-1)^{  \floor{   \frac{((i+j-1)(2i-1)^{(2^{n-2}-1)} }{2^{n-1}}   }     }\Biggl[ \Bigl[  (2^{n-1}+1)\left((i+j-1)(2i-1)^{2^{n-2}-1}\pmod*{2^{n-1}}\right)\\\notag
&-    \left( (i+j-1)(2i-1)^{2^{n-2}-1}\pmod*{2^{n-1}}\right)^{2}-2^{n-2}  \Bigr]\sin\left(\frac{1}{2^{n-2}}\left( (-1)^j\arcsin(x)+\frac{\pi}{2}\left(2\floor{\frac{2j-1}{4}}+1\right)  \right)\right)\Biggr]\\\notag
= &\sum _{j=1}^{2^{n-2}}  \sum _{i=1}^{2^{n-2}} (-1)^{  \floor{   \frac{((i+j-1)(2i-1)^{(2^{n-2}-1)} }{2^{n-1}}   }     }\Biggl[        \Bigl[    (2^{n-1}+1)\left((i+j-1)(2i-1)^{2^{n-2}-1}\pmod*{2^{n-1}}\right)\\\notag
&-    \left( (i+j-1)(2i-1)^{2^{n-2}-1}\pmod*{2^{n-1}}\right)^{2}-2^{n-2}  \Bigr      ]\\\notag
&\Biggl(   (-1)^j    \sin\left(\frac{1}{2^{n-2}} \arcsin(x)\right)      \cos\left(\frac{\pi}{2^{n-1}}\left(2\floor{\frac{2j-1}{4}}+1\right)\right)  \\\notag
&+\cos\left(\frac{1}{2^{n-2}} \arcsin(x)\right)   \sin\left(\frac{\pi}{2^{n-1}}\left(2\floor{\frac{2j-1}{4}}+1\right) \right)      \Biggr)                 \Biggr]\\\notag
\end{align}
The alternate Mathematica code corresponds to
\begin{align}\notag
S_n=\frac{1}{2^{3n-2}}& \sum _{j=1}^{2^{n-2}}  \sum _{i=1}^{2^{n-2}} (-1)^{  \floor{   \frac{((i+j-1)(2i-1)^{(2^{n-2}-1)} }{2^{n-1}}   }     }\Biggl[ \Bigl[  (2^{n-1}+1)\left((i+j-1)(2i-1)^{2^{n-2}-1}\pmod*{2^{n-1}}\right)\\\notag
&-    \left( (i+j-1)(2i-1)^{2^{n-2}-1}\pmod*{2^{n-1}}\right)^{2}-2^{n-2}  \Bigr]\sin\left(\frac{(2j-1)\pi}{2^{n}}\right)\Biggr]\\\notag
=\frac{1}{2^{3n-2}} &\sum _{j=1}^{2^{n-2}}\sin\left(\frac{(2j-1)\pi}{2^{n}}\right)  \sum _{i=1}^{2^{n-2}} (-1)^{  \floor{   \frac{((i+j-1)(2i-1)^{(2^{n-2}-1)} }{2^{n-1}}   }     }\Biggl[ \Bigl[  (2^{n-1}+1)\left((i+j-1)(2i-1)^{2^{n-2}-1}\pmod*{2^{n-1}}\right)\\\notag
&-    \left( (i+j-1)(2i-1)^{2^{n-2}-1}\pmod*{2^{n-1}}\right)^{2}-2^{n-2}  \Bigr]\Biggr]\\\notag
\end{align}
We can average the values of the term $\sin\left(\frac{(2j-1)\pi}{2^{n}}\right)$ in a neighborhood of the point:
$$\frac{2j-1}{4}\int_{\frac{2j-3}{2j-1}}^{\frac{2j+1}{2j-1}}\sin\left(\frac{\pi(2j-1)x}{2^n}\right)=\frac{2^{n-2}}{\pi}\left(\cos\left(\frac{(2j-3)\pi}{2^n}\right)-\cos\left(\frac{(2j+1)\pi}{2^n}\right)\right)$$
to obtain
\begin{align}\notag
\hat{S}_n&
=\frac{1}{2^{2n}\pi} \sum _{j=1}^{2^{n-2}}\left(\cos\left(\frac{(2j-3)\pi}{2^n}\right)-\cos\left(\frac{(2j+1)\pi}{2^n}\right)\right)\\\notag
& \sum _{i=1}^{2^{n-2}} (-1)^{  \floor{   \frac{((i+j-1)(2i-1)^{(2^{n-2}-1)} }{2^{n-1}}   }     }\Biggl[ \Bigl[  (2^{n-1}+1)\left((i+j-1)(2i-1)^{2^{n-2}-1}\pmod*{2^{n-1}}\right)\\\notag
&-    \left( (i+j-1)(2i-1)^{2^{n-2}-1}\pmod*{2^{n-1}}\right)^{2}-2^{n-2}  \Bigr]\Biggr]\\\notag
&=\frac{1}{2^{2n-1}\pi}\sin\left(\frac{\pi}{2^{n-1}}\right) \sum _{j=1}^{2^{n-2}}\sin\left(\frac{(2j-1)\pi}{2^n}\right)\\\notag
& \sum _{i=1}^{2^{n-2}} (-1)^{  \floor{   \frac{((i+j-1)(2i-1)^{(2^{n-2}-1)} }{2^{n-1}}   }     }\Biggl[ \Bigl[  (2^{n-1}+1)\left((i+j-1)(2i-1)^{2^{n-2}-1}\pmod*{2^{n-1}}\right)\\\notag
&-    \left( (i+j-1)(2i-1)^{2^{n-2}-1}\pmod*{2^{n-1}}\right)^{2}-2^{n-2}  \Bigr]\Biggr]\\\notag
\end{align}
\section{{\textbf{Nested Root Function Identities}}}
As $n$ grows, the chain of nested roots in each term of $S_n$ grows. In the limit as $n$ goes to infinity, this becomes an infinite nesting of roots. It is analogous to a continued fraction. We shall call these continued roots. The values of the individual terms of the sum become dense in $[0,1]$ and we see that every real number has a representation as a continued root involving only the coefficient $2$.
Given a recognizable pattern in the distribution of + and - in a continued root, we can use tricks of continued fractions to determine the value. For example, if
\begin{equation}
x=\sqrt{2+\sqrt{2-\sqrt{2-\sqrt{2-\sqrt{2-\sqrt{2-\sqrt{2-\ldots}}}}}}}\notag
\end{equation}
then 
\begin{equation}
x^2-2=2-\left(x^2-2\right)^2\notag
\end{equation}
implying
$x=0$ or $\sqrt{3}$
To return this value to the context of our $sin$ functions, we should divide this by 2 yielding $0$ or $\frac{\sqrt{3}}{2}$, i.e., $\sin(0)$ or $\sin(\frac{\pi}{3})$.
\end{flushleft}

Using this reasoning, we determine
\begin{equation}
\sqrt{2+\sqrt{2+\sqrt{2+\sqrt{2+\sqrt{2+\sqrt{2+\sqrt{2+\ldots}}}}}}}=2\notag
\end{equation}
\begin{equation}
\sqrt{2-\sqrt{2-\sqrt{2-\sqrt{2-\sqrt{2-\sqrt{2-\sqrt{2-\ldots}}}}}}}=1\notag
\end{equation}

We see that
\begin{equation}
\sqrt{2-\sqrt{2+\sqrt{2+\sqrt{2+\sqrt{2+\sqrt{2+\sqrt{2+\ldots}}}}}}}=0\notag
\end{equation}
Hence we can turn any finite continued expansion into an infinte expansion. E.g.,
\begin{align*}\notag
\sqrt{2+\sqrt{2-\sqrt{2}}}&=\sqrt{2+\sqrt{2-\sqrt{2-\sqrt{2-\sqrt{2+\sqrt{2+\sqrt{2+\ldots}}}}}}}\\\notag
&=\sqrt{2+\sqrt{2-\sqrt{2+\sqrt{2-\sqrt{2+\sqrt{2+\sqrt{2+\ldots}}}}}}}\notag
\end{align*}
So, representations are not unique, much as $2.0=1.99999999\ldots$.

Going back to our original definition of $S_n$, it seems natural to consider functions of the form
\begin{align*}\notag
S_5(x)&=\frac{1}{2^{12}}\Big[\frac{1}{\left(2-\sqrt{2-\sqrt{2-2x}}\right)^{3/2}}+\frac{1}{\left(2+\sqrt{2-\sqrt{2-2x}}\right)^{3/2}}+\frac{1}{\left(2-\sqrt{2+\sqrt{2-2x}}\right)^{3/2}}\\&+\frac{1}{\left(2+\sqrt{2+\sqrt{2-2x}}\right)^{3/2}}+\frac{1}{\left(2-\sqrt{2-\sqrt{2+2x}}\right)^{3/2}}+\frac{1}{\left(2+\sqrt{2-\sqrt{2+2x}}\right)^{3/2}}\\&+\frac{1}{\left(2-\sqrt{2+\sqrt{2+2x}}\right)^{3/2}}+\frac{1}{\left(2+\sqrt{2+\sqrt{2+2x}}\right)^{3/2}}\Big]\\\notag
\end{align*}
and to let
\begin{equation}
S(x)=\lim_{n->\infty}S_n(x)\notag
\end{equation}
Note that
\begin{equation}
S(x)=(2^{-3})(S(\frac{\sqrt{2+2x}}{2})+S(\frac{\sqrt{2-2x}}{2}))\notag
\end{equation}
since
\begin{equation}
 S_{n+1}(x)=(2^{-3})(S_n(\frac{\sqrt{2+2x}}{2})+S_n(\frac{\sqrt{2-2x}}{2}))\notag
\end{equation}
\begin{equation}
\implies \lim_{n->\infty}(2^3)S_{n+1}=\lim_{n->\infty}(S_n(\frac{\sqrt{2+2x}}{2})+S_n(\frac{\sqrt{2-2x}}{2}))\notag
\end{equation}
Similarly, (11) implies
\begin{align*}
S(\frac{\sqrt{2+2x}}{2})=(2^{-3})(S(\frac{\sqrt{2+\sqrt{2+2x}}}{2})+S(\frac{\sqrt{2-\sqrt{2+2x}}}{2}))\notag
\end{align*}
and
\begin{equation}
S(x)=(2^{-6})(S(\frac{\sqrt{2+\sqrt{2+2x}}}{2})+S(\frac{\sqrt{2+\sqrt{2-2x}}}{2})+S(\frac{\sqrt{2-\sqrt{2+2x}}}{2})+S(\frac{\sqrt{2-\sqrt{2-2x}}}{2}))\notag
\end{equation}
also,
\begin{equation}
S(x)=8 S\left(1-2 x^2\right)-S\left(\sqrt{1-x^2}\right)\notag
\end{equation}
We see that 

\begin{align*}\notag
\zeta(3)&=\left(\frac{8\pi^3}{7}\right) S(\frac{\sqrt{2}}{2})\\\notag
&=\left(\frac{8\pi^3}{7}\right)\lim_{n\to\infty} S_n(\frac{\sqrt{2}}{2})\\\notag
&=\left(\frac{8\pi^3}{7}\right) \lim_{n\to\infty} \sum _{i=1}^{2^{n-2}}(2^n \sin\left(\frac{\left(2 i-1\right)\pi}{2^n}\right) )^{-3}\notag
\end{align*}
Replacing $x$ by $-x$ in (11) shows S is an even function.\\
Also,\\
\begin{align*}\notag
S(0)=\frac{1}{2^2}S(\frac{1}{\sqrt{2}})=\frac{1}{2^{-5}}(S(\frac{\sqrt{2+\sqrt{2}}}{2})+S(\frac{\sqrt{2-\sqrt{2}}}{2}))\\\notag
S_{n-1}(\frac{1}{\sqrt{2}})=4S_n(0)\implies \zeta(3)=\left(\frac{32\pi^3}{7}\right) S(0)\notag
\end{align*}
Each $S_n$ has singularities at $\pm 1$ corresponding to branch points. Using L'Hopital's rule,
\begin{align}\notag
&\lim_{x\rightarrow 1}\frac{2^{3(n-1)}S_n}{\frac{1}{(1-x)^2}}=0\\\notag
&\lim_{x\rightarrow -1}\frac{2^{3(n-1)}S_n}{\frac{1}{(1+x)^2}}=0\notag
\end{align}
Hence, $S$ has similar singularities at $\pm 1$ and is not globally well defined. There is a neighborhood of $0$ where $S$ is analytic.
We see
\small
\begin{equation}
\frac{8}{\left(2-\sqrt{2-\sqrt{2-\sqrt{2-2 x}}}\right)^{3/2}}=\frac{\left(2+\sqrt{2-\sqrt{2-\sqrt{2-2 x}}}\right)^{3/2} \left(2-\sqrt{2-\sqrt{2-2 x}}\right)^{3/2} \left(2-\sqrt{2-2 x}\right)^{3/2} (2-2 x)^{3/2}}{\left(1-x^2\right)^{3/2}}\notag
\end{equation}
Each term comprising the sum can be put into this form with denominator $(1-x^2)^{3/2}$. Note the function $(1-x^2)^{3/2}S_n$ has the same value at $0$ as $S_n$ but is locally bounded everywhere.. Hence $(1-x^2)^{3/2}S$ has no poles.\\
\begin{definition}
The individual terms of $S_n(x)$ will be refered to as elementary functions. As in Section 2, we will use a shorthand notation to refer to elementary funtions, using square brackets instead of round. E.g., [+---] would refer to $$\left(\frac{1}{2^{15}}\right)\frac{1}{\left(2+\sqrt{2-\sqrt{2-\sqrt{2-2 x}}}\right)^{3/2}}$$  Given an elementary function f, The nth conjugate of f has the same pattern of signs except that the sign in the nth position is reversed. The fist conjugate will be referred to as the conjugate and denoted $\bar f$. Note that the set of elementary functions for $S_n$ can be divided into two parts, those with positive sign in the nth position and those with negative sign in the nth position, nth conjugates. In formulas given below involving these functions $n$ refers to the term being a summand in $S_n(x)$. Note, The number of signs in the notation for an element in $S_n(x)$ is $n-2$, the depth is $n-2$ also. In $S_n$ (values, not functions) the number of signs is the same but the depth is $n-1$.
\end{definition}
\begin{proposition} 
$$S(x)=\frac{1}{8}\sum  _{i=1}^{\infty} \left(\frac{1}{\left(\pi  (2 i-1)-2 \arcsin(x)\right)^3}+\frac{1}{\left(\pi  (2 i-1)+2 \arcsin(x)\right)^3}\right)$$
\end{proposition}
\begin{proof}
We find that for terms of $S_n(x)$ (elementary terms without the leading factor of  $\frac{1}{2^{3n}}$)\\
$$f'= \textrm{sign(f)}\frac{3 f^{4/3}}{2^{n-2} \sqrt{1-x^2} {\bar{f}}^{1/3}}$$
But,
$$\bar{f}=\frac{f}{\left(4 f^{2/3}-1\right)^{3/2}}$$
$$\implies f'= \text{sign(f)}\frac{3}{2^{n-2}} f \sqrt{\frac{4 f^{2/3}-1}{1-x^2}}$$
If we distribute the factor $\frac{1}{2^{3n}}$ that ocurrs in $S_n$ to the individual terms, let $g=(1/2^{3n})f$, then
$$ g'= \textrm{sign(g)}\frac{3}{2^{n-2}} g \sqrt{\frac{2^{2n} g^{2/3}-1}{1-x^2}}$$
Asymptotically. as n goes to infinity, this looks like
$$g'=\textrm{sign(g)}12\frac{ g^{4/3}}{\sqrt{1-x^2}}$$
The next question is, what function results from letting the depth of a term like $e[x]=$ [-++++--] go to infinity. The previous result gives us the differential equation for such a limit, provided we extend the depth properly. That equation has solution
$$k(x)=-\textrm{sign}(e)\frac{27}{\left(12 \arcsin(x)+c_1\right){}^3}$$
We pause for a moment to point out how remarkable it is that all these limits of elementary functions are given here. For example,if we extend the depth of $e[x]=$ [-++++--] by adding + signs to the middle, and take that to the limit, the resulting function is 
$$k(x)=\frac{27}{\left(18\pi-12  \arcsin(x)\right){}^3}=\frac{1}{8\left(3\pi-2 \arcsin(x)\right){}^3}$$

Here is an extension of the earlier table to help extending the depth of these functions:

\begin{center}
\begin{tabular}{|c c c c c c |} 
 \hline
$\pi/512$ &$3\pi/512$ &$5\pi/512$ &$7\pi/512$ &$9\pi/512$ &$11\pi/512$ \\ 
 \hline
(-++++++)&(-+++++-)&(-++++--)&(-++++-+)&(-+++--+)&(-+++---) \\
 \hline
\end{tabular}
\bigskip
\begin{tabular}{|c c c c c c |} 
 \hline
$63\pi/512$ &$127\pi/512$ &$191\pi/512$&$251\pi/512$&$253\pi/512$ &$255\pi/512)$\\ 
 \hline
(-+-++++)& (--+++++)& (+--++++)&(+++++--)&(++++++-)&(+++++++) \\
 \hline
\end{tabular}\\

\end{center}
Note that $$\lim_{n\to\infty}\frac{1}{2^{n}\sin(\frac{2i-1)\pi}{2^n})}=\frac{1}{(2i-1)\pi}$$ Hence the function corresponding to that limit would be expected to take the value $(2i-1)\pi$ at $x=\sqrt{2}/2$.
In  the following table, the first column shows the pattern of infinite nesting, the second column gives the corresponding function, and the third column gives the function value at $x=\sqrt{2}/2$:\\
\begin{center}
\begin{tabular}{|c c c|} 
 \hline
[-+++++...+]&$\frac{1}{8 \left(\pi -2 \arcsin(x)\right)^3}$&$\frac{1}{\pi^3}$\\
 \hline
[-++++...+-]&$\frac{1}{8 \left(\pi +2 \arcsin(x)\right)^3}$&$\frac{1}{(3\pi)^3}$\\
 \hline
[-+++...+--]&$\frac{1}{8 \left(3\pi -2 \arcsin(x)\right)^3}$&$\frac{1}{(5\pi)^3}$\\
 \hline
[-+++...+-+]&$\frac{1}{8 \left(3\pi +2 \arcsin(x)\right)^3}$&$\frac{1}{(7\pi)^3}$\\
 \hline
[-+++...+--+]&$\frac{1}{8 \left(5\pi -2 \arcsin(x)\right)^3}$&$\frac{1}{(9\pi)^3}$\\
\hline
[-+++...+---]&$\frac{1}{8 \left(5\pi +2 \arcsin(x)\right)^3}$&$\frac{1}{(11\pi)^3}$\\
\hline
[-+++...+-+-]&$\frac{1}{8 \left(7\pi -2 \arcsin(x)\right)^3}$&$\frac{1}{(13\pi)^3}$\\
\hline
\end{tabular}
\end{center}
\end{proof}
Now we can write

$$S(x)=\frac{1}{8}\sum _{i=1}^{\infty} \left(\frac{1}{\left(\pi  (2 i-1)-2 \arcsin(x)\right)^3}+\frac{1}{\left(\pi  (2 i-1)+2 \arcsin(x)\right)^3}\right)$$
\begin{proposition}
There is a family of functions $A_n(x)$, $n\in\mathbb{Z}$, $n>0$
$$A_n(x)=\frac{\pi^n}{2^n-1}\sum  _{i=1}^{\infty} \left(\frac{1}{\left(\pi  (2 i-1)-2 \arcsin(x)\right)^n}+\frac{1}{\left(\pi  (2 i-1)+2 \arcsin(x)\right)^n}\right)$$
such that $$A_n(\frac{\sqrt{2}}{2})=\zeta(n)$$
\end{proposition}
\begin{proof}

\begin{align}\notag
&\frac{\pi^n}{2^n-1}\sum  _{i=1}^{\infty} \left(\frac{1}{\left(\pi  (2 i-1)-2 \arcsin(x)\right)^n}+\frac{1}{\left(\pi  (2 i-1)+2 \arcsin(x)\right)^n}\right)|_{x=\sqrt{2}/2}\\\notag
&=\frac{\pi^n}{2^n-1}\sum  _{i=1}^{\infty} \left(\frac{1}{ \left(\pi  (2 i-1)-2(\pi/4)\right)^n }   +\frac{1}{\left(\pi  (2 i-1)+2 (\pi/4)\right)^n} \right)\\\notag
&=\frac{\pi^n}{2^n-1}\sum  _{i=1}^{\infty} \left(\frac{2^n}{\left(\pi (4 i-3)\right)^n}+\frac{2^n}{\left(\pi  (4 i-1)\right)^n} \right)\\\notag
&=\frac{\pi^n}{2^n-1}\sum  _{i=1}^{\infty} \frac{2^n}{\left(\pi (2 i-1)\right)^n}\\\notag
&=\frac{2^n}{2^n-1}\sum  _{i=1}^{\infty} \frac{1}{ (2 i-1)^n}\\\notag
\end{align}
\end{proof}
Of course, this proposition is trivial. The unexpected part is the fact that these functions are equivalent to their nested root counterparts. Is there more to be learned from Taylor series expansions?
\begin {definition}
For $Re(s)>1$
 $$ \Sigma_s(x)=\sum  _{i=1}^{\infty} \left(\frac{1}{\left(\pi  (2 i-1)-2 \arcsin(x)\right)^s}+\frac{1}{\left(\pi  (2 i-1)+2 \arcsin(x)\right)^s}\right)$$
 Note: $$S(x)=\frac{1}{8}\Sigma_3(x)$$
\end{definition}
\begin{proposition}For $Re(s)>1$
$$\zeta(s)=\frac{\pi^s}{2^s-1}\Sigma_s(\frac{\sqrt{2}}{2})$$\notag
\end{proposition}
\begin{proof}
The proof of the previous proposition actually proves this for $Re(s)>1$
\end{proof}
\begin{proposition}\label{sigmazeta}
For $Re(s)>1$
\begin{align}\notag
\Sigma_{s}(x)&=\frac{1}{2^{ s}\pi^{s}}(\zeta(s,\frac{1}{2}-\frac{1}{\pi}\arcsin(x))+\zeta(s,\frac{1}{2}+\frac{1}{\pi}\arcsin(x))\\\notag
\end{align}
\end{proposition}
\begin{proof}
\begin{align}\notag
 \Sigma_s(x)&=\sum  _{i=1}^{\infty} \left(\frac{1}{\left(\pi  (2 i-1)-2 \arcsin(x)\right)^s}+\frac{1}{\left(\pi  (2 i-1)+2 \arcsin(x)\right)^s}\right)\\\notag
 &=\frac{1}{2^s\pi^s}\sum  _{i=1}^{\infty} \left(\frac{1}{\left(i  +\frac{1}{2}- \frac{1}{\pi}\arcsin(x)\right)^s}+\frac{1}{\left(i   +\frac{1}{2}+ \frac{1}{\pi}\arcsin(x)\right)^s}\right)\\\notag
 \end{align}
\end{proof}
Note the previous proposition gives us the analytic continuation of $\Sigma_n(x)$.
\begin{proposition}
For $Re(s)<1)$
\begin{align}\notag
\Sigma_s(z)&=\frac{1}{2^s\pi^{s/2}\Gamma(\frac{s}{2})}\int_0^\infty (\Theta(\frac{1}{2}-\frac{1}{\pi}\arcsin(z),it)-1)t^{-(s+1)/2}dt\\\notag
\end{align}
where $\Theta$ is the Jacobi theta function.
\end{proposition}
\begin{proof}
We have the following well known result for $Re(s)>0$
\begin{align}\notag
\int_0^\infty (\Theta(z,it)-1)t^{s/2}\frac{dt}{t}&=\pi^{-(1-s)/2}\Gamma(\frac{1-s}{2})(\zeta(1-s,z)+\zeta(1-s,1-z))\\\notag
\end{align}
combined with the previous proposition gives the desired result.
\end{proof}
\begin{lemma}
For $n\in \mathbb{Z}^+$
\begin{align}\notag
\sin(n \theta)&=\sum_{r=0}^{\floor{(n-1)/2}}(-1)^r\binom{n}{2r+1}\cos^{n-2r-1}(\theta)\sin^{2r+1}(\theta)\\\notag
\cos(n \theta)&=\sum_{r=0}^{\floor{(n)/2}}(-1)^r\binom{n}{2r}\cos^{n-2r}(\theta)\sin^{2r}(\theta)\\\notag
\end{align}
\end{lemma}
\begin{proof}
Note that it is equivalent to say
\begin{align}\notag
\sin(n \theta)&=\sum_{r=0}^{n}(-1)^r\binom{n}{2r+1}\cos^{n-2r-1}(\theta)\sin^{2r+1}(\theta)\\\notag
\cos(n \theta)&=\sum_{r=0}^{n}(-1)^r\binom{n}{2r}\cos^{n-2r}(\theta)\sin^{2r}(\theta)\\\notag
\end{align}
Since, in the equation for $\sin(n \theta)$,  $r>\floor{(n-1)/2}\implies \binom{n}{2r+1}=0$ since then $2r+1>n$ and $r\in \mathbb{Z}$, similarly for $\cos(n \theta)$. The proof then procedes by induction. True for $n=1$. Assume true for $n$ and show for $n+1$.
\begin{align}\notag
\sin((n+1) \theta)&=\sin(n\theta+\theta)
&=\sin(n\theta)\cos(\theta)+\cos(n\theta)\sin(\theta)\\\notag
&=\sum_{r=0}^{n}(-1)^r\binom{n}{2r+1}\cos^{n-2r-1}(\theta)\sin^{2r+1}(\theta)\cos(\theta)+\sum_{r=0}^{n}(-1)^r\binom{n}{2r}\cos^{n-2r}(\theta)\sin^{2r}(\theta)\sin(\theta)\\\notag
&=\sum_{r=0}^{n}(-1)^r\binom{n}{2r+1}\cos^{n-2r}(\theta)\sin^{2r+1}(\theta)+\sum_{r=0}^{n}(-1)^r\binom{n}{2r}\cos^{n-2r}(\theta)\sin^{2r+1}(\theta)\\\notag
&=\sum_{r=0}^{n+1}(-1)^r\binom{n+1}{2r}\cos^{n-2r}(\theta)\sin^{2r+1}(\theta)\\\notag
\end{align}
by Pascal's identity.
\end{proof}
Just as Newton's generalized binomial; theorem follows from the binomial theorem, we have
\begin{corollary}
For $r\in \mathbb{C}$ with $|\cos(\theta)|>|\sin(\theta)|$
\begin{align}\notag
\sin(r \theta)&=\sum_{j=0}^{\infty}(-1)^j\binom{r}{2j+1}\cos^{r-2j-1}(\theta)\sin^{2j+1}(\theta)\\\notag
\cos(r \theta)&=\sum_{j=0}^{\infty}(-1)^j\binom{r}{2j}\cos^{r-2j}(\theta)\sin^{2j}(\theta)\\\notag
\end{align}
\begin{proof}
Note that if $f(\theta)=\sin(r \theta)$ and $g(\theta)=\cos(r \theta)$ then
\begin{align}\notag
f'&=r g\\\notag
g'&=-r f\\\notag
\end{align}
$f$ and $g$ both satisfy
$$f''=-r^2 f$$
The general solution to this differential equation is 
$$c_1 e^{i r \theta}+c_2 e^{-ir\theta}$$
with constants resolved by evaluation at specific points.
One shows that the series expansions also satisy this equation.
\begin{align}\notag
&\frac{d}{d\theta}\sum_{j=0}^{\infty}(-1)^j\binom{r}{2j+1}\cos^{r-2j-1}(\theta)\sin^{2j+1}(\theta)\\\notag
&\ =\sum_{j=0}^{\infty}(-1)^j(\binom{r}{2j+1}(-(r-2j-1)\cos^{r-2j-2}(\theta)\sin^{2j+2}(\theta)+(2j+1)\cos^{r-2j}(\theta)\sin^{2j}(\theta))\\\notag
&\ =\sum_{j=0}^{\infty}(-1)^j( r\binom{r-1}{2j+1}\cos^{r-2j-2}(\theta)\sin^{2j+2}(\theta)+r\binom{r-1}{2j}\cos^{r-2j}(\theta)\sin^{2j}(\theta))\\\notag
&\ =r\sum_{j=0}^{\infty}(-1)^j(\binom{r-1}{2j-1}\cos^{r-2j}(\theta)\sin^{2j}(\theta)+\binom{r-1}{2j}\cos^{r-2j}(\theta)\sin^{2j}(\theta))\\\notag
&\ =r\sum_{j=0}^{\infty}(-1)^j\left(\binom{r-1}{2j-1}+\binom{r-1}{2j}\right)\cos^{r-2j}(\theta)\sin^{2j}(\theta)\\\notag
&\ =r\sum_{j=0}^{\infty}(-1)^j\binom{r}{2j}\cos^{r-2j}(\theta)\sin^{2j}(\theta)\\\notag
\end{align}
The fourth line in the proof results from re-indexing $j-1\to j$ and the fact that $\binom{r-1}{-1}=0$. The last line results from Pascal's identity. The proof for the cosine expansion is similar. We require  $|\cos(\theta)|>|\sin(\theta)|$ to ensure convergence.
\end{proof}
\end{corollary}
\begin{corollary}
Assume  $|\cos(\theta)|>|\sin(\theta)|$, then
\begin{align}\notag
\frac{1}{\cos^r(\theta)}&=\frac{1}{\sin(r \theta)}\sum_{j=0}^{\infty}(-1)^j\binom{r}{2j+1}\cos^{-2j-1}(\theta)\sin^{2j+1}(\theta)\\\notag
&=\frac{1}{\cos(r \theta)}\sum_{j=0}^{\infty}(-1)^j\binom{r}{2j}\cos^{-2j}(\theta)\sin^{2j}(\theta)\\\notag
\end{align}
\end{corollary}
\begin{corollary}
Assume  $|\sin(\theta)|>|\cos(\theta)|$, then
\begin{align}\notag
\frac{1}{\sin^r(\theta)}&=\frac{1}{\sin(r(\frac{\pi}{2}- \theta))}\sum_{j=0}^{\infty}(-1)^j\binom{r}{2j+1}\sin^{-2j-1}(\theta)\cos^{2j+1}(\theta)\\\notag
&=\frac{1}{\cos(r(\frac{\pi}{2}- \theta))}\sum_{j=0}^{\infty}(-1)^j\binom{r}{2j}\sin^{-2j}(\theta)\cos^{2j}(\theta)\\\notag
\end{align}

\end{corollary}

\begin{proposition}
For  $r\in \mathbb{C}$, $0\le\theta\le\frac{\pi}{2}$
\begin{align}\notag
\frac{1}{\sin^r(\theta)}&=2^{r/2}\sum_{j=0}^\infty (-1)^j \binom{-r/2}{j}\cos^j(2\theta)\\\notag
\end{align}
\end{proposition}
\begin{proof}
\begin{align}\notag
\sin(\theta)&=\sqrt{\frac{1-\cos(2\theta)}{2}}\\\notag
\implies \\\notag
\frac{1}{\sin^r}(\theta)&=\left(\frac{1-\cos(2\theta)}{2}\right)^{-r/2}\\\notag
&=2^{r/2}\left(1-\cos(2\theta)\right)^{-r/2}\\\notag
\end{align}
Hence we can apply the generalized binomial theorem to the expression in the preceeding line.
\end{proof}

\begin{proposition}
For $Re(s)<1$
\begin{align}\notag
\Sigma_s(z)&=\frac{1}{\pi}\Gamma(1-s)\sin(\frac{\pi s}{2})\sum_{j=1}^\infty (-1)^j \frac{1}{j^{1-s}}\left( \left(1-2z^2-2\sqrt{z^2(z^2-1)}\right)^j+\left(1-2z^2+2\sqrt{z^2(z^2-1)}\right)^j   \right)\\\notag
&=\frac{1}{\pi}\Gamma(1-s)\sin(\frac{\pi s}{2})\left({\rm Li}_{1-s}\left(-1+2z^2-2\sqrt{z^2(z^2-1)}\right) +{\rm Li}_{1-s}\left(-1+2z^2+2\sqrt{z^2(z^2-1)}\right)     \right)\\\notag
\end{align}
where Li is the polylog function.
\end{proposition}
\begin{proof}
We begin with Hurwitz's formula.
\begin{align}\notag
\zeta(s,z)&=\frac{2\Gamma(1-s)}{(2\pi)^{1-s}}\left(  \sin(\frac{\pi s}{2}) \sum_{j=1}^\infty\frac{1}{j^{1-s}}\cos(2\pi j z)+\cos(\frac{\pi s}{2})\sum_{j=1}^\infty \frac{1}{j^{1-s}}\sin(2\pi j z)     \right)\\\notag
\end{align}
From Proposition \ref{sigmazeta}, substitute for the $\zeta$ functions and use the trig lemma.
\end{proof}
\begin{proposition}
\begin{align}\notag
\Sigma_{2 n-1}(x)&=\frac{2}{ \pi^{2n-1}}\alpha(2 n-1)+\sum_{i=1}^{\infty}f_{2n-1,2n-1+2i}(x)\alpha(2 n-1+2i)\\\notag
\end{align}
where
$$\alpha(n)=\frac{2^{n}-1}{2^n}\zeta(n)$$
$$f_{2n-1,2n-1+2i }(x)=\frac{2^{2 i+1}}{\pi^{2n-1+2 i}}\binom{2n-2+2i}{2n-2}\arcsin(x)^{2i}$$
\end{proposition}
\begin{proof}
When we write our the series expansion of the $\Sigma_n(x)$ we find we can form the sum into an expression of the form
$$\Sigma_{2 n-1}(x)=\frac{2}{ \pi^{2n-1}}\alpha(2 n-1)+\sum_{i=n+1}^{\infty}f_{2n-1,2 i-1}(x)\alpha(2 i-1)$$
where
$$\alpha(n)=\frac{2^{n}-1}{2^n}\zeta(n)$$
For $n$ odd, $i\ge 1$,
$$f_{n,n+2i }(x)=c_{n,n+2 i}\arcsin(x)^{2i}$$
$$c_{n,n+2 i}= \frac{1}{\pi^{n+2 i}}k_{n,n+2 i}$$
\begin{align*}\notag
k_{n,n+2 i}&=2^{2 i+1}\binom{n+2i-1}{n-1}\\\notag
\end{align*}
For example, 
$$f_{3,5}(x)=\frac{48}{\pi ^5}\arcsin(x)^{2}$$
$$c_{n,n+2}=\frac{8}{\pi^{n+2}}{{n+1}\choose2} $$
$$f_{n,n+2 }(x)=\frac{8}{\pi^{n+2}}{{n+1}\choose2} \arcsin(x)^2$$
\end{proof}
We summarize results relating to $S(x)$:
\begin{proposition}
\begin{align}\notag
S(x)&=\lim_{n\to\infty}S_n(x)\\\notag
&=\frac{1}{8}\sum  _{i=1}^{\infty} \left(\frac{1}{\left(\pi  (2 i-1)-2 \arcsin(x)\right)^3}+\frac{1}{\left(\pi  (2 i-1)+2 \arcsin(x)\right)^3}\right)\\\notag
&=-\frac{1}{128\pi^3}(\psi_2(\frac{\pi-2\arcsin(x)}{2\pi})+\psi_2(\frac{\pi+2\arcsin(x)}{2\pi}))\\\notag
&=\frac{1}{64\pi^3}(\zeta(3,\frac{\pi-2\arcsin(x)}{2\pi})+\zeta(3,\frac{\pi+2\arcsin(x)}{2\pi}))\\\notag
&=\frac{1}{64}\left(\frac{x}{(1-x^2)^{3/2}}+\frac{2}{\pi^3}\zeta(3,\frac{\pi+2\arcsin(x)}{2\pi})\right)\\\notag
&=\frac{1}{8}\Big(  \frac{7}{4\pi^3}\zeta(3)+\sum_{i=1}^{\infty}\frac{2^{2i+3}-1}{4\pi^{2i+3}}(i+1)(2i+1)\arcsin(x)^{2i}\zeta(2i+3)\Big)\\\notag
&=\frac{1}{8}\Big(\frac{7}{4\pi^3}\zeta(3)+\frac{93}{2\pi^5}\zeta(5)x^2+\left(\frac{31}{2\pi^5}\zeta(5)+\frac{1905}{4\pi^7}\zeta(7)\right)x^4+\frac{1}{8}\left(\frac{124}{15\pi^5}\zeta(5)+\frac{635}{2\pi^7}\zeta(7)+\frac{3577}{\pi^9}\zeta(9)\right)x^6\\\notag
&+\left(\frac{186}{35\pi^5}\zeta(5)+\frac{889}{4\pi^7}\zeta(7)+\frac{3577}{\pi^9}\zeta(9)+\frac{92115}{4\pi^{11}}\zeta(11)\right)x^8\\\notag
&+\left(\frac{1984}{525\pi^5}\zeta(5)+\frac{10414}{63\pi^7}\zeta(7)+\frac{46501}{15\pi^9}\zeta(9)+\frac{30705}{\pi^11}\zeta(11)+\frac{270303}{2\pi^{13}}\zeta(13)\right)x^{10}\\\notag
&+\left(\frac{121666}{945\pi^7}\zeta(7)+\frac{71029}{27\pi^9}\zeta(9)+\frac{63457}{2\pi^{11}}\zeta(11)+\frac{450505}{2\pi^{13}}\zeta(13)+\frac{2981797}{2\pi^{15}}\zeta(15)\right)x^{12}+\cdots\\\notag
\end{align}
where $\psi$ is the polygamma function, $\psi_2$ is the third derivative of the logarithm of the gamma function, and the last 2 series expansions are valid over $(-1,1)$. $\zeta(s,a)$ ir the Hurwitz zeta function.
\end{proposition}
\begin{proof}
See above.
\end{proof}

An alternate form of the sum (not equal to $S(x)$, except at $x=\sqrt{2}/2$), is 
$$\tilde{S}(x)=\frac{1}{8}\sum  _{i=1}^{\infty} \left(\frac{1}{\left(i\pi  -2 \arcsin(x)\right)^3}\right)$$
This results in
$$\tilde{S}(x)=\frac{1 }{8 \pi ^3}\zeta(3)+\frac{3}{4 \pi ^4} \zeta(4) x+\frac{3 }{ \pi^5} \zeta(5) x^2 +\frac{1}{8 \pi ^6}\left( \pi ^2 \zeta(4) + 80 \zeta(6) \right)  x^3$$
$$+\frac{1}{\pi ^7} \left(\pi ^2 \zeta(5) + 30 \zeta(7) \right)x^4+\frac{1}{\pi ^8} \left(5 \pi^2 \zeta(6) +\frac{9 \pi ^4}{160}\zeta(4) +84 \zeta(8) \right) x^5+\cdots$$
 $$\zeta(3)=\frac{8\pi^3}{7}S(\frac{\sqrt{2}}{2})=\frac{8\pi^3}{7}A(\frac{\sqrt{2}}{2})$$

\begin{proposition}
There is a family of functions $\tilde{A}_n(x)$
$$\tilde{A}_n(x)=\frac{\pi^n}{2^n-1}\sum _{i=1}^{\infty} \left(\frac{1}{\left(i\pi  -2 \arcsin(x)\right)^n}\right)$$
such that $$\tilde{A}_n(\frac{\sqrt{2}}{2})=\zeta(n)$$
\end{proposition}
\begin{proof}
As above.
\end{proof}

\begin{theorem}
\begin{align}\notag\zeta(n)&=\frac{1}{(2^{n-1}-1)(2^n-1)}\sum_{i=1}^{\infty}\left(\frac{2^{2i+n}-1}{2^{4i}}\right) \binom{n+2i-1}{n-1}\zeta(n+2i)\notag\\
&=\frac{1}{2^{n-1}-1}\sum_{i=1}^{\infty}\frac{1}{2^{i+1}}\binom{n+i-1}{i}\zeta(n+i)\notag
\end{align}

\end{theorem}
\begin{proof}
A more direct proof is provided in the next section. 
For now, we focus on the functions $\Sigma_n(x)$, for odd ${n}$.
As mentioned in a previous theorem, when we write our the series expansion of the $\Sigma_n(x)$ we find we can form the sum into an expression of the form
$$\Sigma_{2 n-1}(x)=\frac{2}{ \pi^{2n-1}}\alpha(2 n-1)+\sum_{i=n+1}^{\infty}f_{2n-1,2 i-1}(x)\alpha(2 i-1)$$
where
$$\alpha(n)=\frac{2^{n}-1}{2^n}\zeta(n)$$
The sums of reciprocal powers that appear in the formula are only over odd integers, hence the above fraction of $\zeta$.
At $x=\frac{\sqrt{2}}{2}$
$$\zeta(3)=\frac{\pi^3}{7}\left(\frac{2}{ \pi^3}\alpha(3)+\sum_{i=3}^{\infty}f_{3,2 i-1}(x)\alpha(2 i-1)\right)$$
$$\implies \zeta(3)=\frac{4\pi^3}{21}\sum_{i=2}^{\infty}f_{3,2 i-1}(x)\alpha(2 i-1)$$
$$\zeta (3)=\frac{4 \pi ^3}{21}( f_{3,5}(x)\alpha (5)+ f_{3,7}(x)\alpha (7)+f_{3,9}(x)\alpha (9)+\cdots)$$
The constant term in the expansion of $\Sigma_{n}(x)$ will be $\frac{2}{\pi^n}\alpha(n)=\frac{2}{\pi^n}\frac{2^{n}-1}{2^n}\zeta(n)$. Hence at $x=\frac{\sqrt{2}}{2}$
$$\zeta(2 n-1)=\frac{2^{2n-2}\pi^{2n-1}}{(2^{2n-2}-1)(2^{2n-1}-1)}\sum_{i=n+1}^{\infty}f_{2n-1,2 i-1}(x)\alpha(2 i-1)$$
Here are the first few $f_{3,i}$:
$$f_{3,5}(x)=\frac{1}{\pi ^5}\left(48 x^2+16 x^4+\frac{128x^6}{15}+\frac{192 x^8}{35}+\frac{2048x^{10}}{525}+\frac{2048 x^{12}}{693}+\frac{16384 x^{14}}{7007}+\cdots\right)$$
$$f_{3,7}(x)=\frac{1}{\pi ^7}\left(480 x^4+320 x^6+224 x^8+\frac{10496 x^{10}}{63}+\frac{122624 x^{12}}{945}+\frac{120832 x^{14}}{1155}+\cdots\right)$$
$$f_{3,9}(x)=\frac{1}{\pi ^9}\left(3584 x^6+3584x^8+\frac{46592x^{10}}{15}+\frac{71168 x^{12}}{27}+\frac{303104 x^{14}}{135}+\cdots\right)$$
It turns out that
$$f_{3,5}(x)=\frac{48}{\pi ^5}\arcsin(x)^{2}$$
As previously mentioned, for $n$ odd, $i\ge 1$,
$$f_{n,n+2i }(x)=c_{n,n+2 i}\arcsin(x)^{2i}$$
$$c_{n,n+2}=\frac{8}{\pi^{n+2}}{{n+1}\choose2} $$
$$f_{n,n+2 }(x)=\frac{8}{\pi^{n+2}}{{n+1}\choose2} \arcsin(x)^2$$
Each constant $c_{n,n+2 i}$ is of the form$ \frac{1}{\pi^{n+2 i}}k_{n,n+2 i}$ where the $k$ are integers. 
\begin{align*}\notag
k_{n,n+2 i}&=2^{2 i+1}\binom{n+2i-1}{n-1}\\\notag
\end{align*}
So we may write

$$\alpha(n)=\frac{\pi^{n}}{2(2^{n-1}-1)}\sum_{i=1}^{\infty} \frac{1}{\pi^{n+2i}}k_{n,n+2i}\arcsin(x)^{2i}\alpha(n+2i)$$
when that sum is evaluated at $x=\frac{\sqrt{2}}{2}$.
\begin{align*}\notag
\alpha(n)&=\frac{1}{2(2^{n-1}-1)}\sum_{i=1}^{\infty} \frac{1}{\pi^{2i}}k_{n,n+2i}\left({\frac{\pi}{4}}\right)^{2i}\alpha(n+2i)\\\notag
&=\frac{1}{2(2^{n-1}-1)}\sum_{i=1}^{\infty} \frac{1}{2^{4 i}}k_{n,n+2i}\alpha(n+2i)\\\notag
&=\frac{1}{2(2^{n-1}-1)}\sum_{i=1}^{\infty} \frac{1}{2^{2 i-1}}\binom{n+2i-1}{n-1}\alpha(n+2i)\notag
\end{align*}
Applying the definition of $\alpha(n)$ gives the first formula for $\zeta(n)$. We go through the same steps with the $\tilde{A}_n(x)$ and arrive at the second formula.

\end{proof}

\section{ A More Direct Proof Of Zeta Series Results}
Once one sees the preceeding derivation, it becomes clear that we can produce $\zeta$ series relationships by expanding sums of the form $$\sum_{i}\frac{1}{((ai+b)-\frac{c}{4})^n}$$ for suitable constants $a,b,c$. For the first theorem below we have $a=4$, $b=2$, $c=4$.
\begin{lemma}
$$(1+x)^{-1}=\sum_{k=0}^{\infty}\binom{n+k-1}{k}(-1)^{k}x^{k}$$
\end{lemma}
\begin{proof}
(Binomial Theorem)
\end{proof}
Notice, the next theorem gives $\zeta$ of odd integer values in terms of $\zeta$ of odd integer values, and $\zeta$ of even values in terms of zeta of even values. 
{\begin{theorem}\begin{align}\notag\zeta(n)&=\frac{1}{(2^{n-1}-1)(2^n-1)}\sum_{i=1}^{\infty}\left(\frac{2^{2i+n}-1}{2^{4i}}\right) \binom{n+2i-1}{n-1}\zeta(n+2i)\\\notag
\end{align}
\end{theorem}
\begin{proof}

$$x$$
\begin{align}\notag
\zeta(n)&=\left(\sum_{i=0}^{\infty}\frac{1}{2^{ni}}\right)\left(\sum_{i=1}^{\infty}\frac{1}{(2i-1)^n}\right)\\\notag
&=\left(\frac{1}{1-\frac{1}{2^n}}\right)\left(\sum_{i=1}^{\infty}\frac{1}{(2i-1)^n}\right)\\\notag
&=\left(\frac{2^n}{2^n-1}\right)\left(\sum_{i=1}^{\infty}\frac{1}{(2i-1)^n}\right)\\\notag
&=\frac{2^n}{2^n-1}\sum_{i=1}^{\infty}\left(\frac{1}{((4i-2)-1)^n}+\frac{1}{((4i-2)+1)^n}\right)\\\notag
&=\frac{2^n}{2^n-1}\sum_{i=1}^{\infty}\left(\frac{1}{(2(2i-1)-1)^n}+\frac{1}{(2(2i-1)+1)^n}\right)\\\notag
&=\frac{1}{2^n-1}\sum_{i=1}^{\infty}   \frac{1}{(2(2i-1))^n}   \left (    \frac{1}{  (  1-\frac{1}{2(2i-1)}  )^n  }  +  \frac{1}{ (  1+\frac{1}{2(2i-1)}  )^n  }   \right )\\\notag
&=\frac{1}{2^n-1}\sum_{i=1}^{\infty}\frac{1}{(2(2i-1))^n}\sum_{k=0}^{\infty}\binom{n+k-1}{k}\left(\frac{1}{(2(2i-1))^k}+\frac{1}{(-1)^{k}(2(2i+1))^k}\right)\\\notag
&=\frac{1}{2^n-1}\sum_{i=1}^{\infty}\left(\frac{2}{(2(2i-1))^n}\sum_{k=0}^{\infty}\binom{n+2k-1}{2k}\frac{1}{(2(2i-1))^{2k}}\right)\\\notag
&=\frac{1}{2^n-1}\sum_{i=1}^{\infty}\left(2\sum_{k=0}^{\infty}\binom{n+2k-1}{2k}\frac{1}{(2(2i-1))^{n+2k}}\right)\\\notag
&=\frac{1}{2^n-1}\sum_{i=1}^{\infty}\left(2\sum_{k=0}^{\infty}\binom{n+2k-1}{2k}\frac{1}{2^{n+2k}}\frac{1}{(2i-1)^{n+2k}}\right)\\\notag
&=\frac{2}{2^n-1}\sum_{k=0}^{\infty}\binom{n+2k-1}{2k}\frac{2^{n+2k}-1}{2^{2n+4k}}\zeta(n+2k)\\\notag
&=\frac{1}{(2^n-1)2^{2n-1}}\sum_{k=0}^{\infty}\binom{n+2k-1}{2k}\frac{2^{n+2k}-1}{2^{4k}}\zeta(n+2k)\\\notag
&=\frac{1}{(2^n-1)2^{2n-1}}\left((2^n-1)\zeta(n)+\sum_{k=1}^{\infty}\frac{2^{n+2k}-1}{2^{4k}}\binom{n+2k-1}{2k}\zeta(n+2k)\right)\\\notag
&=\frac{1}{(2^n-1)(2^{2n-1}-1)}\left(\sum_{k=1}^{\infty}\frac{2^{n+2k}-1}{2^{4k}}\binom{n+2k-1}{n-1}\zeta(n+2k)\right)\\\notag
\end{align}

\end{proof}
\begin{corollary}\label{ocor}
\begin{align}
\zeta(3)&=\frac{1}{21}\sum_{j=2}^{\infty}\frac{2^{2j+1}-1}{2^{2j-3}\left(2^{2j}-1\right)}(2j)(2j-1)\eta(2j+1)\\\notag
\end{align}
\end{corollary}
Next, $a=2$, $b=0$, $c=4$
\begin{theorem}
\begin{align}\notag
\zeta(n)&=\frac{1}{2^{n-1}-1}\sum_{i=1}^{\infty}\frac{1}{2^{i+1}}\binom{n+i-1}{i}\zeta(n+i)\\\notag
\end{align}
\end{theorem}
\begin{proof}
\begin{align}\notag
\zeta(n)&=\sum_{i=1}^{\infty}\frac{1}{i^n}\\\notag
&=\left(\frac{2^n}{2^n-1}\right)\sum_{i=1}^{\infty}\frac{1}{(2i-1)^n}\\\notag
&=\left(\frac{2^n}{2^n-1}\right)\sum_{i=1}^{\infty}\frac{1}{(2i(1-\frac{1}{2i}))^n}\\\notag
&=\left(\frac{2^n}{2^n-1}\right)\sum_{i=1}^{\infty}\sum_{k=0}^{\infty}\binom{n+k-1}{k}\frac{1}{(2i)^{n+k}}\\\notag
&=\left(\frac{1}{2^n-1}\right)\sum_{k=0}^{\infty}\frac{1}{2^k}\binom{n+k-1}{k}\zeta(n+k)\\\notag
&=\left(\frac{1}{2^n-1}\right)\left(\zeta(n)+\sum_{k=1}^{\infty}\frac{1}{2^k}\binom{n+k-1}{k}\zeta(n+k)\right)\\\notag
&=\left(\frac{1}{2(2^{n-1}-1)}\right)\sum_{k=1}^{\infty}\frac{1}{2^k}\binom{n+k-1}{k}\zeta(n+k)\\\notag
&=\frac{1}{2^{n-1}-1}\sum_{k=1}^{\infty}\frac{1}{2^{k+1}}\binom{n+k-1}{k}\zeta(n+k)\\\notag
\end{align}
\end{proof}
\section{Infinite Linear Systems Of Equations}
Let
\begin{align}\notag
\bar A(i,j)&=\frac{2^{2(i+j)+1}-1}{2^{4j}(2^{2i}-1)(2^{2i+1}-1)} \binom{2(i+j)}{2i}\\\notag
A(i,j)&=\frac{1}{2^{2 j}(2^{2i-1}-1)} \binom{2(i+j)}{n-1}\\\notag
\end{align}
then
\begin{align}\notag
\alpha(n)&=\sum_{i=1}^{\infty} A((n-1)/2,i)\alpha(n+2i)\\\notag
\zeta(n)&=\sum_{i=1}^{\infty}\bar A((n-1)/2,i)\zeta(n+2i)\notag
\end{align}

We have an infinite system of linear equations $K A=0$, where $K$ an infinte upper triangular matrix and  $A$ is the list $(\alpha(3),\alpha(5),\alpha(7),...)$.
Here is a finite approximation
\begin{equation}\notag
\left(
\begin{array}{cccccc}
 1 & -A(1,1) & -A(1,2) & -A(1,3) & -A(1,4) & -A(1,5) \\
 0 & 1 & -A(2,1) & -A(2,2) & -A(2,3) & -A(2,4) \\
 0 & 0 & 1 & -A(3,1) & -A(3,2) & -A(3,3) \\
 0 & 0 & 0 & 1 & -A(4,1) & -A(4,2) \\
 0 & 0 & 0 & 0 & 1 & -A(5,1) \\
 0 & 0 & 0 & 0 & 0 & 1 \\
\end{array}
\right)
\left(
\begin{array}{c}
\alpha (3)\\
\alpha (5)\\
\alpha (7)\\
\alpha (9)\\
\alpha (11)\\
\alpha (13)\\
\end{array}
\right)=
\left(
\begin{array}{c}
 A(1,6)\alpha (15)\\
 A(2,5)\alpha (15)\\
 A(3,4)\alpha (15)\\
A(4,3)\alpha (15)\\
A(5,2)\alpha (15)\\
 A(6,1)\alpha (15)\\
\end{array}
\right)
\end{equation}
We have another set of equations where, in the above set we substitute $\bar{A}$ for $A$ and $\zeta$ for $\alpha$.
Since $\lim_{n\to\infty}\alpha(n)=1$, we replace $\alpha(15)$ with $1$.
So we have a sequence of equations of the form
$$V_n=K_n A_n$$ with solution vectors $X_n$. Call the above, with $\alpha(15)$ replaced with $1$,  $$V_5=K_5 A_5$$

 Hence
\begin{equation}\notag
\left(
\begin{array}{c}
\alpha (3)\\
\alpha (5)\\
\alpha (7)\\
\alpha (9)\\
\alpha (11)\\
\alpha (13)\\
\end{array}
\right)=
\left(
\begin{array}{cccccc}
  1 & -A(1,1) & -A(1,2) & -A(1,3) & -A(1,4) & -A(1,5) \\
 0 & 1 & -A(2,1) & -A(2,2) & -A(2,3) & -A(2,4) \\
 0 & 0 & 1 & -A(3,1) & -A(3,2) & -A(3,3) \\
 0 & 0 & 0 & 1 & -A(4,1) & -A(4,2) \\
 0 & 0 & 0 & 0 & 1 & -A(5,1) \\
 0 & 0 & 0 & 0 & 0 & 1 \\
\end{array}
\right)^{-1}
\left(
\begin{array}{c}
 A(1,6)\\
 A(2,5)\\
 A(3,4)\\
A(4,3)\\
A(5,2)\\
 A(6,1)\\
\end{array}
\right)
\end{equation}
\\

The $\alpha(3)$ approximation equals the first term of the solution vector $X_n$. The first term of the solution vector $X_n$ is also the last term in the first row of the matrix $K_{n+1}^{-1}$. 
\begin{equation}\notag
X_n(1)=K_{n+1}^{-1}(1,n+1)\notag
\end{equation}
In other words, the first row of $K^{-1}$ is the sequence of approximations to $\alpha(3)$.\\
Let
\begin{align}\notag
B(1,n)&=\sum _{i=2}^{n}A(1,i-1)\\\notag
B(2,n)&=\sum _{i=2}^{n-1} A(1,i-1) A(i,n-i)\\\notag
B(3,n)&=\sum _{i=2}^{n-2} \left(\sum _{j=1}^{n-i-1} A(1,i-1) A( i,j) A(i+j,n-i-j)\right)\\\notag
B(4,n)&=\sum _{i=2}^{n-3} \left(\sum _{j=1}^{n-i-1} \left(\sum _{k=1}^{n-i-j-1} A(1,i-1) A( i,j) A(i+j,k) A(i+j+k,n-i-j-k)\right)\right)\\\notag
\vdots\\\notag
\end{align}
And we define
\begin{align}\notag
U(1)&=1\\\notag
U(n)&=\sum_{i=1}^{n-1}B(i,n), \text{ for }n\ge2\\\notag
\end{align}
Then
\begin{equation}\notag
\left(
\begin{array}{cccccc}
U(1)&U(2)&U(3)&U(4)&U(5)&U(6)\\
\end{array}
\right)
\end{equation}
is the first row of $K_5^{-1}$.

The infinite sequence of $U(i)$ form the first row of $K^{-1}$
But this approach has a problem. Notice the second to last row of $K_n$ corresponds to $\alpha(2 n+1)=A(n,1)\alpha(2n+3)+A(n,2)\alpha(2n+5)$. It only uses the2 terms of the series expansion for $\alpha(2n+1)$ and that error compounds through the Gaussian elimination.\\
One way around this is to instead look at equations of the form $$V_n=K_n B_n$$ where, for example,

\begin{equation}\notag
B_5=
\left(
\begin{array}{c}
\alpha (15) A(1,6)+\alpha (17) A(1,7)+\alpha (19) A(1,8)+\alpha (21) A(1,9)+\alpha (23) A(1,10)\\
\alpha (15) A(2,5)+\alpha (17) A(2,6)+\alpha (19) A(2,7)+\alpha (21) A(2,8)+\alpha (23) A(2,9)\\
\alpha (15) A(3,4)+\alpha (17) A(3,5)+\alpha (19) A(3,6)+\alpha (21) A(3,7)+\alpha (23) A(3,8)\\
\alpha (15) A(4,3)+\alpha (17) A(4,4)+\alpha (19) A(4,5)+\alpha (21) A(4,6)+\alpha (23) A(4,7)\\
\alpha (15) A(5,2)+\alpha (17) A(5,3)+\alpha (19) A(5,4)+\alpha (21) A(5,5)+\alpha (23) A(5,6)\\
\alpha (15) A(6,1)+\alpha (17) A(6,2)+\alpha (19) A(6,3)+\alpha (21) A(6,4)+\alpha (23) A(6,5)\\
\end{array}
\right)
\end{equation}
$B_n$ would have $n$ rows with sums of $n$ terms on each row.

$$\alpha(3)=\lim_{n\to\infty} \sum_{i=1}^{n}\prod_{j=1}^{n}U_iA(i,n+j)\alpha(2(n+i+j+1))$$

but since $\lim_{n\to\infty}\alpha(n)=1$. We can replace this with
\begin{equation}\notag
B_5=
\left(
\begin{array}{c}
 A(1,6)+ A(1,7)+ A(1,8)+ A(1,9)+ A(1,10)\\
 A(2,5)+ A(2,6)+ A(2,7)+ A(2,8)+ A(2,9)\\
 A(3,4)+ A(3,5)+ A(3,6)+ A(3,7)+ A(3,8)\\
 A(4,3)+ A(4,4)+ A(4,5)+ A(4,6)+ A(4,7)\\
 A(5,2)+ A(5,3)+ A(5,4)+ A(5,5)+ A(5,6)\\
 A(6,1)+ A(6,2)+ A(6,3)+ A(6,4)+ A(6,5)\\
\end{array}
\right)
\end{equation}

$$\alpha(3)=\lim_{n\to\infty} \sum_{i=1}^{n}\prod_{j=1}^{n}U_iA(i,n+j)$$

\section{{\textbf{Partial List Of Results}}}

\begin{flalign}\notag
\zeta(3)&=\frac{2}{7}\int_{0}^{\pi/2}(x(\pi-x)) \csc\left( x\right)dx\\\notag
\zeta(3)&=\frac{1}{7}\pi^2\left(  1+\sum_{j=1}^{\infty}(-1)^j \left(  (2j+1)+2j(j+1)\ln\left(\frac{j}{j+1}\right)\right) \right)\\\notag
&=\frac{1}{7}\pi^2\left(  1+2\sum_{j=1}^{\infty} \left(        1+2j\left(   \ln\left(\left(\frac{2j-1}{2j+1}\right)\left(\frac{4j^2}{4j^2-1}\right)^{2j}\right) \right)       \right)        \right)\\\notag
&=\frac{2}{7}\pi^2\lim_{n\to\infty}\sum_{j=1}^{n}(-1)^{j}\frac{j(n!)^2}{(n-j)!(n+j)!}\ln\left(    \left(\frac{2j-1}{2j+1}\right)     \left(  \frac{4j^2}{4j^2-1}\right)^{j}\right)\\\notag
&=\frac{4}{7}\pi^2\lim_{n\to\infty}\sum_{j=2}^{n-1}(-1)^{j}\frac{(n!)^2j^2}{(n-j)!(n+j)!}\ln(j)\\\notag
\zeta(3)&=\frac{80\pi^2}{280+3\pi^2}\left(\frac{\pi^2}{144}-\frac{\ln(2)}{6}+\frac{1}{2}+\sum_{j=4}^{\infty}(-1)^{j}\frac{1}{ (j+1)(j+2)}\eta(j)\right)\\\notag
&=\frac{80\pi^2}{280-3\pi^2}\Bigl(\frac{\pi^2}{144}+\frac{\ln(2)}{6}+\sum_{j=4}^{\infty}\frac{1}{(j+1)(j+2)}\eta(j)\Bigr)\\\notag
&=\frac{35\pi^2}{7(35-3\pi^2)}\Bigl(-5-\frac{5\pi^2}{12}+\frac{40\ln(2)}{3}-\sum_{k=6}^{\infty}\Bigl((k-2)+\frac{4}{k}\sum_{j=1}^{k-2}(-1)^j \binom{k}{j+2}\eta(j)\Bigr)\Bigr)\\\notag
&=\lim_{n\to\infty}800\sum_{j=1}^{n-5}(-1)^{j+1}\frac{n!}{(j+5)(j+5)!(n-j)!}\eta(3+j)\\\notag
&=\frac{2}{7}\pi^2\sum_{j=1}^{\infty}\frac{1}{(j+1)(j+2)}\eta(j)\\\notag
&=\frac{2}{7}\pi^2\Bigl(\frac{1}{4}+\sum_{j=1}^{\infty}\left(\frac{1}{(2j+1)(2j+2)} \right)\eta(2j)\Bigr)\\\notag
&=\frac{2}{7}\pi^2\Bigl(\frac{3}{8}+\sum_{j=1}^{\infty}\left(\frac{2j+3}{2^{2j+1}(2j+1)(2j+2)}\right)\eta(2j)\Bigr)\\\notag
&=\frac{1}{7}\pi^2\Bigl(1-\sum_{j=1}^{\infty}\frac{1}{(2^{2j-1}-1)(j+1)(2j+1)}\eta(2j)\Bigr)\\\notag
&=\frac{1}{7}\pi^2\Bigl(2\ln\left(2\right)+\sum_{j=1}^{\infty}\frac{1}{j(2j-1)}\left(\eta(2j-2)-1\right)\Bigr)\\\notag
&=\frac{1}{7}\pi^2\Bigl(\ln(\frac{\pi}{2})-\frac{1}{2}+\sum_{j=1}^{\infty}\frac{1}{j(j+1)}\eta(2j)\Bigr)\\\notag
&=\frac{80}{3}\Bigl(\frac{1}{4}-\frac{\ln(2)}{6}-\sum_{j=2}^{\infty}\frac{1}{(2j+2)(2j+3)}\eta(2j+1)\Bigr)\\\notag
&=\frac{1}{21}\sum_{j=2}^{\infty}\frac{2^{2j+1}-1}{2^{2j-3}\left(2^{2j}-1\right)}(2j)(2j-1)\eta(2j+1)\\\notag
&=\frac{2}{7}\Bigl(\frac{\pi^3}{8}+\sum_{j=1}^{\infty}(-1)^j\frac{\pi^{2j+1}}{(2j+1)2^{2j+1}}\left(\frac{\pi^2}{4(2j)!}E_{2j}+\frac{1}{(2j-2)!}E_{2j-2}\right)\Bigr)\\\notag
\end{flalign}
\begin{flalign}\notag
&=\frac{1}{3}\sum_{j=1}^{\infty}\frac{1}{2^{j+2}}(j+1)(j+2)\zeta(3+j)\\\notag
&=\frac{1}{3}\sum_{j=1}^{\infty}\frac{1}{2^{j+2}-1}(j+1)(j+2)\eta(3+j)\\\notag
&=\frac{2}{3}\Bigl(4\pi-8\ln(2)-\frac{\pi^2}{3}-\sum_{j=1}^{\infty}\frac{1}{2^{j-1}}\eta(3+j)\Bigr)\\\notag
\end{flalign}
\begin{align}
\zeta(k)&=\frac{1}{(2^{k-1}-1)(2^k-1)}\sum_{j=1}^{\infty}\left(\frac{2^{2j+k}-1}{2^{4j}}\right) \binom{k+2j-1}{k-1}\zeta(k+2j)\\\notag
&=\frac{1}{2^{k-1}-1}\sum_{j=1}^{\infty}\frac{1}{2^{j+1}}\binom{k+j-1}{j}\zeta(k+j)\\\notag
&=\frac{1}{k-1}\Bigl(k-\sum_{j=1}^{\infty}\binom{j+k-1}{j+1}(\zeta(k+j)-1)\Bigr)\\\notag
&=\lim_{n\to\infty}\sum_{j=1}^{n}(-1)^{j+1}\binom{n}{j}\zeta(k+j h),{\rm\   }h>0\\\notag
&=\lim_{n\to\infty}\frac{\prod_{j\ {\rm odd}}\zeta(k+jh)^{\binom{n}{j}}}{\prod_{j\ {\rm even}}\zeta(k+jh)^{\binom{n}{j}}}\\\notag
&=\lim_{n\to\infty}\sum_{j=1}^{n}(-1)^{j-1}\binom{n}{j}\zeta(\ln(\exp(k)+j h))\\\notag
&=\lim_{n\to\infty}\sum_{j=1}^{n}(-1)^{j+1}\binom{n}{j} \zeta(k+h\sum_{k=1}^{j}\frac{1}{2^k}))\\\notag
\eta(k)&=\frac{1}{2}\sum_{j=1}^{\infty}\frac{1}{2^{k+j-1}-1}\binom{k+j-1}{j}\eta(k+j)\\\notag
&=\lim_{n\to\infty}\sum_{j=1}^{n}(-1)^{j+1}\binom{n}{j}\eta(k+j h),{\rm\   }h>0\\\notag
B_{2j}&=\frac{(2j)!}{(2^{2j-1}-1)\pi^{2j}}\lim_{n\to\infty}\sum_{i=1}^{n}(-1)^{i+j}\frac{(n!)^2}{i^{2j}(n-i)!(n+i)!}\\\notag
\pi&=\sum_{j=1}^{\infty}a_j\zeta(3)^j\thickspace\thickspace\textrm{           (see cor. {\ref{sumpi}})  }
\end{align}
\begin{flalign}\notag
&\sum_{j=1}^{n}(-1)^{j-1}\frac{(n!)^2}{(n-j)!(n+j)!}=\frac{1}{2}\\\notag
&\sum_{j=1}^{\infty}\frac{1}{2^j}\eta(j)=\frac{\pi}{4}\\\notag
&\sum_{j=1}^{\infty}\frac{1}{(2j)(2j+1)} \eta(2j-1)=\frac{1}{4}\\\notag
&\sum_{j=1}^{\infty}\frac{1}{(2j)(2j+1)} \eta(2j)=\frac{1}{2}\left(1-\ln\left(\frac{\pi}{2}\right)\right)\\\notag
\end{flalign}
\begin{section}{{\textbf{Bibliography}}}
\frenchspacing
\begin{itemize}
\item{[1]} Johan W{\"a}stlund,``Summing inverse squares by euclidean geometry'', \\(PDF) http://www.math.chalmers.se/~wastlund/Cosmic.pdf
\item{[2]} 3BlueBrown (YouTube channel video), \"Why is pi here? And why is it squared? A geometric answer to the Basel problem".
\item{[3]} Wolfram Mathworld,``Apery's Constant", https://mathworld.wolfram.com/AperysConstant.html
\item{[4]} C. Nash and D. O'Connor, ``Determinants of Laplacians, the Ray-Singer torsion on lens
spaces and the Riemann zeta function", J. Math. Phys. 36(1995), 1462-1505.
\item{[5]} Maarten J. Kronenburg,``The Binomial Coefficient for Negative Arguments", arXiv:1105.3689 [math.CO],18 May 2011
\item{[6]} Michael Penn, ``Integral Of $(\ln(\cos(x))^3$'', YouTube video.
\item{[7]} Arkadiusz Wesolowski,``OEIS", Oct.16, 2013
\item{[8]} Peter Paule and Markus Schorn, `` A Mathematica Version of Zeilberger's Algorithm for Proving Binomial Coefficient Identities", J. Symbolic Computation (1994) 11, 1-000
\item{[9]} Gosper, Jr., Ralph William "Bill" (January 1978) [1977-09-26]. "Decision procedure for indefinite hypergeometric summation" (PDF). Proceedings of the National Academy of Sciences of the United States of America.
\item{[10]} Petkovšek, Marko; Wilf, Herbert; Zeilberger, Doron (1996). A = B. Home Page for the Book "A=B". A K Peters Ltd. ISBN 1-56881-063-6. Archived 
\item{[11]} Morse, P. M. and Feshbach, H. Methods of Theoretical Physics, Part I. New York: McGraw-Hill, pp. 411-413, 1953.
\item{[12]} "Finite-difference calculus", Encyclopedia of Mathematics, EMS Press, 2001 [1994]
\item{[13]} Charles Jordan “On Stirling’s numbers”, Tohoku Math. Journal, Vol. 37, 1933
\item{[14]} Henry W. Gould,  "Tables of Combinatorial Identities", edited by Jocelyn Quaintance, Vol. 8
\item{[15]} Philippe Flajolet and Robert Sedgewick, "Mellin transforms and asymptotics: Finite differences and Rice's integrals", Theoretical Computer Science 144 (1995) pp 101–124.
\item{[16]} Murray R. Spiegel, ``Schaum's Outline Of Theorey And Problems Of Calculus Of Finite Differences And Difference Equations'',McGraw-Hill, 1971
\item{[17]} R. Kl\'en, M. Visuri and M. Vuorinen, ``On Jordan type inequalities for hyperbolic functions'',J. Inequal. Appl., Vol 2010, Article no. 362548, 2010
\item{[18]} Samuel Greitzer “Many cheerful Facts” Arbelos 4 (1986), no. 5, 14-17
\item{[20]} L. Fairbanks, ``Powers of Cosine and Sine, arXiv:2308.04437  [math.NT]
\end{itemize}
\end{section}
\end{document}